\documentclass[12pt,letterpaper]{amsart}
\usepackage{fullpage}
\usepackage{hyperref}
\usepackage{amssymb}
\usepackage{enumitem}
\usepackage{amsmath}
\usepackage{color}
\usepackage{aliascnt}
\usepackage{amsthm}
\usepackage[all,cmtip]{xy} 
\usepackage[light]{anttor}
\usepackage[T1]{fontenc}
\usepackage{amssymb}
\usepackage{mathtools}
\usepackage{color}
\usepackage{float}
\usepackage{tikz}
\usepackage{tikz-cd}
\usetikzlibrary{arrows.meta}
\usetikzlibrary{calc}
\usepackage[utf8]{inputenc}
\usepackage[english]{babel}
\usepackage{amsmath,calligra,mathrsfs}
\usepackage{faktor}
\usepackage{csquotes}
\usepackage{graphicx} 

\usepackage[backend=bibtex,sorting=anyt]{biblatex}
\renewbibmacro{in:}{}
\bibliography{References}

\newcommand{\Torp}{\mathbf{Tor}^{\ell^p_c}}
\newcommand{\Extp}{\mathbf{Ext}_{\ell^p_c}}
\newcommand{\im}{\textnormal{im }}

\newcommand{\id}{\textnormal{id}}
\newcommand{\colim}{\textnormal{colim}}
\newcommand{\cat}[1]{\mathbf{#1}}

\newcommand{\shtensor}{\otimes_{\cat{sh}}}
\newcommand{\Ptensor}{\otimes_{\cat{P}}}
\newcommand{\ptensor}{\otimes_{\ell^p_c}}
\newcommand{\ktensor}{\otimes_{\cat{k}}}
\newcommand{\cpnorm}{\ell^p_c}
\newcommand{\op}{\textnormal{op}}

\newcommand{\Hom}{\textnormal{Hom}}

\newcommand{\scHom}{\mathscr{H}\text{\kern -5pt {\calligra\large om}}\,}

\DeclareMathOperator*{\argmin}{arg\,min}

\newtheorem{theorem}{Theorem}[section]

\newaliascnt{corollary}{theorem}
\newtheorem{corollary}[corollary]{Corollary}
\aliascntresetthe{corollary}

\newaliascnt{proposition}{theorem}
\newtheorem{proposition}[proposition]{Proposition}
\aliascntresetthe{proposition}

\newaliascnt{claim}{theorem}

\aliascntresetthe{claim}

\theoremstyle{definition}
\newaliascnt{definition}{theorem}
\newtheorem{definition}[definition]{Definition}
\aliascntresetthe{definition}

\theoremstyle{remark}
\newaliascnt{remark}{theorem}
\newtheorem{remark}[remark]{Remark}
\aliascntresetthe{remark}

\theoremstyle{plain}
\newaliascnt{lemma}{theorem}
\newtheorem{lemma}[lemma]{Lemma}
\aliascntresetthe{lemma}

\theoremstyle{definition}
\newaliascnt{example}{theorem}
\newtheorem{example}[example]{Example}
\aliascntresetthe{example}

\usepackage[capitalize,noabbrev]{cleveref}
\Crefname{theorem}{Theorem}{Theorems}

\Crefname{corollary}{Corollary}{Corollaries}

\Crefname{proposition}{Proposition}{Propositions}

\Crefname{claim}{Claim}{Claims}

\Crefname{definition}{Definition}{Definitions}

\Crefname{remark}{Remark}{Remarks}

\Crefname{lemma}{Lemma}{Lemmas}

\Crefname{example}{Example}{Examples}

\newcommand{\od}{\stackrel{\mbox {\tiny {def}}}{=}}

\title{A Continuum of K\"unneth Theorems for Persistence Modules}
\author{Nikola Mili\'cevi\'c}

\begin{document}

\begin{abstract}
 We develop new aspects of the homological algebra theory for persistence modules, in both the one-parameter and multi-parameter settings. For a poset $P$ and an order preserving map $\varphi:P\times P\to P$, we introduce a novel tensor product of persistence modules indexed by $P$, $\otimes_{\varphi}$. We prove that each $\otimes_{\varphi}$ has a right adjoint, $\scHom^{\varphi}$, the internal hom of persistence modules that also depends on $\varphi$. We prove that every $\otimes_{\varphi}$ yields a K\"unneth short exact sequence of chain complexes of persistence modules. Dually, the $\scHom^{\varphi}$ also has an associated K\"unneth short exact sequence in cohomology. As special cases both of these short exact sequences yield Universal Coefficient Theorems. We show how to apply these to chain complexes of persistence modules arising from filtered CW complexes.
 
For the special case of $P=\mathbb{R}_+$, the $p$-quasinorms for each $p\in (0,\infty]$ yield a distinct $\ptensor$ and its adjoint $\scHom^{\cpnorm}$. We compute their derived functors, $\Torp$ and $\Extp$ explicitly for interval modules. We show that the Universal Coefficient Theorem developed can be used to compute persistent Borel-Moore homology of a filtration of non-compact spaces. Finally, we show that for every $p\in [1,\infty]$ the associated K\"unneth short exact sequence can be used to significantly speed up and approximate persistent homology computations in a product metric space $(X\times Y,d^p)$ with the distance $d^p((x,y),(x',y'))=||d_X(x,x'),d_Y(y,y')||_p$.
\end{abstract}

\maketitle

\tableofcontents

\section{Introduction}

Given a parametrized filtration of topological spaces, the goal of persistent homology is understanding how homology groups of the spaces vary with the choice of parameters. This information is encoded in an algebraic object called a persistence module. If the parameters are one-dimensional such as $\mathbb{R}$, this is called one-parameter persistence, while parameters with more than one dimension such as $\mathbb{R}^n$ yield multi-parameter persistence. Persistence modules have a particularly rich algebraic structure and have been considered through many different perspectives. They have been studied as sheaves \cite{Kashiwara2018,Fernandes2025,berkouk2024,Berkouk2021,Berkouk2022,Bubenik2021,curry2014sheaves,Miller2023,harsu2024}, graded modules \cite{Bubenik2021,Carlsson2009,Lesnick2015,Zomorodian2005,Miller2025,Miller2025,Harrington2019,Geist2023}, quiver representations \cite{Takeuchi2021,Escolar2016,Carlsson2010,Oudot2016} and as functors \cite{Bubenik2021,Bubenik2014,Bubenik2015}. 

Much of the effort in exploiting this rich algebraic foundation went into defining and computing distances of persistence modules \cite{Bubenik2015,Chazal2016,medinamardones2025}, and defining new invariants of these modules that can be used as features for learning in data science \cite{angel2025,Botnan2024,Dey2024,Blanchette2022,Lupo2022,Fernandes2025,JMLR:v21:19-054,Bubenik2015Landscapes}. 
On the other hand, other works attempted to utilize the algebra of persistence modules to develop analogues of classical results from algebraic topology in the persistent setting. Examples of this include persistent spectral sequences \cite{Torras-Casas2023,goldfarb2019,Ren2021-ox,romero2014}, Universal Coefficient Theorems \cite{Boissonnat2019,Bubenik2021,Obayashi2023}, K\"unneth theorems  \cite{Bubenik2021,gakhar2019,Polterovich2017,Carlsson2020} and cup products \cite{Contessoto2022,Memoli2024}.

The focus of this manuscript is the further development of K\"unneth theorems for persistence modules. We use ideas from sheaf theory to further develop the homological algebra of persistence modules. These new techniques are used to compute persistent homology of product spaces.

\subsection*{Our contributions}

A filtration of a CW complex $X$ is a function $f:X\to \mathbb{R_+}$, where $\mathbb{R}_+\od\{x\in \mathbb{R}\,|\, x\ge 0\}$, such that $f(\partial \sigma) \le f(\sigma)$ for all cells $\sigma$ of $X$. The sublevel sets of the filtration, $X^f_a\od f^{-1}[0,a]$, for $a\in \mathbb{R}_+$, produce a filtration of $X$ by subcomplexes. Let $H_n(X^f_a,\cat{k})$ be the $n$-th homology group of $X_a$ with coefficients in a field $\cat{k}$. For each $a\le b$ in $\mathbb{R}_+$, $n\in \mathbb{N}$, the inclusion maps $X^f_a\hookrightarrow X^f_b$ induce $\cat{k}$-linear maps 
\[H_n(X^f_a;\cat{k})\to H_n(X^f_b;\cat{k}).\] 
These data are assembled into a functor, called a persistence module, $\mathcal{H}^f_n(X):\mathbb{R}_+\to \cat{Vect}_k$ that assigns to each $a\in \mathbb{R}_+$ the $\cat{k}$-vector space $H_n(X^f_a;\cat{k})$ and to each pair $a\le b$ the $\cat{k}$-linear map $H_n(X^f_a;\cat{k})\to H_n(X^f_b;\cat{k})$ induced by inclusion.

The classical K\"unneth theorem tells us that there is a short exact sequence of abelian groups for two CW complexes $X$ and $Y$: 
\[0\to \bigoplus_{i+j=n}H_i(X)\otimes H_j(Y)\to H_n(X\times Y)\to \bigoplus_{i+j=n-1}\mathbf{Tor}(H_i(X),H_j(Y))\to 0.\]

Let $f:X\to \mathbb{R}_{+}$ and $g:Y\to \mathbb{R}_+$ be two filtrations of CW complexes. Given these classical foundations, we can now formulate our central question:

\textbf{Question}: Is there a filtration $F:X\times Y\to \mathbb{R}_+$ such that there exists a short exact sequence of persistence modules:

\[0\to \bigoplus_{i+j=n}\mathcal{H}^f_i(X)\otimes \mathcal{H}^g_j(Y)\to \mathcal{H}^F_n(X\times Y)\to \bigoplus_{i+j=n-1}\mathbf{Tor}(\mathcal{H}^f_i(X),\mathcal{H}^g_j(Y))\to 0?\]

This question was answered in \cite{Bubenik2021} for two very special examples; $F\in \{f+g,\max(f,g)\}$. In order to even attempt answering this question for more general choices of $F$, one needs to define a tensor product and its derived functor, $\cat{Tor}$, for persistence modules. Furthermore, the choice of a filtration $F$ is expected to influence the definition of $\otimes$ and hence $\cat{Tor}$, as was first observed in \cite{Bubenik2021}, where the two considered choices for $F$, $F\in \{f+g,\max(f,g)\}$, had distinct K\"unneth short exact sequences, each using a different tensor product.

In this work we extend this result to much greater generality. Namely, suppose the function $\varphi:\mathbb{R}_{+}\times \mathbb{R}_{+}\to \mathbb{R}_{+}$ is order-preserving and that $f:X\to \mathbb{R}_{+}$, $g:Y\to\mathbb{R}_+$  are filtrations of CW complexes. We prove that for any such $\varphi$, there is a K\"unneth theorem for the filtration 
\[F\od\varphi\circ (f\times g):X\times Y\to \mathbb{R}_{+},\] 
of the CW complex $X\times Y$. This result is encapsulated in \cref{theorem:kunneth_filtered_cw}.

In order to prove this result we significantly expanded the available homological algebra toolkit for the category of persistence modules. We showcase how to define new interesting functors, which are analogous to the six Grothendieck operations for sheaves. For generalized persistence modules, i.e., functors $M,N:\cat{P}\to \cat{Vect_k}$ where $\cat{P}$ is a poset category, and an order-preserving morphism $\varphi:P\times P\to P$ we define the tensor product $M\otimes_{\varphi}N$ and its adjoint $\scHom^{\varphi}(M,N)$. Furthermore, we compute their derived functors, $\cat{Tor}^{\varphi}$ and $\cat{Ext}_{\varphi}$ explicitly for interval modules, which is crucial for applying the K\"unneth short exact sequences for persistence modules in data science. We prove the existence of K\"unneth short exact sequences of persistence modules 
\[\bigoplus_{i+j=n}H_i(K)\otimes_{\varphi} H_j(L)\to H_n(K\otimes_{\varphi} L)\to\bigoplus_{i+j=n-1} \cat{Tor}^{\varphi}_1(H_{i}(K),H_j(L))\quad (\text{\cref{theorem:kunneth_1}})\]
and 
\[\prod_{i+j=n-1}\cat{Ext}^1_{\varphi}(H_i(K),H_{j}(L))\to H^n(\scHom^{\varphi}(K, L))\to\prod_{i+j=n} \scHom^{\varphi}(H_i(K),H_{j}(L))\quad (\text{\cref{theorem:kunneth_2}})\]
for complexes of persistence modules $K$ and $L$. By making the complex $L$ be a single persistence module concentrated in degree $0$ we get Universal Coefficient Theorems for persistence modules (\cref{theorem:UCT_1,theorem:UCT_2}).

For the special case $P=\mathbb{R}_+$ and letting $p\in (0,\infty]$ and setting 
\[\varphi(a,b)\od (a^p+b^p)^{\frac{1}{p}},\] 
we immediately get a continuum of distinct short exact K\"unneth sequences, each corresponding to a different and novel tensor product in the abelian category of persistence modules. This allows us to use the K\"unneth short exact sequences we developed to compute homology of filtered products of CW complexes in applications.

We show that the Universal Coefficient Theorem for persistence modules (\cref{theorem:UCT_2}) can be used to calculate the Borel-Moore persistent homology of a filtration of non-compact spaces (\cref{theorem:borel_moore}). To our knowledge, this is the first such computation in the persistent setting.

We apply these results to approximate Vietoris-Rips persistent homology computations in a product of metric spaces. Namely if $(X,d_X)$ and $(Y,d_Y)$ are two finite metric spaces we are interested in computing the Vietoris-Rips persistent homology of $(X\times Y,d^p)$ where 
\[d^p((x,y),(x',y'))=||d_X(x,x'),d_Y(y,y')||_p,\]
 for $p\in [1,\infty]$ (\cref{section:approximating}). Computing persistence modules arising from filtrations of simplicial complexes is computationally intractable as the dimension and number of points of the simplicial complex increase. This work has significant implications for faster computations and approximations of persistent homology of product spaces as we can compute persistent homology of a smaller sample of points and in lower dimensions in the two base spaces. A natural example of this is the torus where so far mostly the sup and sum metrics have been used. However there is an interest in computing persistent homology in the $\ell^2$ metric as well \cite{Adams2025}. We demonstrate the potential of this approach with several examples.

\subsection*{Related work}
Our work builds upon and extends several important prior results.
Tensor products and K\"unneth theorems for persistence modules were considered in \cite{Polterovich2017,Carlsson2020} for the case $p=1$, and in \cite{Bubenik2021,gakhar2019} for the case $p\in \{1,\infty\}$. The use of these theorems has found applications in calculating higher dimensional persistent homology groups of tori, in both the sum and the supremum metrics.  Homological algebra techniques for persistence modules have been developed also in \cite{harsu2024,Kashiwara2018,Miller2025,Berkouk2021,Fernandes2025}.

\section{Preliminaries}
We assume the reader is familiar with category theory, particularly limit and colimit constructions, abelian categories, and derived functors. For background on these topics, see for example \cite[Chapter 1]{Kashiwara1990}. 

Given a category $\cat{C}$ and objects $A$ and $B$ in $\cat{C}$ the class of morphisms in $\cat{C}$ between $A$ and $B$ will be denoted by $\Hom_{\cat{C}}(A,B)$. We now list a couple of categories that will be used throughout this work. 

\begin{enumerate}[left=0pt]
        \item Let $P=(P,\le)$ be a poset. The category $\cat{P}$ is the category whose objects are the $p\in P$ and the morphisms are 
        \[\Hom_{\cat{P}}(p,p')\od\begin{cases}
            *,& p\le p'\\
            \varnothing,& \textnormal{otherwise}
        \end{cases}.\]
        \item Let $\cat{k}$ be a field. Then we have the category $\cat{Vect_k}$ whose objects are $\cat{k}$-vector spaces and the morphism $\Hom_{\cat{k}}(V,W)$ are the $\cat{k}$-linear maps between $V$ and $W$. Note that we write $\Hom_{\cat{k}}(V,W)$ instead of $\Hom_{\cat{Vect_k}}(V,W)$ for simplicity.
        \item Let $\cat{C}$ and $\cat{D}$ be categories. The functor category $\cat{D^C}$ is the category whose objects are functors $F:\cat{C}\to \cat{D}$ and the morphism $\Hom_{\cat{D^C}}(F,G)$ are the natural transformations between functors. 
\end{enumerate}

An \emph{up-set} in a poset $(P,\le )$ is a subset $U\subset P$ such that if $x\in U$ and $x\le y$ then $y\in U$. For $a\in P$ denote by $U_a$ the \emph{principal up-set} at $a$, i.e., $U_a\od\{x\in P\,|\, a\le x\}$. A \emph{down-set} in a poset $(P,\le)$ is a subset $D\subset P$ such that if $x\in D$ and $y\le x$ then $y\in D$. For 
$a\in P$ denote by  $D_a$ the \emph{principal down-set} at $a$, i.e., $D_a\od\{x\in P\,|\, x\le a\}$. 

A large focus of our manuscript will be the category $\cat{R^n_+}$ corresponding to the poset $(\mathbb{R}^n_{+},\le)$, where $\le$ denotes the product partial order on $\mathbb{R}^n_{+}\od \{(x_1,\dots, x_n)\in \mathbb{R}^n\,|\,\forall 1\le i\le n, x_i\ge 0\}$. That is, $(x_1,\dots, x_n)\le (y_1,\dots, y_n)$ if and only if $x_i\le y_i$ for all $i$.

\subsection{Direct and inverse image of functors on posets}
Here we recall different functorial constructions associated to a morphism of posets $f:P\to Q$. References are made to \cite[Chapter 5]{curry2014sheaves}. Observe that a poset morphism $f:P\to Q$ is also a functor $f:\cat{P}\to \cat{Q}$. Assume that $\cat{C}$ is a complete and cocomplete category, e.g. the category $\cat{Vect_k}$.

\begin{definition}
\label{def:inverse_image}
    Given a poset morphism $f:P\to Q$ and a functor $G:\cat{Q}\to \cat{C}$, the \emph{inverse image} of $G$ is the functor $f^{-1}G:\cat{P}\to \cat{C}$ defined by  $(f^{-1}G)_p\od G_{f(p)}$, and $(f^{-1}G)_{p\le p'}=G_{f(p)\le f(p')}$.
\end{definition}

We thus have a functor $f^{-1}:\cat{C^Q}\to \cat{C^P}$. It turns out that the inverse image functor has both a left and a right adjoint which we describe below. 

\begin{definition}
\label{def:direct_image_1}
    Given a poset morphism $f:P\to Q$ and a functor $F:\cat{P}\to \cat{C}$, the \emph{direct image functor} of $F$ is the functor $f_*F$ defined by
    \[(f_*F)_q\od \underset{q\le f(p)}{\lim}F_p,\]
    and  $(f_*F)_{q\le q'}$ is given by the induced canonical map. 
\end{definition}

It follows from the definition that  $f_*:\cat{C^P}\to \cat{C^Q}$ is a functor. Furthermore,  $f_*$ is the right adjoint of $f^{-1}$.

\begin{theorem}{\cite[Theorem 5.3.1]{curry2014sheaves}}
\label{theorem:direct_image_adjunction_1}
    The functor $f^{-1}$ is the left adjoint of $f_*$. That is, there are natural bijections
    \[\Hom_{\cat{C^P}}(f^{-1}G,F)\approx\Hom_{\cat{C^Q}}(G,f_*F),\]
    for any functors $F:\cat{P}\to \cat{C}$ and $G:\cat{Q}\to \cat{C}$.
\end{theorem}

We also have the following construction.

\begin{definition}
\label{def:direct_image_2}
    Given a poset morphism $f:P\to Q$ and a functor $F:\cat{P}\to \cat{C}$, the \emph{direct image with open support} of $F$ is the functor $f_{\dagger}F:\cat{Q}\to \cat{C}$ defined by 
    \[(f_{\dagger}F)_q\od \underset{f(p)\le q}{\colim}F_p,\]
    and $(f_{\dagger}F)_{q\le q'}$ is given by the induced canonical map. 
\end{definition}

This defines a functor $f_{\dagger}:\cat{C^P}\to \cat{C^Q}$, which is the left adjoint of $f^{-1}$.

\begin{theorem}{\cite[Theorem 5.3.1]{curry2014sheaves}}
\label{theorem:direct_image_adjunction_2}
    The functor $f^{-1}$ is the right adjoint of $f_{\dagger}$. That is, there are natural bijections
    \[\Hom_{\cat{C^Q}}(f_{\dagger}F,G)\approx\Hom_{\cat{C^P}}(F,f^{-1}G),\]
    for any functors $F:\cat{P}\to \cat{C}$ and $G:\cat{Q}\to \cat{C}$.
\end{theorem}

\begin{lemma}
\label{lemma:colimits_and_limits_over_diagrams_with_0_and_1}
    Let $F:\cat{P}\to \cat{C}$ be a functor from $\cat{P}$ to a cocomplete category $\cat{C}$. If $P$ has a greatest element $p$, $q\le p$ for all $q\in P$, then there is a canonical isomorphism $\underset{q\in P}{\colim\,}F_q\approx F_p$. If $\cat{C}$ is complete and $P$ has a lowest element $p$, $p\le q$ for all $q\in P$, then there is a canonical isomorphism $
    \underset{q\in P}{\lim\,} F_q\approx F_p$.
\end{lemma}

\begin{proof}
	Suppose that $\cat{C}$ is cocomplete. Then $\underset{q\in P}{\colim\,}F_q$ is well defined. Furthermore, since $p$ is the greatest element, $\cat{P}$ has a terminal element. Thus the result follows. The other case is analogous. 
\end{proof}

Every poset $(P,\le)$ can be considered as a topological space in the following way. The \emph{Alexandrov} topology on $P$ is the topology whose open sets are the up-sets in $P$. If we reverse the order on $P$ we get a poset $(P,\le^{\op})$ which we will also denote by $P^{\op}$. Note that the down-sets (resp. up-sets) in $P$ are the up-sets (resp. down-sets) in $P^{\op}$. 

It turns out every functor $F:\cat{P}\to \cat{C}$ is equivalent to a sheaf on $P$ with the Alexandrov topology, when $\cat{C}$ is complete. The correspondence is given by setting the sheaf value $F(U)\od \underset{p\in U}{\lim\,} F_p$, for every up-set $U$ in $P$. Similarly, every functor $\cat{P}\to \cat{C}$ is equivalent to a cosheaf on $P^{\text{op}}$, when $\cat{C}$ is cocomplete. The correspondence is given by setting the cosheaf value $F(D)\od \underset{p\in D}{\colim\,} F_p$, for every down-set $D$ in $P$. 

\begin{theorem}{\cite[Theorem 4.2.10]{curry2014sheaves}}
\label{theorem:functors_and_sheaves}
    Let $(P,\le)$ be a poset and let $\cat{C}$ be a complete category. Then, there is an isomorphism of categories between the functor category $\cat{C^{P}}$ and the category $\cat{Sh}(P,\cat{C})$ of sheaves on $P$ valued in $\cat{C}$ with the Alexandrov topology. Let $\cat{D}$ a cocomplete category. Then, there is an isomorphism of categories between the functor category $\cat{P^D}$ and the category $\cat{CoSh}(P^{\op},\cat{C})$ of cosheaves on $P^{\op}$ with the Alexandrov topology.
\end{theorem}

\section{Persistence modules}
Here, we recall the necessary background on persistence modules. Let $(P,\le)$ be a poset.

\subsection{Persistence modules}
\label{subsection:persistence_modules_category}
Persistence modules have many equivalent formulations; e.g. they can be regarded as graded modules, functors or sheaves \cite{Bubenik2021}. In this work we view them as functors $M:\cat{P}\to \cat{Vect_k}$. The main examples we will be working with are $P=\mathbb{R}^n_+$ (and $n=1$ when applying the K\"unneth short exact sequences). However the theory we develop is more general. 

    A \emph{persistence module} is a functor $M: \cat{P} \to \cat{Vect_k}$. A morphims $f:M\to N$ between persistence modules $M$ and $N$ is a natural transformation. We denote the category of persistence modules and their morphisms by $\cat{k^P}$ instead of $\cat{Vect_k^P}$ for the sake of simplicity. Thus, the sets of morphisms between persistence modules in $\cat{k^P}$ will be denoted by $\Hom_{\cat{k^P}}(M,N)$.
    
\begin{remark}
\label{remark:persistence_modules_as_sheaves}
    By \cref{theorem:functors_and_sheaves}, we can apply ideas from sheaf theory to persistence modules. We will not use the formulation of persistence modules as sheaves on $\mathbb{R}^n$ with the Alexandrov topology explicitly. However, this simple observation allows us to import many sheaf-theoretic results in the setting of persistence modules for free.
\end{remark}

\begin{definition}
    Let $(P,\le)$ be a poset. We call $U\subset P$ \emph{convex} if $a\le c\le b$, with $a,b\in U$ implies that $c\in U$. We call $U\subset P$ \emph{connected} if for any two $a,b\in U$ there exists a sequence $a=p_0\le q_1\ge p_1\le q_2\ge\cdots \ge p_n\le q_n=b$ for some $n\in \mathbb{N}$ such that all $p_i,q_i\in U$ for $1\le i\le n$. A connected convex subset of a poset is called an \emph{interval}.
\end{definition}

Let $A\subset P$ be a convex subset. The \emph{indicator persistence module} on $A$ is the persistence module $\cat{k}[A]:\cat{P}\to \cat{Vect_k}$ given by
\[\cat{k}[A]_a\od \begin{cases}
    \cat{k},& a\in A\\
    0,& a\not\in A
\end{cases},\]
and all the maps $\cat{k}[A]_{a\le b}$, where $a,b\in A$ are identity maps. If $A$ is an interval, then $\cat{k}[A]$ is called an \emph{interval persistence module}. If $A$ is an interval on the real line, say $A=[a,b)$, we will write $\cat{k}[a,b)$ instead of $\cat{k}[[a,b)]$ for brevity. Interval modules play a central role in persistence theory due to the following result.

\begin{theorem}{\cite{Crawley-Boevey2015}}[Structure theorem for one-parameter persistence modules]
\label{theorem:structure_theorem}
	Let $M:\cat{R}_+\to \cat{Vect_k}$ be a persistence module such that $M_x$ is finite-dimensional for each $x\in \mathbb{R}_+$. Then $M$ decomposes as a direct sum 
	\[M\approx \bigoplus_{i}\cat{k}[A_i]\]
	where each $A_i\subset \mathbb{R}_+$ is an interval. This decomposition is unique up to isomorphism.
\end{theorem}

As a category of functors valued in an abelian category, $\cat{k^P}$ is an abelian category \cite{Grothendieck1957}. Direct sums, kernels and cokernels are taken elementwise.  Exact sequences, left and right exact functors and their derived functors of persistence modules are also well defined for the category $\cat{k^P}$. 

For any persistence modules $M,N:\cat{P}\to \cat{Vect_k}$, the set of natural transformations $\Hom_{\cat{k^P}}(M,N)$ has the structure of a $\cat{k}$-vector space. Indeed, if $f,g:M\to N$ are natural transformations and $a\in \cat{k}$, we define $(a\cdot f+g):M\to N$ by
\[(a\cdot f+g)_p\od a\cdot f_p+g_p\in \Hom_{\cat{k}}(M_p,N_p), \forall p\in P,\]
where the operations on the right hand side are the addition and scalar multiplication of linear maps. 

\begin{lemma}
\label{lemma:nonzero_map_of_interval_modules}
    Suppose that $A$ and $B$ are intervals in $P$ such that $A\cap B$ is connected.
    Assume there is a nonzero map $f:\cat{k}[A]\to \cat{k}[B]$ of interval modules. Then (up to isomorphism) $f_a=1$ if $a\in A\cap B$ and $f_a=0$ otherwise.
\end{lemma}

\begin{proof}
This is an elementary consequence of the commutativity that diagrams of natural transformations have to satisfy. More details are given in \cref{section:appendix_A}.
\qedhere
\end{proof}

The connectedness hypothesis for $A\cap B$ is necessary, as the following example shows.

\begin{example}{\cite[Proposition 3.10]{Miller2025}}
    Let $D,U\subset P$ be a down-set and an up-set, respectively. Then
    $\Hom_{\cat{P}}(\cat{k}[U],\cat{k}[D])$ is a product of copies of $\cat{k}$, one for each connected component of $U\cap D$, 
    \[\Hom_{\cat{P}}(\cat{k}[U],\cat{k}[D])\approx \cat{k}^{\pi_0(U\cap D)}.\]
    
\end{example}

The following observation will be useful for computing internal hom functors of interval modules in \cref{section:4}.

\begin{definition}
\label{def:nicely_overlap}
    Suppose $A$ and $B$ are intervals in $P$ such that $A\cap B\neq \varnothing$ is connected. We say that $A$ and $B$ \emph{overlap nicely} if the following two conditions are both satisfied:
    \begin{enumerate}[left=0pt]
        \item There is no $b\in B\setminus A$ such that there is an $a\in A\cap B$ with $a\le b$.
        \item There is no $a\in A\setminus B$ such that there is a $b\in A\cap B$ with $a\le b$.
    \end{enumerate}
\end{definition}

\begin{lemma}
\label{lemma:morphisms_between_intervals}
    Let $A,B$ be intervals of $P$ such that $A\cap B$ is connected. Then 
    \[\Hom_{\cat{k^{P}}}(\cat{k}[A],\cat{k}[B]) \approx
  \begin{cases}
    \cat{k}, &\text{if } A \text{ and } B \text{ overlap nicely}\\
    0, &\textnormal{otherwise}
  \end{cases}.\]
\end{lemma}

\begin{proof}
The commutativity of diagrams that define natural transformations force this result. A detailed proof is  given in \cref{section:appendix_A}.
\end{proof}

\begin{example}
\label{example:natural_transformations_form_a_vector_space}
For $P=\mathbb{R}_+$, suppose $\mathbf{k}[a,b)$ and $\mathbf{k}[c,d)$ are interval modules. Then, due to the constraints of commutative squares for natural transformations, we have:
\[\Hom_{\cat{k^{R_{+}}}}(\mathbf{k}[a,b),\mathbf{k}[c,d)) \approx
  \begin{cases}
    \mathbf{k}, &\text{if } c \le a < d \le b\\
    0, &\textnormal{otherwise}
  \end{cases}
.\]
This was observed in \cite[Appendix 9]{Bubenik2018}, but it also follows by observing that the intervals $[a,b)$ and $[c,d)$ overlap nicely precisely when $c \le a < d \le b$ and applying \cref{lemma:morphisms_between_intervals}.
\end{example}

\begin{proposition}[Limit characterization of $\Hom_{\cat{k^P}}(-,-)$]
\label{proposition:limit_characterization_of_hom}
Let $M,N:\cat{P}\to \cat{Vect_k}$ be persistence modules. Then
\[\Hom_{\cat{k^P}}(M,N)\approx \underset{a\le b}{\lim}\,\Hom_{\cat{k}}(M_a,N_b).\]
\end{proposition}

\begin{proof}
    Let $f\in \Hom_{\cat{k}}(M_a,N_a)$ and $g\in \Hom_{\cat{k}}(M_b,N_b)$, where
    $a\le b$. The canonical maps $\Hom_{\cat{k}}(M_a,N_a)\to \Hom_{\cat{k}}(M_a,N_b)$ and $\Hom_{\cat{k}}(M_b,N_b)\to \Hom_{\cat{k}}(M_a,N_b)$ are the post-composition and pre-composition by $N_{a\le b}$ and $M_{a\le b}$, respectively. If $f$ and $g$ are components of a natural transformation in $\Hom_{\cat{k^p}}(M,N)$, then the square in \cref{fig:limit_char_of_natural_transformations} commutes. Equivalently, $f$ and $g$ are mapped to the same morphism under the above maps (see \cref{fig:limit_char_of_natural_transformations}).
\end{proof}

\begin{figure}[H]
\centering
\begin{tikzcd}
M_{a}\arrow[d,"f"]\arrow[rr,"M_{a\le b}"]&&M_{b}\arrow[d,"g"]\\
N_a\arrow[rr,"N_{a\le b}"]&&N_b
\end{tikzcd}
\caption{Commutativity of natural transformations is equivalent to a limit characterization of the appropriate hom sets. }
\label{fig:limit_char_of_natural_transformations}
\end{figure}

\subsection{Tensor product and internal hom of persistence modules}
Here we recall the tensor product and internal hom of persistence modules.

\begin{definition}
\label{def:sheaf_tensor_product}
Let $M,N:\cat{P}\to \cat{Vect_k}$ be persistence modules. The \emph{tensor product} of $M$ and $N$ is the persistence module $M\Ptensor N:\cat{P}\to \cat{Vect_k}$ defined by \[(M\Ptensor N)_p\od M_p\ktensor N_p,\] 
for $p\in P$ and $(M\Ptensor N)_{p\le q}\od M_{p\le q}\ktensor N_{p\le q}$ for $p\le q$ in $P$. Here $\ktensor$ is the tensor product of $\cat{k}$-vector spaces over $\cat{k}$.
\end{definition}

The bifunctor $-\Ptensor-$ and its derived functor were studied extensively in \cite{Bubenik2021} where it was denoted by $\shtensor$. This tensor product is the sheaf tensor product of $M$ and $N$ when $M$ and $N$ are considered as sheaves, see \cref{remark:persistence_modules_as_sheaves}.

\begin{example}
\label{example:sheaf_tensor}
    Assume that $(P,\le)=(\mathbb{R}^n_{+},\le)$ and let $U,V\subset \mathbb{R}^n_{+}$ be intervals. Let $\cat{k}[U]$ and $\cat{k}[V]$ be the corresponding interval modules. For $a\in \mathbb{R}^n_{+}$, 
   \[(\cat{k}[U]\otimes_{\cat{R}^n_+} \cat{k}[V])_a=\cat{k}[U]_a\ktensor \cat{k}[V]_a\approx\begin{cases}
   \cat{k}, & a\in U\cap V\\
   0, & \text{otherwise}
   \end{cases}.\] 
   If $U\cap V$ is connected it follows that $\cat{k}[U]\otimes_{\cat{R}^n_+} \cat{k}[V]\approx \cat{k}[U\cap V]$. If $n=1$, $U=[a,b)$ and $V=[c,d)$ we have that: 
    \[\cat{k}[U]\otimes_{\cat{R}_+} \cat{k}[V]\approx\cat{k}[[a,b)\cap[b,c)]\approx\cat{k}[\max(a,c),\min(\max(b,c),\max(a,d))).\]
\end{example}

For $p\in P$, let $U_p=\{p'\in P\,|\, p\le p' \}$. The partial order on $P$ restricts to a partial order on $U_p$ and $\cat{U}_p$ is a full subcategory of $\cat{P}$. Furthermore, any persistence module $M$ on $\cat{P}$ restricts to a persistence module $M|_{\cat{U}_p}$ on $\cat{U}_p$ in the obvious way.

\begin{definition}
\label{def:sheaf_hom}
    Given persistence modules $M,N:\cat{P}\to \cat{Vect_k}$ there is a persistence module $\scHom_{\cat{P}}(M,N):\cat{P}\to \cat{Vect_k}$ defined by 
\[\scHom_{\cat{P}}(M,N)_p\od \Hom_{\cat{k}^{\cat{U}_p}}(M|_{\cat{U}_p},N|_{\cat{U}_p}),\]
for any $p\in P$. We call $\scHom_{\cat{P}}(M,N)$ the \emph{internal hom} of $M$ and $N$.
\end{definition}
In other words, the $\cat{k}$-vector space $\scHom_{\cat{P}}(M,N)_p$ is the vector space of natural transformation between the restrictions of functors $M|_{\cat{U_p}}$ and $N|_{\cat{U_p}}$. The linear maps 
\[\scHom_{\cat{P}}(M,N)_{p\le q}:\scHom_{\cat{P}}(M,N)_p\to \scHom_{\cat{P}}(M,N)_q,\] are then given by the obvious restrictions of natural transformations (deleting the components of the natural transformations that are not indexed by elements in $U_q$). This internal hom functor is the sheaf internal hom of $M$ and $N$ when $M$ and $N$ are considered as sheaves, see \cref{remark:persistence_modules_as_sheaves}.

\begin{example}
\label{example:sheaf_hom_of_interval_modules}
The following computation is similar to \cite[Example 4.3]{Bubenik2021}, except that we assume $P=\mathbb{R}_+$ instead of $P=\mathbb{R}$. Consider interval modules $\mathbf{k}[a,b)$ and $\mathbf{k}[c,d)$. We have the following. 
\[\scHom_{\cat{R}_+}(\cat{k}[a,b),\cat{k}[c,d))\approx \begin{cases}
0 & \text{if } a < b\le c < d\\
0 & \text{if } a < c\le b < d\\
\cat{k}[c,d) & \text{if } a < c < d\le b\\
0 & \text{if } c\le a < b < d\\
\cat{k}[0,d) & \text{if } c\le a < d\le b\\
0 & \text{if } c < d\le a < b
\end{cases}\]
To see this, note that by definition we have 
\begin{gather*}
\scHom_{\cat{R}_+}(\mathbf{k}[a,b),\mathbf{k}[c,d))_x
\od\Hom_{\cat{k}^{\cat{[x,\infty)}}}(\cat{k}[a,b)|_{\cat{[x,\infty)}},\cat{k}[c,d)|_{\cat{[x,\infty)}})
\end{gather*}
Thus, we need to compute the set of natural transformations between the functors \[\mathbf{k}[a,b)|_{\cat{[x,\infty)}},\mathbf{k}[c,d)|_{\cat{[x,\infty)}}:\cat{[x,\infty)}\to \mathbf{Vect}_{\mathbf{k}},
\]
where $[x,\infty)$ is given the total linear order induced from $\mathbb{R}_{\ge 0}$. 
Note that $\mathbf{k}[a,b)|_{\cat{[x,\infty)}}$ is nonzero if and only if $x \in [0,b)$.
Consider the case $c\le a < d\le b$. 
As in \cref{example:natural_transformations_form_a_vector_space}, we see that 
\[\text{Hom}_{\cat{k^{[x,\infty)}}} (\mathbf{k}[a,b)|_{\cat{[x,\infty)}},\mathbf{k}[c,d)|_{\cat{[x,\infty)}})\approx \begin{cases}\mathbf{k}, & x\in [0,d)\\
0, & \text{otherwise}
\end{cases}.\] 

The other cases may be computed similarly.
Similar arguments also shows that
\begin{equation*}
  \scHom_{\cat{R}_+}(\mathbf{k}[a,\infty),\mathbf{k}[c,d)) \approx
  \begin{cases}
    \mathbf{k}[c,d) &\text{if } a < c\\
    \mathbf{k}[0,d) &\text{if } c \leq a < d \\
    0 &\text{if } d \leq a
  \end{cases}
\end{equation*}
and that 
\[\scHom_{\cat{R}_+}(\cat{k}[a,b),\cat{k}[\mathbb{R}_+])=0, \quad \text{ and }\quad \scHom_{\cat{R}_+}(\cat{k}[a,\infty),\cat{k}[\mathbb{R}_+])\approx\cat{k}[\mathbb{R}_+].\]
\end{example}

From classical sheaf theory and by \cref{remark:persistence_modules_as_sheaves} we immediately have the following (For a proof, see for example \cite[Chapter 2]{Kashiwara1990}).

\begin{theorem}
\label{theorem:hom_tensor_adjunction}
    For any persistence module $M$, the functor $M\Ptensor -$
is left adjoint to the functor $\scHom_{\cat{P}}(M,-)$. That is, there are natural bijections
\[\Hom_{\cat{k^P}}(M\Ptensor N,L)\approx \Hom_{\cat{k^P}}(N,\scHom_{\cat{P}}(M,L)),\]
for any persistence modules $M, N$ and $L$. Similarly the functor $-\Ptensor N$ is left adjoint to the functor $\scHom_{\cat{P}}(N,-)$. 
That is, there are natural bijections
\[\Hom_{\cat{k^P}}(M\Ptensor N,L)\approx \Hom_{\cat{k^P}}(M,\scHom_{\cat{P}}(N,L)),\]
for any persistence modules $M, N$ and $L$. 
\end{theorem}

\subsection{Chain complexes of persistence modules from CW complex filtrations}
\label{subsection:chain_complexes_of_pers_mod}
Here we recall how to construct chain complexes of persistence modules from a filtered CW complex, mostly following the ideas presented in \cite[Section 2.5]{Bubenik2021}. A \emph{filtration} on a CW complex $X$ is a function $f:X\to \mathbb{R}_+$ that is constant on the cells of $X$ and such that $f(\partial\sigma)\le f(\sigma)$ for all cells $\sigma$. 

For $a\in \mathbb{R}_+$, let $X^f_a$ be the subcomplex of $X$ defined by 
\[X^f_a\od f^{-1}[0,a].\] The collection of CW complexes $\{X^f_a\}_{a\in \mathbb{R}_{+}}$ with the inclusion maps $X^f_a\hookrightarrow X^f_b$ whenever $a\le b$ is a filtered CW complex. The inclusion maps induce $\cat{k}$-linear maps on cellular homology with coefficients in a field $\cat{k}$, $H_n(X^f_a;\cat{k})\to H_n(X^f_b;\cat{k})$. Let $\mathcal{H}_n(X^f)$ denote the resulting persistence module.

Let $X$ be a CW complex with filtration $f$. Let $X^{(m)}$ denote the set of $m$-cells of X. For $m\ge 0$, define 
\[C^f_m(X)\od\bigoplus_{\sigma\in X^{(m)}}\cat{k}[f(\sigma),\infty).\]
 For $\sigma\in X^{(m)}$ let $\sigma$ also denote the generator of $\cat{k}[f(\sigma),\infty)_{f(\sigma)}$. For $a\ge f(\sigma)$, let $\sigma_a$ denote $\cat{k}[f(\sigma),\infty)_{f(\sigma)\le a}(\sigma)$ i.e., $\sigma_a$ is the image of the generator $\sigma$ under the linear map $\cat{k}[f(\sigma),\infty)_{f(\sigma)\le a}$. Similarly, define $\alpha_a\in C^f_m(X)$ for an $m$-chain in the cellular chain complex on $X$. Define $d^f_m:C^f_m(X)\to C^f_{m-1}(X)$ to be the natural transformation obtained by extending the definition $(d^f_m)_a(\sigma_a)\od (\partial\sigma)_a$ linearly. Then by construction $(C^f_{\bullet}(X),d_{\bullet})$ is a chain complex of persistence modules. Let $H_n(C^f(X))$ be the homology of the chain complex $(C^f_{\bullet}(X),d_{\bullet})$. The following is an observation made in \cite[Lemma 2.26]{Bubenik2021} for filtrations on a CW complex $f:X\to \mathbb{R}$. However the same proof applies here when we restrict the codomain to $\mathbb{R}_{+}$. 

\begin{lemma}
    Let $X$ be a CW complex with a filtration as above. For all $n\in \mathbb{N}$, $H_n(C^f(X))\approx \mathcal{H}_n(X^f)$ as persistence modules.
\end{lemma}

\begin{proof}
Same arguments as in \cite[Lemma 2.26]{Bubenik2021} apply.
\end{proof}

The following example illustrates this result.

\begin{example}
\label{example:triangle_filtration}
Let $0\le a_0\le a_1\le a_2\le b_0\le b_1\le b_2\le c$ be real numbers and consider the filtration of the $2$-simplex $X$ in \cref{fig:21} given by $f(x)=a_0$, $f(y)=a_1$, $f(z)=a_2$, $f([x,y])=b_0$, $f([y,z])=b_1$, $f([x,z])=b_2$ and $f([x,y,z])=c$.
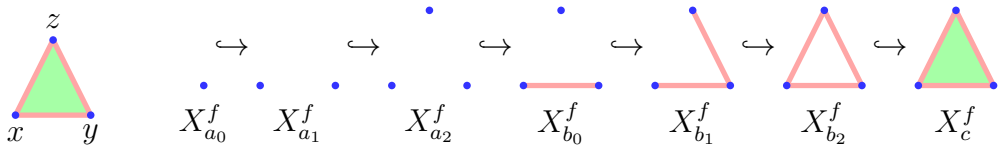
\begin{figure}[H]
\centering
\begin{tikzpicture}[line cap=round,line join=round,x=1.0cm,y=1.0cm,scale=0.5]
\draw[fill=green!35, line width=0.2mm] (18,2)--(20,2)--(19,4)--cycle;
\draw[color=red!35, line width=0.7mm] (18,2)--(20,2)--(19,4)--cycle;
\fill [color=blue!80] (18,2) circle (1mm);
\fill [color=blue!80] (20,2) circle (1mm);
\fill [color=blue!80] (19,4) circle (1mm);
\draw[color=black] (18,1.5) node {$x$};
\draw[color=black] (20,1.5) node {$y$};
\draw[color=black] (19,4.5) node {$z$};
\end{tikzpicture}
\hspace{0.5cm}
\begin{tikzpicture}[line cap=round,line join=round,x=1.0cm,y=1.0cm,scale=0.5]
\fill [color=blue!80] (-1,2) circle (1mm);
\draw[color=black] (-1,1) node {$X^f_{a_0}$};
\draw[color=black] (-0.25,3) node[scale=1] {$\hookrightarrow$};

\fill [color=blue!80] (0.5,2) circle (1mm);
\fill [color=blue!80] (2.5,2) circle (1mm);

\draw[color=black] (1.5,1) node {$X^f_{a_1}$};
\draw[color=black] (3.25,3) node[scale=1] {$\hookrightarrow$};
\fill[color=blue!80] (4,2) circle(1mm);
\fill[color=blue!80] (6,2) circle(1mm);
\fill[color=blue!80] (5,4) circle (1mm);
\draw[color=black] (6.75,3) node[scale=1] {$\hookrightarrow$};
\draw[color=black] (5,1) node {$X^f_{a_2}$};
\draw[color=red!35, line width=0.7mm] (7.5,2)--(9.5,2);
\fill[color=blue!80](7.5,2) circle(1mm);
\fill[color=blue!80] (9.5,2) circle (1mm);
\fill[color=blue!80] (8.5,4) circle (1mm);
\draw[color=black] (8.5,1) node {$X^f_{b_0}$};

\draw[color=black] (10.25,3) node[scale=1] {$\hookrightarrow$};

\draw[color=red!35, line width=0.7mm] (11,2)--(13,2);
\draw[color=red!35, line width=0.7mm] (13,2)--(12,4);
\fill[color=blue!80](11,2) circle(1mm);
\fill[color=blue!80] (13,2) circle (1mm);
\fill[color=blue!80] (12,4) circle (1mm);
\draw[color=black] (12,1) node {$X^f_{b_1}$};

\draw[color=black] (13.75,3) node[scale=1] {$\hookrightarrow$};

\draw[color=red!35, line width=0.7mm] (14.5,2)--(16.5,2);
\draw[color=red!35, line width=0.7mm] (16.5,2)--(15.5,4);
\draw[color=red!35, line width=0.7mm] (15.5,4)--(14.5,2);
\fill[color=blue!80](14.5,2) circle(1mm);
\fill[color=blue!80] (16.5,2) circle (1mm);
\fill[color=blue!80] (15.5,4) circle (1mm);

\draw[color=black] (15.5,1) node {$X^f_{b_2}$};

\draw[color=black] (17.25,3) node[scale=1] {$\hookrightarrow$};
\draw[fill=green!35, line width=0.2mm] (18,2)--(20,2)--(19,4)--cycle;
\draw[color=red!35, line width=0.7mm] (18,2)--(20,2)--(19,4)--cycle;
\fill[color=blue!80](18,2) circle(1mm);
\fill[color=blue!80] (20,2) circle (1mm);
\fill[color=blue!80] (19,4) circle (1mm);
\draw[color=black] (19,1) node {$X^f_c$};
\end{tikzpicture}
\caption{A $2$-simplex $X$ with vertices $\{x,y,z\}$ (left) and sublevel sets of the filtration of $X$ by $f$ (right). Here $X^f_t\od f^{-1}([0,t])$.}
\label{fig:21}
\end{figure}

The corresponding chain complex of persistence modules associated with this filtration is:
\begin{align*}
    C^f_0(X)&=\cat{k}[a_0,\infty)\oplus\cat{k}[a_1,\infty)\oplus\cat{k}[a_2,\infty),\\
    C^f_1(X)&=\cat{k}[b_0,\infty)\oplus\cat{k}[b_1,\infty)\oplus \cat{k}[b_2,\infty),\\
    C^f_2(X)&=\cat{k}[c,\infty), C^f_n(X)=0,n\ge 3,
\end{align*}
with boundary maps induced by the boundary maps of the $2$-simplex $X$.
The homology of the chain complex $C^f(X)$ is:
\[H_0(C^f(X))=\mathcal{H}_0(X^f)=\cat{k}[a_0,\infty)\oplus\cat{k}[a_1,b_0)\oplus \cat{k}[a_2,b_1),
H_1(C^f(X))=\mathcal{H}_1(X^f)=\cat{k}[b_2,c),\]
and $H_n(C^f(X))=\mathcal{H}_n(X^f)=0$ for $n\ge 2$.
\end{example}

\section{$\varphi$-tensor products and internal homs of persistence modules}
\label{section:4}
Here we introduce a tensor product for persistence modules, $\otimes_{\varphi}$ that depends on an order-preserving map $\varphi:P\times P \to P$. We show this ``$\varphi$-dependent" tensor product has an adjoint, a ``$\varphi$-dependent" internal hom functor, $\scHom^{\varphi}$. We determine the acyclic objects with respect to $\otimes_{\varphi}$, classify injective and projective persistence modules for the case $P=\mathbb{R}_+$. We then use projective and injective resolutions of persistence modules to compute the derived functors $\cat{Tor}^{\varphi}$ and $\cat{Ext}_{\varphi}$ of $\otimes_{\varphi}$ and $\scHom^{\varphi}$, respectively, for interval modules.

\subsection{$\varphi$-tensor products}
 Given a poset $P=(P,\le)$ consider the product poset $P\times P$ and the canonical projection maps $\pi_i:P\times P\to P$, $i=1,2$. Recall the direct image functor with open support (\cref{def:direct_image_2}).

\begin{definition}
\label{def:external_tensor_product}
    Given persistence modules $M,N:\cat{P}\to \cat{Vect_k}$ the \emph{external tensor product} of $M$ and $N$ is the persistence module $M\boxtimes N:\cat{P\times P}\to \cat{Vect_k}$ defined by \[M\boxtimes N\od\pi_{1}^{-1}M\otimes_{\cat{P\times P}}\pi_{2}^{-1}N.\]
\end{definition}

The external tensor product allows us to tensor persistence modules over $P\times P$ by first taking their inverse images under projection maps.. The map $\varphi:P\times P\to P$ allows us to take a direct image of this construction back onto $P$. 

\begin{definition}
\label{def:generalized_convolution}
    Given persistence modules $M,N:\cat{P}\to \cat{Vect_k}$, define \[M \otimes_{\varphi} N \od \varphi_{\dagger}(M \boxtimes N).\] 
    We call $M\otimes_{\varphi} N$ the $\varphi$\emph{-tensor product} of $M$ and $N$.
\end{definition}

The following result allows us to compute $M\otimes_{\varphi}N$ in practice.

\begin{lemma}
\label{lemma:generalized_convolution}
    For every $p\in P$,  $(M\otimes_{\varphi}N)_p= \underset{\varphi(a,b)\le p}{\colim} (M_a\ktensor N_b)$.
\end{lemma}

\begin{proof}
    Let $p \in P$. The statement follows from the definitions:
\begin{align*}
    (\varphi_{\dagger}(M \boxtimes N))_p&\od \underset{\varphi(a,b)\le p}{\colim}(M \boxtimes N)_{(a,b)} & (\textnormal{\cref{def:direct_image_2}})\\
    &\od\underset{\varphi(a,b)\le p}{\colim}(\pi_1^{-1}M \otimes_{\cat{P\times P}} \pi_2^{-1} N)_{(a,b)} & (\textnormal{\cref{def:external_tensor_product}})\\
    &\od \underset{\varphi(a,b)\le p}{\colim}(\pi_1^{-1}M_{(a,b)}\otimes_{\cat{P\times P}} \pi_2^{-1}N_{(a,b)})& (\textnormal{\cref{def:external_tensor_product}})\\ 
    &\od\underset{\varphi(a,b)\le p}{\colim}(M_{\pi_1(a,b)} \ktensor N_{\pi_2(a,b)})& (\textnormal{\cref{def:inverse_image}})\\
    &=\underset{\varphi(a,b)\le p}{\colim}(M_a \ktensor N_b).&\qedhere
\end{align*} 
\end{proof}

\begin{corollary}
\label{corollary:tensor_product_of_infinite_intervals}
Let $M,N:\cat{R}_+\to \cat{Vect_k}$ be the interval modules $M=\cat{k}[a,\infty)$ and $N=\cat{k}[b,\infty)$. Let $\varphi:\mathbb{R}_+\times \mathbb{R}_+ \to \mathbb{R}_+$ be an order-preserving map. Then
\[M\otimes_{\varphi}N\approx \cat{k}[\varphi(a,b),\infty).\]
\end{corollary}

\begin{proof}
	By \cref{lemma:generalized_convolution} we have $(M\otimes_{\varphi}N)_x=\underset{\varphi(c,d)\le x}{\colim} (M_c\ktensor N_d)=\underset{(c,d)\in\ \varphi^{-1}(D_x)}{\colim} (M_c\ktensor N_d)$ for all $x\in \mathbb{R}_+$, where $D_x=[0,x]$ is the principal down-set at $x$ in $\mathbb{R}_+$. 
	
	Suppose that $x<\varphi(a,b)$. Then $x<\varphi(c,d)$ for any $(a,b)\le (c,d)$ in $\mathbb{R}^2_+$ because $\varphi$ is order-preserving. However, $M_c$ is only nonzero if $a\le c$ and $N_d$ is only nonzero if $b\le d$. Thus, for any $(c,d)\in \varphi^{-1}(D_x)$, $M_c=N_d=0$. Therefore $(M\otimes_{\varphi}N)_x=0$.
	
	Now suppose that $\varphi(a,b)\le x$. Then $(a,b)\in\varphi^{-1}(D_x)$. Furthermore, every nonzero  term $M_c\otimes_{\cat{k}}N_d\approx \cat{k}$, for $(c,d)\in \varphi^{-1}(D_x)$, satisfies $(a,b)\le (c,d)$. Therefore, every nonzero summand of $\bigoplus\limits_{(c,d)\in \varphi^{-1}(D_x)}(M_c\otimes_{\cat{k}} N_d)$ is in the image of $M_a\otimes_{\cat{k}}N_b$. Additionally, if $M_c\otimes_{\cat{k}}N_d$ is nonzero for $(c,d)\in \varphi^{-1}(D_x)$ then so is $M_{e}\otimes_{\cat{k}}N_f$ for any $(c,d)\le (e,f)$, $(e,f)\in \varphi^{-1}(D_x)$. Hence $(M\otimes_{\varphi}N)_x\approx \cat{k}$. Finally, for $\varphi(a,b)\le x\le y$ in $\mathbb{R}_+$,  $(M\otimes_{\varphi}N)_{x\le y}:(M\otimes_{\varphi}N)_x\to (M\otimes_{\varphi}N)_y$ will be the identity map for the same reasons.	
\end{proof}

\subsection{$\varphi$-internal homs}

Here $\varphi:P\times P \to P$ is still an order-preserving map, and $\pi_i:P\times P\to P$, for $i=1,2$ are the canonical projection maps. Recall the direct image functor (\cref{def:direct_image_1}).

\begin{definition}
\label{def:convolution_hom}
    For persistence modules $M,N:\cat{P}\to \cat{Vect_k}$, define 
    \[\scHom^{\varphi}(M,N)\od \pi_{2*}\scHom_{\cat{P\times P}}(\pi_1^{-1}M,\varphi^{-1}N),\]
    and  
    \[^{\varphi}\scHom(M,N)\od \pi_{1*}\scHom_{\cat{P\times P}}(\pi_2^{-1}M,\varphi^{-1}N).\]
\end{definition}

These definitions are similar to the definition of the right adjoint of the convolution functor for sheaves on $\mathbb{R}^n$ \cite{Guillermou2014,Tamarkin2018}, where $\varphi:\mathbb{R}^n\times \mathbb{R}^n\to \mathbb{R}^n$ is the addition map $\varphi(x,y)=x+y$, and $\mathbb{R}^n$ has the standard topology which is very different than our setting. However, there the exceptional inverse image $\varphi^{!}$ functor of sheaves is used and not $\varphi^{-1}$ like in our definition above. If the map $\varphi$ is symmetric we have that $\scHom^{\varphi}(M,N)=\,^{\varphi}\scHom(M,N)$ (\cref{corollary:tensor_for_symmetric_phi}). The following result allows us to compute $\scHom^{\varphi}(M,N)$ and $^{\varphi}\scHom(M,N)$ in practice.

\begin{proposition}
\label{proposition:limit_characterization_of_internal_hom}
    For every $p\in P$, we have natural isomorphisms
     \[\scHom^{\varphi}(M,N)_p\approx \underset{a\le b}{\lim\,} \Hom_{\cat{k}}(M_a, N_{\varphi(b,p)}),\]
    and
    \[^{\varphi}\scHom(M,N)_p\approx \underset{a\le b}{\lim\,} \Hom_{\cat{k}}(M_a, N_{\varphi(p,b)}).\]
\end{proposition}

\begin{proof}
Recall that for a poset $P$, we denote by $U_{a}$ the principal up-set at $a\in P$.
For $p\in P$ we have the following natural isomorphisms:
    \begin{align*}
        \scHom^{\varphi}(M,N)_p&\od (\pi_{2*}\scHom_{\cat{P\times P}}(\pi_1^{-1}M,\varphi^{-1}N))_p& (\textnormal{\cref{def:convolution_hom}})\\
        &\od\underset{p\le \pi_2(a,b)}{\lim} \scHom_{\cat{P\times P}}(\pi_1^{-1}M,\varphi^{-1}N)_{(a,b)} & (\textnormal{\cref{def:direct_image_1}})\\
        &\od\underset{p\le \pi_2(a,b)}{\lim} \Hom_{\cat{k}^{\cat{U}_{(a,b)}}}(\pi^{-1}_1M|_{\cat{U}_{(a,b)}},\varphi^{-1}N|_{\cat{U}_{(a,b)}})& (\textnormal{\cref{def:sheaf_hom}})\\
        &=\underset{p\le b}{\lim\,}\underset{a\in P}{\lim\,} \Hom_{\cat{k}^{\cat{U}_{(a,b)}}}(\pi^{-1}_1M|_{\cat{U}_{(a,b)}},\varphi^{-1}N|_{\cat{U}_{(a,b)}})\\
        &\approx \underset{a\in P}{\lim\,} \Hom_{\cat{k}^{\cat{U}_{(a,p)}}}(\pi^{-1}_1M|_{\cat{U}_{(a,p)}},\varphi^{-1}N|_{\cat{U}_{(a,p)}}) & (\textnormal{\cref{lemma:colimits_and_limits_over_diagrams_with_0_and_1}}) \\
        &\approx \underset{a\in P}{\lim\,} \underset{(a,p)\le (c,d)\le (b,f)}{\lim\,} \Hom_{\cat{k}}((\pi^{-1}_1M|_{\cat{U}_{(a,p)}})_{(c,d)},(\varphi^{-1}N|_{\cat{U}_{(a,p)}})_{(b,f)}) &(\textnormal{\cref{proposition:limit_characterization_of_hom}})\\
        &\od \underset{a\in P}{\lim\,} \underset{(a,p)\le (c,d)\le (b,f)}{\lim\,} \Hom_{\cat{k}}(M_{\pi_1(c,d)},N_{\varphi(b,f)})& (\textnormal{\cref{def:inverse_image}})\\ 
        &= \underset{a \in P}{\lim\,} \underset{a\le c\le b}{\lim\,} \underset{p\le f}{\lim\,}\Hom_{\cat{k}}(M_c,N_{\varphi(b,f)})\\ 
        &\approx \underset{a \in P}{\lim\,} \underset{a\le c\le b}{\lim\,} \Hom_{\cat{k}}(M_c,N_{\varphi(b,p)}) & (\textnormal{\cref{lemma:colimits_and_limits_over_diagrams_with_0_and_1}}) \\
        &=\underset{a \in P}{\lim\,} \underset{a\le b}{\lim\,}  \underset{a\le c\le b}{\lim\,}\Hom_{\cat{k}}(M_c,N_{\varphi(b,p)}) \\
        &\approx\underset{a \in P}{\lim\,} \underset{a\le b}{\lim\,} \Hom_{\cat{k}}(M_a,N_{\varphi(b,p)}) & (\textnormal{\cref{lemma:colimits_and_limits_over_diagrams_with_0_and_1}})\\
        &=\underset{a\le b}{\lim\,}  \Hom_{\cat{k}}(M_a,N_{\varphi(b,p)})
    \end{align*}      
    The proof for the other statement is similar.
\end{proof}

\begin{theorem}
\label{theorem:convolution_adjunction}
    The functor $M\otimes_{\varphi}-$ is left adjoint to the functor $\scHom^{\varphi}(M,-)$. In other words, there are natural bijections
    \[\Hom_{\cat{k^P}}(M\otimes_{\varphi}N,L)\approx \Hom_{\cat{k^P}}(N,\scHom^{\varphi}(M,L)),\]
    for any persistence modules $M,N$ and $L$. Similarly, the functor $-\otimes_{\varphi}M$ is left adjoint to the functor $^{\varphi}\scHom(M,-)$. In other words, there are natural bijections
    \[\Hom_{\cat{k^P}}(N\otimes_{\varphi}M,L)\approx \Hom_{\cat{k^P}}(N,\,^{\varphi}\scHom(M,L)),\]
    for any persistence modules $M,N$ and $L$.
\end{theorem}

\begin{proof}
    We have the following canonical bijections:
    \begin{align*}
        \Hom_{\cat{k^P}}(M\otimes_{\varphi}N,L)&\od\Hom_{\cat{k^P}}(\varphi_{\dagger}(M\boxtimes N),L)  & (\textnormal{\cref{def:generalized_convolution}})\\
        &\approx \Hom_{\cat{k^{P\times P}}}(M\boxtimes N,\varphi^{-1}L) & (\textnormal{\cref{theorem:direct_image_adjunction_2}})\\
        &\od\Hom_{\cat{k^{P\times P}}}(\pi_1^{-1}M\otimes_{\cat{P\times P}} \pi_2^{-1}N,\varphi^{-1}L)& (\textnormal{\cref{def:external_tensor_product}})\\
        &\approx \Hom_{\cat{k^{P\times P}}}(\pi_2^{-1}N,\scHom_{\cat{P\times P}}(\pi_1^{-1}M,\varphi^{-1}L))& (\textnormal{\cref{theorem:hom_tensor_adjunction}})\\
        &\approx\Hom_{\cat{k^P}}(N,\pi_{2*}\scHom_{\cat{P\times P}}(\pi_1^{-1}M,\varphi^{-1}L)& (\textnormal{\cref{theorem:direct_image_adjunction_1}})\\
        &\od\Hom_{\cat{k^P}}(N,\scHom^{\varphi}(M,L))& (\textnormal{\cref{def:convolution_hom}}).
    \end{align*}
    
        Similarly, we also have the following canonical bijections:
    \begin{align*}
        \Hom_{\cat{k^P}}(N\otimes_{\varphi}M,L)&\od\Hom_{\cat{k^P}}(\varphi_{\dagger}(N\boxtimes M),L)  & (\textnormal{\cref{def:generalized_convolution}})\\
        &\approx \Hom_{\cat{k^{P\times P}}}(N\boxtimes M,\varphi^{-1}L) & (\textnormal{\cref{theorem:direct_image_adjunction_2}})\\
        &\od\Hom_{\cat{k^{P\times P}}}(\pi_1^{-1}N\otimes_{\cat{P\times P}} \pi_2^{-1}M,\varphi^{-1}L)& (\textnormal{\cref{def:external_tensor_product}})\\
        &\approx \Hom_{\cat{k^{P\times P}}}(\pi_1^{-1}N,\scHom_{\cat{P\times P}}(\pi_2^{-1}M,\varphi^{-1}L))& (\textnormal{\cref{theorem:hom_tensor_adjunction}})\\
        &\approx\Hom_{\cat{k^P}}(N,\pi_{1*}\scHom_{\cat{P\times P}}(\pi_2^{-1}M,\varphi^{-1}L)& (\textnormal{\cref{theorem:direct_image_adjunction_1}})\\
        &\od\Hom_{\cat{k^P}}(N, \,^{\varphi}\scHom(M,L))& (\textnormal{\cref{def:convolution_hom}}).&\qedhere
    \end{align*}
\end{proof}

Since left adjoints are right exact and right adjoints are left exact we immediately have the following.

\begin{corollary}
\label{corollary:tensor_and_internal_hom_adjoints}
    The functors $M\otimes_{\varphi}-, -\otimes_{\varphi}M :\cat{k^P}\to \cat{k^P}$ are right exact, and the functors $\scHom^{\varphi}(M,-),\,^{\varphi}\scHom(M,-):\cat{k^P}\to\cat{k^P}$ are left exact.
\end{corollary}

Furthermore, if $\varphi$ is symmetric we also have the following.

\begin{corollary}
\label{corollary:tensor_for_symmetric_phi}
    If $\varphi$ is symmetric, then $M\otimes_{\varphi}-$ and $-\otimes_{\varphi}M$ are naturally isomorphic as functors and so are  $\scHom^{\varphi}(M,-)$ and $^{\varphi}\scHom(M,-)$.
\end{corollary}

\begin{proof}
	For any persistence module $N:\cat{k}\to \cat{Vect_k}$, by \cref{lemma:generalized_convolution}, since $\varphi$ is symmetric and $\otimes_{\cat{k}}$ is commutative we have 
	\[
	(M\otimes_{\varphi}N)_p= \underset{\varphi(a,b)\le p}{\colim} (M_a\ktensor N_b)
	\approx\underset{\varphi(b,a)\le p}{\colim} (N_b\ktensor M_a)
	=(N\otimes_{\varphi} M)_p.
	\]
	Therefore $M\otimes_{\varphi}-$ and $-\otimes_{\varphi}M$ are naturally isomorphic. Since adjoints are unique up to isomorphism, it follows that $\scHom^{\varphi}(M,-)$ and $^{\varphi}\scHom(M,-)$ are also naturally isomorphic by \cref{corollary:tensor_and_internal_hom_adjoints}.
\end{proof}

\subsection{The p-quasinorm tensor product}
Here we assume we are working over the poset $\mathbb{R}^n_{+}$. For any $p\in (0,\infty]$, the \emph{componentwise} $p$-quasinorm $||(-,-)||^c_p:\mathbb{R}^n_{+}\times \mathbb{R}^n_{+}\to \mathbb{R}^n_{+}$ is defined by 
\[||((x_1,\dots,x_n),(y_1,\dots y_n))||_p^c\od(||(x_1,y_1)||_p,\dots,||(x_n,y_n)||_p),\]
where $||-,-||_p:\mathbb{R}_{+}\times \mathbb{R}_{+}\to \mathbb{R}_{+}$ is the $p$-quasinorm on $\mathbb{R}^2_{+}$, $||(a,b)||_p\od(a^p+b^p)^{\frac{1}{p}}$. 

When $p\in [1,\infty]$, then $||\cdot ||_p$ is a norm, however since we want to allow for the highest level of generality as possible, we also allow for $p\in (0,1)$ when $||\cdot ||_p$ is only a quasinorm. Note that the map $\ell_c^p$ is order-preserving with respect to the partial orders. Indeed, $\ell^p_c$ is the composition 
\[((x_1,\dots, x_n),(y_1,\dots,y_n)))\mapsto ((x_1,y_1),\dots, (x_n,y_n))\mapsto (||(x_1,y_1)||_p,\dots, ||(x_n,y_n)||_p),\]
and every map in the composition is order-preserving. We will denote the map $\||-,-||_p^c$ by $\ell_c^p$ for the sake of brevity (the ``c" represents the word ``componentwise").

Recalling \cref{def:generalized_convolution}, we observe the existence of a continuum of different tensor product functors for persistence modules on $\cat{R}^n_+$.

\begin{definition}
    Let $M,N:\cat{R^n_+}\to \cat{Vect_k}$ be persistence modules and $p\in (0,\infty]$. The \emph{componentwise $p$-quasinorm
    tensor product} or simply $\ell^{p}_c$-\emph{tensor} of $M$ and $N$ is $M\ptensor N$. 
\end{definition}

\cref{lemma:generalized_convolution} gives a formula for calculating $M\ptensor N$.

\begin{corollary}
\label{corollary:convolution_as_colimit}
    Let $M,N:\cat{R^n_+}\to \cat{Vect_k}$ be persistence modules. Then for all $x\in \mathbb{R}^n_{+}$ we have that
    \[(M\ptensor N)_x=\underset{\cpnorm(a,b)\le x}{\colim}(M_a\ktensor N_b).\]
\end{corollary}

Recall \cref{def:sheaf_tensor_product}.  We immediately observe the following.

\begin{lemma}
    \label{lemma:sh_tensor_as_a_convolution}
    For any persistence modules $M,N:\cat{R}^n_+\to \cat{Vect_k}$, there are natural isomorphisms:
    \[M\otimes_{\ell^{\infty}_c}N\approx M\otimes_{\cat{R}^n_+}N.\]
\end{lemma}

\begin{proof}
    Let $x\in \mathbb{R}^n_{+}$. Note that $\ell^{\infty}_c(a,b)\le x$ if and only if $a\le x$ and $b\le x$, for any $a,b\in \mathbb{R}^n_+$. Therefore we have natural isomorphisms
    \begin{align*}
        (M\otimes_{\ell^{\infty}_c}N)_x&= \underset{\ell^{\infty}_c(a,b)\le x}{\colim}(M_a\ktensor N_b) & (\textnormal{\cref{corollary:convolution_as_colimit}})\\
        &=\underset{(a,b)\le (x,x)}{\colim}(M_a\ktensor N_b)\\
        &\approx M_x\ktensor N_x&(\textnormal{\cref{lemma:colimits_and_limits_over_diagrams_with_0_and_1}})\\
        &\od (M\otimes_{\cat{R}^n_+} N)_x.
    \end{align*}
    This holds for each $x$, and since each such isomorphism is natural, we have a natural isomorphism  $M\otimes_{\ell^{\infty}_c}N\approx M\otimes_{\cat{R}^n_+}N$.
\end{proof}

Using the colimit characterization of $\ptensor$, we show that $\ptensor$-products of interval modules are themselves interval modules.

\begin{proposition}
\label{proposition:tensor_product_formulas}
    Let $p\in (0,\infty]$. Then 
    \begin{enumerate}[left=0pt]
        \item $\cat{k}[a,b)\ptensor \cat{k}[c,d)\od\cat{k}[||(a,c)||_p, \min (||(b,c)||_p,||(a,d)||_p))$.
        \item $\cat{k}[a,\infty)\ptensor \cat{k}[c,d)\approx \cat{k}[\cpnorm(a,c),\cpnorm(a,d))\od \cat{k}[||(a,c)||_p,||(a,d)||_p)$.
        \item $\cat{k}[a,\infty)\ptensor \cat{k}[c,\infty)\approx \cat{k}[\cpnorm(a,c),\infty)\od \cat{k}[||(a,c)||_p,\infty)$.
    \end{enumerate}
\end{proposition}

\begin{proof}
    We prove the isomorphism in (1). The other isomorphisms follow by similar arguments. Let $M=\cat{k}[a,b)$ and $N=\cat{k}[c,d)$. Assume that $\cpnorm(b,c)\le \cpnorm(a,d)$. \cref{fig:convolution_of_interval_modules} illustrates this case for $p\in \{\frac{1}{2},1,2,\infty\}$. Let $r\in \mathbb{R}_{+}$ and let $X_r\od\bigoplus\limits_{\cpnorm(s,t)=r}(M_s\ktensor N_t)$. For $\cpnorm(a,c)\le r\le \cpnorm (b,c)$, every summand of $X_r$ is in the image of $M_a\ktensor N_b\approx \cat{k}$, and therefore 
    $(M\ptensor N)_r= \underset{\cpnorm(s,t)\le r}{\colim} (M_s\ktensor N_t)\approx k$. Furthermore for $\cpnorm(a,c)\le r\le r'\le \cpnorm (b,c)$, $(M\ptensor N)_{r\le r'}$ is the identity map.

    For $\cpnorm(b,c)\le r$, each nonzero summand $M_s\ktensor N_t$ of $X_r$ has $t>c$ and thus lies in the image of $M_s\ktensor N_c\approx k$. However, the map 
    $M_s\ktensor N_c\to M_l\ktensor N_t$ where $l$ is such that $\cpnorm(l,c)=r$ has to be the zero map as $\cpnorm (b,c)\le r$ and thus $b\le l$ and therefore $M_l=0$ (\cref{fig:convolution_of_interval_modules}). Hence $M\ptensor N\approx \cat{k}[\cpnorm(a,c),\cpnorm (b,c))$.

    If the case was that $\cpnorm(a,d)\le \cpnorm(b,c)$, then the same argument shows that $M\ptensor N\approx \cat{k}[\cpnorm(a,c),\cpnorm(a,d))$. Combining these two cases we have 
    $M\ptensor N\approx \cat{k}[\cpnorm (a,c),\min(\cpnorm(b,c),\cpnorm(a,d)))$.
\end{proof}

\begin{figure}[H]
\centering
\begin{tikzpicture}[line cap=round,line join=round,,x=1.0cm,y=1.0cm,scale=0.5]
\draw[->,color=black,dashed] (0,0) -- (11.5,0);
\draw[->,color=black,dashed] (0,0) -- (0,11.5);
\draw[color=red,line width=0.8mm] (2,0) -- (5,0);
\draw[color=red,line width=0.8mm] (0,3) -- (0,9);
\fill [color=blue] (2,0) circle (1.5mm);
\fill [color=blue] (5,0) circle (1.5mm);
\fill [color=blue] (0,3) circle (1.5mm);
\fill [color=blue] (0,9) circle (1.5mm);
\fill[fill=red!25] (5,3) rectangle (2,9);
\fill [color=red] (2,3) circle (1.5mm);
\fill [color=red] (4,3) circle (1.5mm);
\draw[color=black] (4,2.5) node[scale=0.8]{$(s,c)$};
\draw[color=black] (2,2.5) node[scale=0.8]{$(a,c)$};
\draw[color=black,dashed,line width=0.3mm,samples=100,domain=0:9.9,smooth] 
  plot (\x, {(9.90^0.5 - \x^0.5)^2});
  \draw[color=black,dashed,line width=0.3mm,samples=100,domain=0.5:9.9,smooth] 
  plot (\x, {(15.74^0.5 - \x^0.5)^2});
     \draw[color=red,line width=0.5mm,samples=100,domain=2.0:5.0,smooth] 
  plot (\x, {(15.74^0.5 - \x^0.5)^2});
   \draw[color=black,dashed,line width=0.3mm,samples=100,domain=1.0:9.9,smooth] 
  plot (\x, {(19.48^0.5 - \x^0.5)^2});
   \draw[color=red,line width=0.5mm,samples=100,domain=2.0:5.0,smooth] 
  plot (\x, {(19.48^0.5 - \x^0.5)^2});
  \draw[color=black,line width=0.3mm,samples=100,domain=2.0:3.0,smooth] 
  plot (\x, {(12^0.5 - \x^0.5)^2});
    \draw[color=black,line width=0.3mm,samples=100,domain=2.0:5.0,smooth] 
  plot (\x, {(17.94^0.5 - \x^0.5)^2});
      \draw[color=black,dashed,line width=0.3mm,samples=100,domain=5.0: {(17.94^0.5 - 3^0.5)^2},smooth] 
  plot (\x, {(17.94^0.5 - \x^0.5)^2});
  \draw[color=black] (4,{(17.94^0.5-4^0.5)^2-0.5}) node[scale=0.8]{$(s,t)$};
\draw[color=black] (7,2.5) node[scale=0.8]{$(l,c)$};
\fill [color=red] (4,{(17.94^0.5-4^0.5)^2})  circle (1.5mm);
\fill [color=red] ({(17.94^0.5 - 3^0.5)^2},3) circle (1.5mm);
 \fill [color=red] (2,3) circle (1.5mm);
\draw[color=black] (2,-0.5) node[scale=0.8] {$a$};
\draw[color=black] (5,-0.5) node[scale=0.8] {$b$};
\draw[color=black] (-0.5,3) node[scale=0.8] {$c$};
\draw[color=black] (-0.5,9) node[scale=0.8] {$d$};
\draw[color=black] (2.5,1) node[scale=0.8]{$\ell^{1/2}_c(a,c)$};
\draw[color=black] (6,1.5) node[scale=0.8] {$\ell^{1/2}_c(b,c)$};
\draw[color=black] (10.5,2.5) node[scale=0.8] {$\ell^{1/2}_c(a,d)$};
\end{tikzpicture}
\hspace{0.3cm}
\begin{tikzpicture}[line cap=round,line join=round,,x=1.0cm,y=1.0cm,scale=0.5]
\draw[->,color=black,dashed] (0,0) -- (11.5,0);
\draw[->,color=black,dashed] (0,0) -- (0,11.5);
\draw[color=red,line width=0.8mm] (2,0) -- (5,0);
\draw[color=red,line width=0.8mm] (0,3) -- (0,9);
\fill [color=blue] (2,0) circle (1.5mm);
\fill [color=blue] (5,0) circle (1.5mm);
\fill [color=blue] (0,3) circle (1.5mm);
\fill [color=blue] (0,9) circle (1.5mm);

\fill[fill=red!25] (5,3) rectangle (2,9);
\draw[color=black,dashed,line width=0.3mm] (0,5) -- (5,0);
\draw[color=black,dashed,line width=0.3mm] (0,8) -- (8,0);
\draw[color=black,dashed,line width=0.3mm] (0,11)--(11,0);
\draw[color=red, line width=0.5mm] (2,6)--(5,3);
\draw[color=red, line width=0.5mm] (2,9) -- (5,6);
\draw[color=black] (4,5.5) node[scale=0.8]{$(s,t)$};
\draw[color=black] (4,2.5) node[scale=0.8]{$(s,c)$};
\draw[color=black] (7,2.5) node[scale=0.8]{$(l,c)$};
\draw[color=black] (2,2.5) node[scale=0.8]{$(a,c)$};
\draw[color=black,line width=0.3mm] (2,4)--(3,3);
\draw[color=black,line width=0.3mm](2,8)--(5,5);
\draw[color=black,dashed,line width=0.3mm] (5,5)--(7,3);
\draw[color=black] (2,-0.5) node[scale=0.8] {$a$};
\draw[color=black] (5,-0.5) node[scale=0.8] {$b$};
\draw[color=black] (-0.5,3) node[scale=0.8] {$c$};
\draw[color=black] (-0.5,9) node[scale=0.8] {$d$};
\draw[color=black] (3,1) node[scale=0.8] {$\ell^1_c(a,c)$};
\draw[color=black] (6,1) node[scale=0.8] {$\ell^1_c(b,c)$};
\draw[color=black] (9,4) node[scale=0.8] {$\ell^1_c(a,d)$};
\fill [color=red] (2,3) circle (1.5mm);
\fill [color=red] (4,6) circle (1.5mm);
\fill [color=red] (7,3) circle (1.5mm);
\fill [color=red] (4,3) circle (1.5mm);
\end{tikzpicture}
\hspace{0.3cm}
\begin{tikzpicture}[line cap=round,line join=round,,x=1.0cm,y=1.0cm,scale=0.5]
\draw[->,color=black,dashed] (0,0) -- (11.5,0);
\draw[->,color=black,dashed] (0,0) -- (0,11.5);
\draw[color=red,line width=0.8mm] (2,0) -- (5,0);
\draw[color=red,line width=0.8mm] (0,3) -- (0,9);
\fill [color=blue] (2,0) circle (1.5mm);
\fill [color=blue] (5,0) circle (1.5mm);
\fill [color=blue] (0,3) circle (1.5mm);
\fill [color=blue] (0,9) circle (1.5mm);
\fill[fill=red!25] (5,3) rectangle (2,9);
\draw[dashed, thick] ({sqrt(13)},0) arc[start angle=0, end angle=90, radius={sqrt(13)}];
\draw[dashed, thick] ({sqrt(34)},0) arc[start angle=0, end angle=90, radius={sqrt(34)}];
\draw[dashed, thick] ({sqrt(85)},0) arc[start angle=0, end angle=90, radius={sqrt(85)}];

\pgfmathsetmacro{\startangleA}{atan2(3,5)}
\pgfmathsetmacro{\endangleA}{atan2({sqrt(30)},2)}
\draw[color=red, line width=0.5mm] (5,3) arc[start angle=\startangleA, end angle=\endangleA, radius={sqrt(34)}];
\pgfmathsetmacro{\startangleB}{atan2(5,5)}
\pgfmathsetmacro{\endangleB}{atan2({sqrt(46)},2)}
\draw[color=black, line width=0.3mm] (5,5) arc[start angle=\startangleB, end angle=\endangleB, radius={sqrt(50)}];
\pgfmathsetmacro{\startangleC}{atan2(3,{sqrt(41)})}
\pgfmathsetmacro{\endangleC}{atan2(5,5)}
\draw[color=black, dashed,line width=0.3mm] ({sqrt(41)},3) arc[start angle=\startangleC, end angle=\endangleC, radius={sqrt(46)}];
\pgfmathsetmacro{\startangleD}{atan2({sqrt(60)},5)}
\pgfmathsetmacro{\endangleD}{atan2(9,2}
\draw[color=red, line width=0.5mm] (5,{sqrt(60)}) arc[start angle=\startangleD, end angle=\endangleD, radius={sqrt(85)}];
\pgfmathsetmacro{\startangleE}{atan2(3,3)}
\pgfmathsetmacro{\endangleE}{atan2({sqrt(14)},2}
\draw[color=black, line width=0.3mm] (3,3) arc[start angle=\startangleE, end angle=\endangleE, radius={sqrt(18)}];
\draw[color=black] (4,5.5) node[scale=0.8]{$(s,t)$};
\draw[color=black] (4,2.5) node[scale=0.8]{$(s,c)$};
\draw[color=black] (7,2.5) node[scale=0.8]{$(l,c)$};
\draw[color=black] (2,2.5) node[scale=0.8]{$(a,c)$};
\draw[color=black] (2,-0.5) node[scale=0.8] {$a$};
\draw[color=black] (5,-0.5) node[scale=0.8] {$b$};
\draw[color=black] (-0.5,3) node[scale=0.8] {$c$};
\draw[color=black] (-0.5,9) node[scale=0.8] {$d$};
\draw[color=black] (2.1,1) node[scale=0.8]{$\ell^2_c(a,c)$};
\draw[color=black] (7,1) node[scale=0.8] {$\ell^2_c(b,c)$};
\draw[color=black] (8.5,6) node[scale=0.8] {$\ell^2_c(a,d)$};
\fill [color=red] (2,3) circle (1.5mm);
\fill [color=red] (4,{sqrt(34)}) circle (1.5mm);
\fill [color=red] ({sqrt(41)},3) circle (1.5mm);
\fill [color=red] (4,3) circle (1.5mm);
\end{tikzpicture}
\hspace{0.3cm}
\begin{tikzpicture}[line cap=round,line join=round,,x=1.0cm,y=1.0cm,scale=0.5]
\draw[->,color=black,dashed] (0,0) -- (11.5,0);
\draw[->,color=black,dashed] (0,0) -- (0,11.5);
\draw[color=red,line width=0.8mm] (2,0) -- (5,0);
\draw[color=red,line width=0.8mm] (0,3) -- (0,9);
\fill [color=blue] (2,0) circle (1.5mm);
\fill [color=blue] (5,0) circle (1.5mm);
\fill [color=blue] (0,3) circle (1.5mm);
\fill [color=blue] (0,9) circle (1.5mm);
\fill[fill=red!25] (5,3) rectangle (2,9);
\draw[color=black,dashed,line width=0.3mm] (0,3) -- (3,3);
\draw[color=black,dashed,line width=0.3mm] (3,3) -- (3,0);
\draw[color=black,dashed,line width=0.3mm] (0,5) -- (5,5);
\draw[color=black,dashed,line width=0.3mm] (5,5) -- (5,0);
\draw[color=black,dashed,line width=0.3mm] (0,9) -- (9,9);
\draw[color=black,dashed,line width=0.3mm] (9,9) -- (9,0);
\draw[color=red, line width=0.5mm] (2,3)--(3,3);
\draw[color=red, line width=0.5mm] (2,5)--(5,5);
\draw[color=red, dashed, line width=0.5mm] (5,5)--(5,3);
\draw[color=red, dashed, line width=0.5mm] (2,9) -- (5,9);
\draw[color=black] (4,6.5) node[scale=0.8]{$(s,t)$};
\draw[color=black] (4,2.5) node[scale=0.8]{$(s,c)$};
\draw[color=black] (7,2.5) node[scale=0.8]{$(l,c)$};
\draw[color=black] (2,2.5) node[scale=0.8]{$(a,c)$};
\draw[color=black,line width=0.3mm] (2,3.5)--(3.5,3.5);
\draw[color=black,line width=0.3mm] (3.5,3.5)--(3.5,3);
\draw[color=black,line width=0.3mm] (2,7)--(5,7);
\draw[color=black,dashed,line width=0.3mm] (5,7)--(7,7);
\draw[color=black,dashed,line width=0.3mm] (7,7)--(7,3);
\draw[color=black] (2,-0.5) node[scale=0.8] {$a$};
\draw[color=black] (5,-0.5) node[scale=0.8] {$b$};
\draw[color=black] (-0.5,3) node[scale=0.8] {$c$};
\draw[color=black] (-0.5,9) node[scale=0.8] {$d$};
\draw[color=black] (1.5,1) node[scale=0.8]{$\ell^{\infty}_c(a,c)$};
\draw[color=black] (6.5,1) node[scale=0.8] {$\ell^{\infty}_c(b,c)$};
\draw[color=black] (10.5,1) node[scale=0.8] {$\ell^{\infty}_c(a,d)$};
\fill [color=red] (2,3) circle (1.5mm);
\fill [color=red] (4,7) circle (1.5mm);
\fill [color=red] (7,3) circle (1.5mm);
\fill [color=red] (4,3) circle (1.5mm);
\end{tikzpicture}
\caption {\rmfamily Illustration of the $\cpnorm$-tensor product of interval modules: $\mathbf{k}[a,b)\ptensor \mathbf{k}[c,d) =\mathbf{k}[\cpnorm(a,c),\min(\cpnorm(a,d),\cpnorm(b,c))$. Only the cases $p=\frac{1}{2}$ (top left), $p=1$ (top right), $p=2$ (bottom left), $p=\infty$ (bottom right), and $\cpnorm(b,c)\le \cpnorm(a,d)$ are shown. The dashed lines represent the level sets of the maps $\cpnorm:\mathbb{R}^2_{+}\to \mathbb{R}_{+}$. For example, the curve $\ell^2_c(b,c)$ is the set of all points $(x,y)\in \mathbb{R}^2_{+}$ such that $\ell^2_c(x,y)=\ell^2_c(b,c)$. The red rectangle are the points in $\mathbb{R}^2_+$ where persistence module $\cat{k}[a,b)\boxtimes \cat{k}[c,d)$ is nonzero.}
\label{fig:convolution_of_interval_modules}
\end{figure}
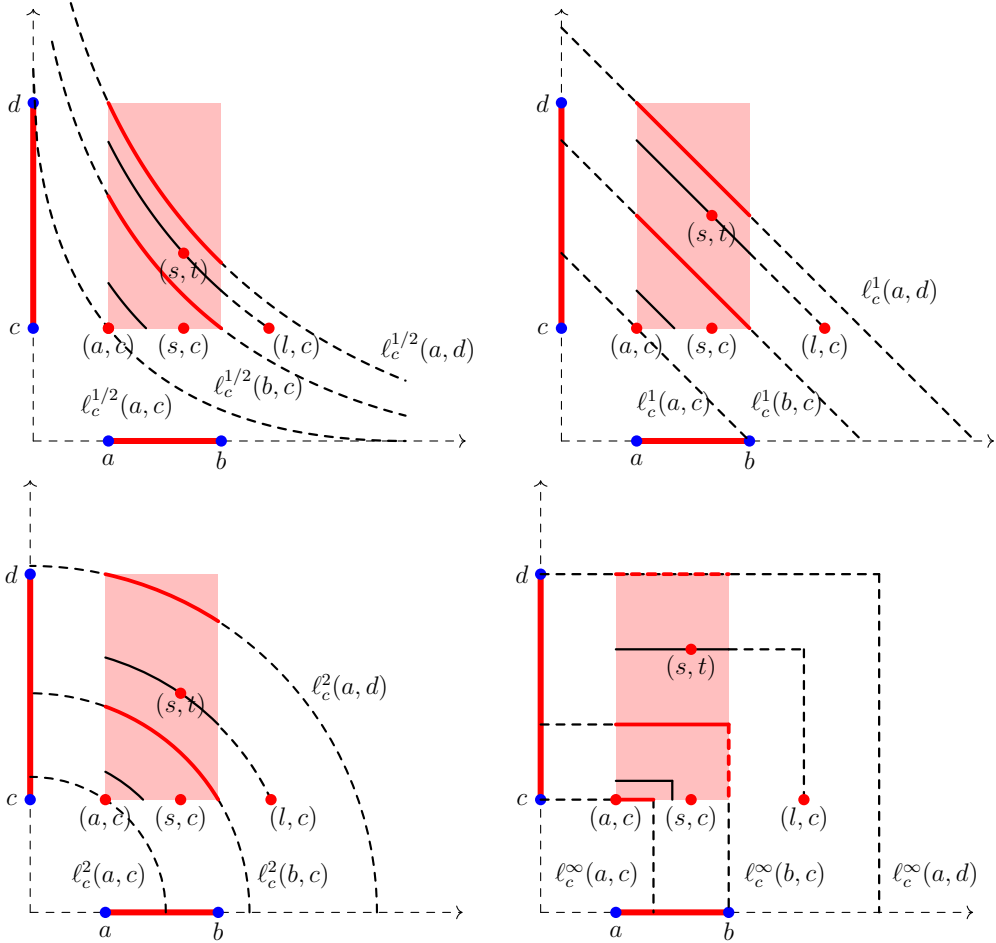

\begin{example}
\label{example:multiparameter_tensor_of_rect_modules}
Let $M,N:\cat{R}^n_+\to \cat{Vect_k}$ be interval modules of the form $M=\cat{k}[A]$ and $N=\cat{k}[B]$, where $A$ and $B$ are the hyper-rectangles
 \[A=[a_1,b_1)\times \dots \times [a_n,b_n), \textnormal{ and } B=[c_1,d_1)\times \dots \times [c_n,d_n).\] 
Consider the hyper-rectangle 
\[C=[||(a_1,c_1)||_p,\min(||(b_1,c_1)||_p,||(a_1,d_1)||_p)\times\dots\times [||(a_n,c_n)||_p,\min(||(b_n,c_n)||_p,||(a_n,d_n)||_p).\]
Then, similar to the one-parameter case, we find that  $M\ptensor N\approx \mathbf{k}[C]$.
\end{example}

\subsection{The p-quasinorm internal hom}
Here we define the adjoint to $\ptensor$,  $\scHom^{\cpnorm}$, for $p\in (0,\infty]$. We also show this functor applied to one-parameter interval modules returns an interval module. Recall \cref{def:convolution_hom}.
\begin{definition}
	Let $M,N:\cat{R}^n_+\to \cat{Vect_k}$ be persistence modules and $p\in (0,\infty]$. The \emph{componentwise }$p$-\emph{quasinorm internal hom} of simply $\cpnorm$-\emph{internal hom} of $M$ and $N$ is $\scHom^{\cpnorm}(M,N)$.
\end{definition}

We immediately have the following.

\begin{corollary}
    For any persistence modules $M,N:\cat{R}^n_+\to \cat{Vect_k}$, there exist natural isomorphisms
    \[\scHom^{\ell^{\infty}_c}(M,N)\approx \scHom_{\cat{R}^n_+}(M,N).\]
\end{corollary}

\begin{proof}
    The result follows from \cref{lemma:sh_tensor_as_a_convolution,theorem:convolution_adjunction,theorem:hom_tensor_adjunction} since adjoints are unique up to natural isomorphisms.
\end{proof}

Therefore, by \cref{example:sheaf_hom_of_interval_modules} we also have the following.

\begin{corollary}
\label{corollary:inf_sc_hom}
Let $M,N:\cat{R}_+\to \cat{Vect_k}$ be the interval modules $M=\cat{k}[a,b)$ and $N=\cat{k}[c,d)$. Then
    \[\scHom^{\ell^{\infty}_c}(M,N)=\begin{cases}
0 & \text{if } a < b\le c < d\\
0 & \text{if } a < c\le b < d\\
\cat{k}[c,d) & \text{if } a < c < d\le b\\
0 & \text{if } c\le a < b < d\\
\cat{k}[0,d) & \text{if } c\le a < d\le b\\
0 & \text{if } c < d\le a < b
\end{cases}.\]
\end{corollary}

Thus for $p=\infty$, $\scHom^{\ell^p_c}(-,-)$ sends interval modules to interval modules (at least in the one-parameter case). For all other $p\in (0,\infty)$ the same is true. 

\begin{proposition}
\label{proposition:internal_hom_formulas}
    Let $p\in (0,\infty)$. Then 
    \begin{enumerate}[left=0pt]
        \item $\scHom^{\cpnorm}(\cat{k}[a,b), \cat{k}[c,d))\approx \cat{k}[(\max(0,c^p-a^p,d^p-b^p))^{\frac{1}{p}},(\max(0,d^p-a^p))^{\frac{1}{p}})$.
        \item $\scHom^{\cpnorm}(\cat{k}[a,\infty), \cat{k}[c,d)\approx \cat{k}[(\max(0,c^p-a^p))^{\frac{1}{p}},(\max(0,d^p-a^p))^{\frac{1}{p}})$.
        \item $\scHom^{\cpnorm}(\cat{k}[a,\infty), \cat{k}[c,\infty))\approx \cat{k}[(\max(0,c^p-a^p))^{\frac{1}{p}},\infty)$.
    \end{enumerate}
\end{proposition}

\begin{proof}
    We prove the isomorphism in (1). The other isomorphisms follow by similar arguments. Let $M=\cat{k}[a,b)$ and $N=\cat{k}[c,d)$. Assume that $0\le c^p-a^p\le d^p-b^p$. \cref{fig:hom_of_interval_modules} illustrates this case for $p\in \{\frac{1}{2},1,2,\infty\}$ and $0<c^p-a^p$. Observe that
     \[\pi_1^{-1}([a,b))=\{(s,t)\in \mathbb{R}^2_{+}\, |\, a\le s<b\},\] 
    and
     \[(\cpnorm)^{-1}([c,d))=\{(s,t)\in\mathbb{R}^{2}_{+}\,|\, c\le \cpnorm(s,t)< d\}\] 
     are both intervals in the poset $\mathbb{R}^2_{+}$. Therefore, $\pi_1^{-1}M$ and $(\cpnorm)^{-1}N$ are interval modules as well (as functors on $\cat{R}^2_{+}$). For any $(s,t)\in \mathbb{R}^2_{+}$ let $U_{(s,t)}=\{(s',t') \in \mathbb{R}^2_{+}\,|\, s\le s', t\le t'\}$ be the principal up-set at $(s,t)$. Then, the restrictions $\pi_1^{-1}M|_{\cat{U}_{(s,t)}}$ and $(\cpnorm)^{-1}N|_{\cat{U}_{(s,t)}}$ are also interval modules. For any $x\in \mathbb{R}_+$ we have
    \begin{align*}
        \scHom^{\cpnorm}(M,N)_x&\od \pi_{2*}\scHom_{\cat{k^{R^2_+}}}(\pi_1^{-1}M,(\cpnorm)^{-1}N)_x& (\textnormal{\cref{def:convolution_hom}})\\
        &\od \underset{x\le \pi_2(s,t)}{\lim}\scHom_{\cat{k^{R^2_+}}}(\pi_1^{-1}M,(\cpnorm)^{-1}N)_{(s,t)}& (\textnormal{\cref{def:direct_image_1}})\\
        &\approx \scHom_{\cat{k^{R^2_+}}}(\pi_1^{-1}M,(\cpnorm)^{-1}N)_{(0,x)}& (\textnormal{\cref{lemma:colimits_and_limits_over_diagrams_with_0_and_1}})\\
        &\od \Hom_{\cat{k}^{\cat{U}_{(0,x)}}}(\pi_1^{-1}M|_{\cat{U}_{(0,x)}},(\cpnorm)^{-1}N|_{\cat{U}_{(0,x)}})& (\textnormal{\cref{def:sheaf_hom}})
    \end{align*}
    Thus, to find $\scHom^{\cpnorm}(M,N)_x$ we need to compute the set of natural transformations between the restrictions $\pi_1^{-1}M|_{\cat{U}_{(0,x)}}$ and $(\cpnorm)^{-1}N|_{\cat{U}_{(0,x)}}$.
    There are several cases to consider:
    
    \begin{enumerate}[left=0pt]
        \item Suppose $x< (d^p-b^p)^{\frac{1}{p}}$. Then, when restricted to $U_{(0,x)}$ the intervals $\pi_1^{-1}([a,b))$ and $(\cpnorm)^{-1}([c,d))$ do not overlap nicely. There is a part of $(\cpnorm)^{-1}([c,d))$ that lies to the ``right" of the intersection $\pi_1^{-1}([a,b))\cap (\cpnorm)^{-1}([c,d))$ (see \cref{fig:hom_of_interval_modules}). By \cref{lemma:morphisms_between_intervals}, it follows that $\Hom_{\cat{k}^{\cat{U}_{(0,x)}}}(\pi_1^{-1}M|_{\cat{U}_{(0,x)}},(\cpnorm)^{-1}N|_{\cat{U}_{(0,x)}})=0$.
        \item Suppose $(d^p-b^p)^{\frac{1}{p}}\le x<(d^p-a^p)^{\frac{1}{p}}$. Then, when restricted to $U_{(0,x)}$ the intervals $\pi_1^{-1}([a,b))$ and $(\cpnorm)^{-1}([c,d))$ overlap nicely (see \cref{fig:hom_of_interval_modules}). By \cref{lemma:morphisms_between_intervals}, it follows that $\Hom_{\cat{k}^{\cat{U}_{(0,x)}}}(\pi_1^{-1}M|_{\cat{U}_{(0,x)}},(\cpnorm)^{-1}N|_{\cat{U}_{(0,x)}})\approx\cat{k}$.
        \item Suppose $(d^p-a^p)^{\frac{1}{p}}\le x$. Then, when restricted to $U_{(0,x)}$ the intervals $\pi_1^{-1}([a,b))$ and $(\cpnorm)^{-1}([c,d))$ do not intersect (see \cref{fig:hom_of_interval_modules}). By \cref{lemma:nonzero_map_of_interval_modules} it follows that $\Hom_{\cat{k}^{\cat{U}_{(0,x)}}}(\pi_1^{-1}M|_{\cat{U}_{(0,x)}},(\cpnorm)^{-1}N|_{\cat{U}_{(0,x)}})=0$.
    \end{enumerate}

    On the other hand, if we had assumed $0\le d^p-b^p\le c^p-a^p$ by using similar arguments and considering \cref{fig:hom_of_interval_modules2}, which illustrates this case, we would arrive at the following cases:

    \begin{enumerate}[left=0pt]
        \item Suppose $x< (c^p-a^p)^{\frac{1}{p}}$. Then, when restricted to $U_{(0,x)}$ the intervals $\pi_1^{-1}([a,b))$ and $(\cpnorm)^{-1}([c,d))$ do not overlap nicely. There is a part of $(\pi_1)^{-1}([a,b))$ that lies to the ``left" of the intersection $\pi_1^{-1}([a,b))\cap (\cpnorm)^{-1}([c,d))$ (see \cref{fig:hom_of_interval_modules2}). By \cref{lemma:morphisms_between_intervals}, it follows that $\Hom_{\cat{k}^{\cat{U}_{(0,x)}}}(\pi_1^{-1}M|_{\cat{U}_{(0,x)}},(\cpnorm)^{-1}N|_{\cat{U}_{(0,x)}})=0$.
        \item Suppose $(c^p-a^p)^{\frac{1}{p}}\le x<(d^p-a^p)^{\frac{1}{p}}$. Then, when restricted to $U_{(0,x)}$ the intervals $\pi_1^{-1}([a,b))$ and $(\cpnorm)^{-1}([c,d))$ overlap nicely (see \cref{fig:hom_of_interval_modules2}). By \cref{lemma:morphisms_between_intervals}, it follows that $\Hom_{\cat{k}^{\cat{U}_{(0,x)}}}(\pi_1^{-1}M|_{\cat{U}_{(0,x)}},(\cpnorm)^{-1}N|_{\cat{U}_{(0,x)}})\approx\cat{k}$.
        \item Suppose $(d^p-a^p)^{\frac{1}{p}}\le x$. Then, when restricted to $U_{(0,x)}$ the intervals $\pi_1^{-1}([a,b))$ and $(\cpnorm)^{-1}([c,d))$ do not intersect (see \cref{fig:hom_of_interval_modules2}). By \cref{lemma:nonzero_map_of_interval_modules} it follows that $\Hom_{\cat{k}^{\cat{U}_{(0,x)}}}(\pi_1^{-1}M|_{\cat{U}_{(0,x)}},(\cpnorm)^{-1}N|_{\cat{U}_{(0,x)}})=0$.
    \end{enumerate}

    Now, combining these cases we would get the formula \[\scHom^{\cpnorm}(\cat{k}[a,b), \cat{k}[c,d))\approx \cat{k}[(\max(c^p-a^p,d^p-b^p))^{\frac{1}{p}},(d^p-a^p)^{\frac{1}{p}}).\] 
    However, since our persistence modules are indexed by $\mathbb{R}_+$ the formula also has to include the potential cases where some of the values $c^p-a^p$, $d^p-b^p$ and $d^p-a^p$ could be negative. Using similar arguments and \cref{lemma:morphisms_between_intervals}, we deduce the isomorphism 
    \[\scHom^{\cpnorm}(\cat{k}[a,b), \cat{k}[c,d))\approx \cat{k}[(\max(0,c^p-a^p,d^p-b^p))^{\frac{1}{p}},(\max(0,d^p-a^p))^{\frac{1}{p}}).\] For example, if $d^p-a^p<0$, then also $d^p-b^p<0$ and $c^p-a^p<0$ and in particular the intervals $\pi_1^{-1}([a,b)])$ and $(\cpnorm)^{-1}([c,d))$ do not intersect in $\mathbb{R}^2_+$ and thus $\scHom^{\cpnorm}(\cat{k}[a,b), \cat{k}[c,d))\approx\cat{k}[0,0)=0$.
\end{proof}

\begin{figure}[H]
\centering
\begin{tikzpicture}[line cap=round,line join=round,,x=1.0cm,y=1.0cm,scale=0.5]
\draw[->,color=black,dashed] (0,0) -- (11.5,0);
\draw[->,color=black,dashed] (0,0) -- (0,11.5);
\fill[fill=red!25] (5,0) rectangle (2,11.5);
\fill[cyan!35,, fill opacity=0.4] 
  (0,3) -- (0,9) --
  plot[samples=100,domain=0:9,smooth] (\x, {(9^0.5 - \x^0.5)^2}) --
  (9,0)--(3,0)--
  plot[samples=100,domain=0:3,smooth] (\x, {(3^0.5 - \x^0.5)^2}) --
  cycle;
\draw[color=red,line width=0.8mm] (2,0) -- (5,0);
\draw[color=cyan,line width=0.8mm] (0,3) -- (0,9);
\fill [color=blue] (2,0) circle (1.5mm);
\fill [color=blue] (5,0) circle (1.5mm);
\fill [color=blue] (0,3) circle (1.5mm);
\fill [color=blue] (0,9) circle (1.5mm);
\draw[color=black,line width=0.3mm] (0,{(9^0.5 - 2^0.5)^2}) -- (11.5,{(9^0.5 - 2^0.5)^2});
\draw[color=black,line width=0.3mm] (0,{(3^0.5 - 2^0.5)^2}) -- (11.5,{(3^0.5 - 2^0.5)^2});
\draw[color=black,line width=0.3mm] (0,{(9^0.5 - 5^0.5)^2}) -- (11.5,{(9^0.5 - 5^0.5)^2});
\draw[color=black] (10,3) node[scale=0.8] {$(d^{1/2}-a^{1/2})^2$};
\draw[color=black] (10,1) node[scale=0.8] {$(d^{1/2}-b^{1/2})^{2}$};
\draw[color=black] (10,-0.5) node[scale=0.8] {$(c^{1/2}-a^{1/2})^2$};
  \draw[color=cyan,dashed,line width=0.3mm,samples=100,domain=0:9,smooth] 
  plot (\x, {(9^0.5 - \x^0.5)^2});
  \draw[color=cyan,line width=0.3mm,samples=100,domain=0:3,smooth] 
  plot (\x, {(3^0.5 - \x^0.5)^2});
\draw[color=black] (2,-0.5) node[scale=0.8] {$a$};
\draw[color=black] (5,-0.5) node[scale=0.8] {$b$};
\draw[color=black] (-0.5,3) node[scale=0.8] {$c$};
\draw[color=black] (-0.5,9) node[scale=0.8] {$d$};
\end{tikzpicture}
\hspace{0.3cm}
\begin{tikzpicture}[line cap=round,line join=round,,x=1.0cm,y=1.0cm,scale=0.5]
\draw[->,color=black,dashed] (0,0) -- (11.5,0);
\draw[->,color=black,dashed] (0,0) -- (0,11.5);
\fill[fill=red!25] (5,0) rectangle (2,11.5);
\draw[color=red,line width=0.8mm] (2,0) -- (5,0);
\draw[color=cyan,line width=0.8mm] (0,3) -- (0,9);
\fill[cyan!25,fill opacity=0.4] (0,3) -- (3,0) -- (9,0) -- (0,9) -- cycle;  
\fill [color=blue] (2,0) circle (1.5mm);
\fill [color=blue] (5,0) circle (1.5mm);
\fill [color=blue] (0,3) circle (1.5mm);
\fill [color=blue] (0,9) circle (1.5mm);
\draw[color=black,line width=0.3mm] (0,7) -- (11.5,7);
\draw[color=black,line width=0.3mm] (0,1) -- (11.5,1);
\draw[color=black,line width=0.3mm] (0,4) -- (11.5,4);
\draw[color=black] (10,7.5) node[scale=0.8] {$d-a$};
\draw[color=black] (10,4.5) node[scale=0.8] {$d-b$};
\draw[color=black] (10,1.5) node[scale=0.8] {$c-a$};
\draw[color=cyan,dashed,line width=0.3mm] (9,0) -- (0,9);
\draw[color=cyan,line width=0.3mm] (3,0) -- (0,3);
\draw[color=black] (2,-0.5) node[scale=0.8] {$a$};
\draw[color=black] (5,-0.5) node[scale=0.8] {$b$};
\draw[color=black] (-0.5,3) node[scale=0.8] {$c$};
\draw[color=black] (-0.5,9) node[scale=0.8] {$d$};
\end{tikzpicture}

\begin{tikzpicture}[line cap=round,line join=round,,x=1.0cm,y=1.0cm,scale=0.5]
\draw[->,color=black,dashed] (0,0) -- (11.5,0);
\draw[->,color=black,dashed] (0,0) -- (0,11.5);
\fill[fill=red!25] (5,0) rectangle (2,11.5);
\draw[color=red,line width=0.8mm] (2,0) -- (5,0);
\draw[color=cyan,line width=0.8mm] (0,3) -- (0,9);
\fill [color=blue] (2,0) circle (1.5mm);
\fill [color=blue] (5,0) circle (1.5mm);
\fill [color=blue] (0,3) circle (1.5mm);
\fill [color=blue] (0,9) circle (1.5mm);
\fill[fill=red!25] (5,3) rectangle (2,9);
\draw[color=black,line width=0.3mm] (0,{sqrt(5)}) -- (11.5,{sqrt(5)});
\draw[color=black,line width=0.3mm] (0,{sqrt(77)}) -- (11.5,{sqrt(77)});
\draw[color=black,line width=0.3mm] (0,{sqrt(56)}) -- (11.5,{sqrt(56)});
\draw[color=black] (10,{sqrt(5)+0.5}) node[scale=0.8] {$(c^2-a^2)^{\frac{1}{2}}$};
\draw[color=black] (10,{sqrt(77)+0.5}) node[scale=0.8] {$(d^2-a^2)^{\frac{1}{2}}$};
\draw[color=black] (10,{sqrt(56)+0.5}) node[scale=0.8] {$(d^2-b^2)^{\frac{1}{2}}$};
\path[fill=cyan!35, fill opacity=0.4, draw=none]
        (9,0)                    
        arc[start angle=0, end angle=90, radius=9]  
        -- (0,3)                                    
        arc[start angle=90, end angle=0, radius=3]  
        -- cycle;  
\draw[color=cyan, thick] (3,0) arc[start angle=0, end angle=90, radius=3];
\draw[color=cyan, dashed, thick] (9,0) arc[start angle=0, end angle=90, radius=9];
\draw[color=black] (2,-0.5) node[scale=0.8] {$a$};
\draw[color=black] (5,-0.5) node[scale=0.8] {$b$};
\draw[color=black] (-0.5,3) node[scale=0.8] {$c$};
\draw[color=black] (-0.5,9) node[scale=0.8] {$d$};
\end{tikzpicture}
\hspace{0.3cm}
\begin{tikzpicture}[line cap=round,line join=round,,x=1.0cm,y=1.0cm,scale=0.5]
\fill[fill=red!25] (5,0) rectangle (2,11.5);

\draw[->,color=black,dashed] (0,0) -- (11.5,0);
\draw[->,color=black,dashed] (0,0) -- (0,11.5);
\draw[color=red,line width=0.8mm] (2,0) -- (5,0);
\draw[color=cyan,line width=0.8mm] (0,3) -- (0,9);
\fill [color=blue] (2,0) circle (1.5mm);
\fill [color=blue] (5,0) circle (1.5mm);
\fill [color=blue] (0,3) circle (1.5mm);
\fill [color=blue] (0,9) circle (1.5mm);
\draw[color=cyan,line width=0.3mm] (0,3) -- (3,3);
\draw[color=cyan,line width=0.3mm] (3,3) -- (3,0);
\draw[color=cyan,line width=0.3mm] (0,9) -- (9,9);
\draw[color=cyan,dashed,line width=0.3mm] (9,9) -- (9,0);
\fill[cyan!25,fill opacity=0.4] (0,3) -- (3,3) -- (3,0) -- (9,0) -- (9,9) --(0,9) -- cycle;
\draw[color=black,line width=0.3mm] (0,3) -- (11.5,3);
\draw[color=black,line width=0.3mm] (0,5) -- (11.5,5);
\draw[color=black,line width=0.3mm] (0,9) -- (11.5,9);
\draw[color=black] (10,3.5) node[scale=0.8] {$c$};
\draw[color=black] (10,5.5) node[scale=0.8] {$b$};
\draw[color=black] (10,9.5) node[scale=0.8] {$d$};
\draw[color=black] (2,-0.5) node[scale=0.8] {$a$};
\draw[color=black] (5,-0.5) node[scale=0.8] {$b$};
\draw[color=black] (-0.5,3) node[scale=0.8] {$c$};
\draw[color=black] (-0.5,9) node[scale=0.8] {$d$};
\end{tikzpicture}
\caption {\rmfamily Illustration of the $\scHom^{\cpnorm}$ functor applied to two interval modules: $M=\mathbf{k}[a,b)$ and $N=\mathbf{k}[c,d)$. Only the cases $p=\frac{1}{2}$ (top left), $p=1$ (top right), $p=2$ (bottom left) and $p=\infty$ (bottom right) and $0\le c^p- a^p\le d^p-b^p$ are shown. The red regions are points where the module $\pi_{1}^{-1}M:\mathbb{R}^2_{+}\to\cat{Vect_k}$ is nonzero. The cyan regions are the points where the module $(\cpnorm)^{-1}N:\mathbb{R}^2_{+}\to\cat{Vect_k}$ is nonzero.}
\label{fig:hom_of_interval_modules}
\end{figure}
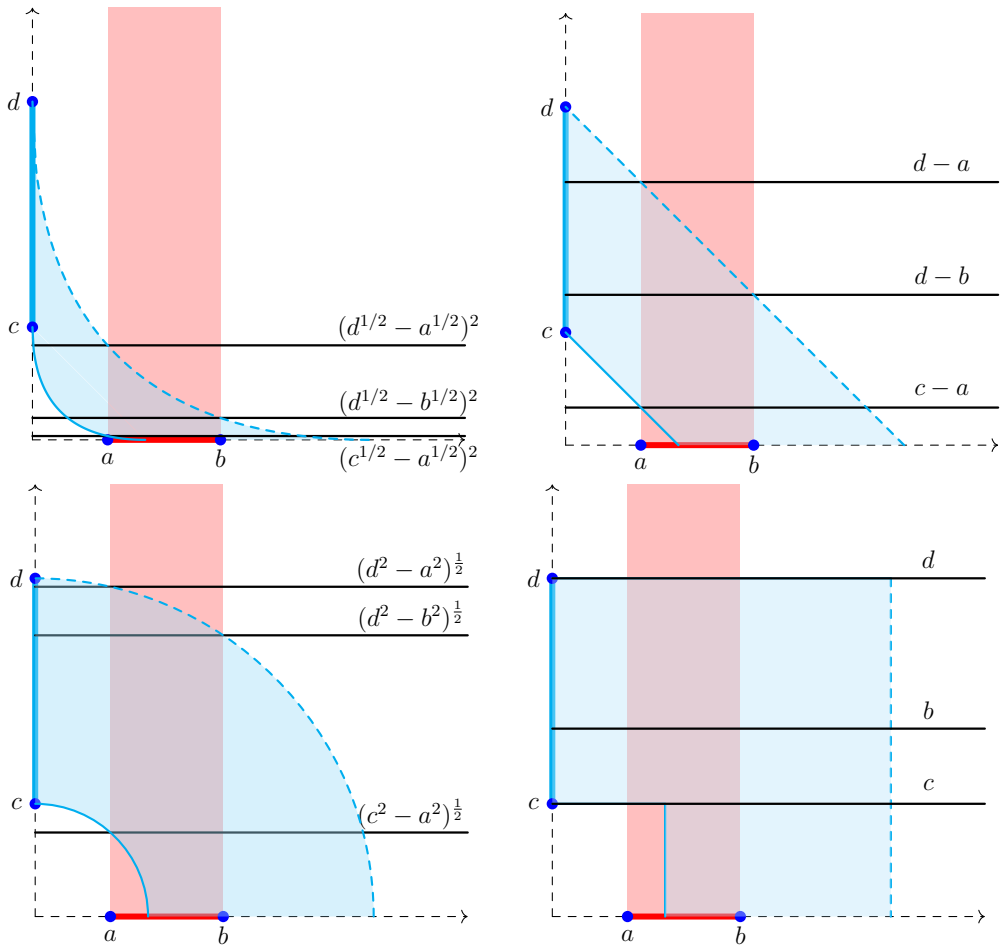

\begin{figure}[H]
\centering
\begin{tikzpicture}[line cap=round,line join=round,,x=1.0cm,y=1.0cm,scale=0.5]
\draw[->,color=black,dashed] (0,0) -- (11.5,0);
\draw[->,color=black,dashed] (0,0) -- (0,11.5);
\fill[fill=red!25] (5,0) rectangle (2,11.5);
\draw[color=red,line width=0.8mm] (2,0) -- (5,0);
\draw[color=cyan,line width=0.8mm] (0,6) -- (0,8);
\fill[cyan!35,, fill opacity=0.4] 
  (0,6) -- (0,8) --
  plot[samples=100,domain=0:8,smooth] (\x, {(8^0.5 - \x^0.5)^2}) --
  (8,0)--(5,0)--
  plot[samples=100,domain=5:0,smooth] (\x, {(6^0.5 - \x^0.5)^2}) --
  cycle;
 \draw[color=cyan,dashed,line width=0.3mm,samples=100,domain=0:8,smooth] 
  plot (\x, {(8^0.5 - \x^0.5)^2});
  \draw[color=cyan,line width=0.3mm,samples=100,domain=0:6,smooth] 
  plot (\x, {(6^0.5 - \x^0.5)^2});
\fill [color=blue] (2,0) circle (1.5mm);
\fill [color=blue] (5,0) circle (1.5mm);
\fill [color=blue] (0,6) circle (1.5mm);
\fill [color=blue] (0,8) circle (1.5mm);
\draw[color=black,line width=0.3mm] (0,{(8^0.5 - 5^0.5)^2}) -- (11.5,{(8^0.5 - 5^0.5)^2});
\draw[color=black,line width=0.3mm] (0,{(8^0.5 - 2^0.5)^2}) -- (11.5,{(8^0.5 - 2^0.5)^2});
\draw[color=black,line width=0.3mm] (0,{(6^0.5 - 2^0.5)^2}) -- (11.5,{(6^0.5 - 2^0.5)^2});
\draw[color=black] (10, {(8^0.5 - 2^0.5)^2+0.5}) node[scale=0.8] {$(d^{1/2}-a^{1/2})^2$};
\draw[color=black] (10,{(8^0.5 - 5^0.5)^2-0.75}) node[scale=0.8] {$(d^{1/2}-b^{1/2})^2$};
\draw[color=black] (10,{(6^0.5 - 2^0.5)^2+0.5}) node[scale=0.8] {$(c^{1/2}-a^{1/2})^2$};
\draw[color=black] (2,-0.5) node[scale=0.8] {$a$};
\draw[color=black] (6,-0.5) node[scale=0.8] {$b$};
\draw[color=black] (-0.5,6) node[scale=0.8] {$c$};
\draw[color=black] (-0.5,8) node[scale=0.8] {$d$};
\end{tikzpicture}
\begin{tikzpicture}[line cap=round,line join=round,,x=1.0cm,y=1.0cm,scale=0.5]
\fill[fill=red!25] (5,0) rectangle (2,11.5);
\draw[->,color=black,dashed] (0,0) -- (11.5,0);
\draw[->,color=black,dashed] (0,0) -- (0,11.5);
\draw[color=red,line width=0.8mm] (2,0) -- (5,0);
\draw[color=cyan,line width=0.8mm] (0,6) -- (0,8);
\fill[cyan!25,fill opacity=0.4] (0,6) -- (6,0) -- (8,0) -- (0,8) -- cycle;  
\fill [color=blue] (2,0) circle (1.5mm);
\fill [color=blue] (5,0) circle (1.5mm);
\fill [color=blue] (0,6) circle (1.5mm);
\fill [color=blue] (0,8) circle (1.5mm);
\draw[color=black,line width=0.3mm] (0,4) -- (11.5,4);
\draw[color=black,line width=0.3mm] (0,6) -- (11.5,6);
\draw[color=black,line width=0.3mm] (0,3) -- (11.5,3);
\draw[color=black] (10,6.5) node[scale=0.8] {$d-a$};
\draw[color=black] (10,3.5) node[scale=0.8] {$d-b$};
\draw[color=black] (10,4.5) node[scale=0.8] {$c-a$};
\draw[color=cyan,dashed,line width=0.3mm] (8,0) -- (0,8);
\draw[color=cyan,line width=0.3mm] (6,0) -- (0,6);
\draw[color=black] (2,-0.5) node[scale=0.8] {$a$};
\draw[color=black] (6,-0.5) node[scale=0.8] {$b$};
\draw[color=black] (-0.5,6) node[scale=0.8] {$c$};
\draw[color=black] (-0.5,8) node[scale=0.8] {$d$};
\end{tikzpicture}

\begin{tikzpicture}[line cap=round,line join=round,,x=1.0cm,y=1.0cm,scale=0.5]
\draw[->,color=black,dashed] (0,0) -- (11.5,0);
\draw[->,color=black,dashed] (0,0) -- (0,11.5);
\fill[fill=red!25] (6,0) rectangle (2,11.5);
\draw[color=red,line width=0.8mm] (2,0) -- (6,0);
\draw[color=cyan,line width=0.8mm] (0,6) -- (0,8);
\fill [color=blue] (2,0) circle (1.5mm);
\fill [color=blue] (6,0) circle (1.5mm);
\fill [color=blue] (0,6) circle (1.5mm);
\fill [color=blue] (0,8) circle (1.5mm);
\draw[color=black,line width=0.3mm] (0,{sqrt(32)}) -- (11.5,{sqrt(32)});
\draw[color=black,line width=0.3mm] (0,{sqrt(60)}) -- (11.5,{sqrt(60)});
\draw[color=black,line width=0.3mm] (0,{sqrt(28)}) -- (11.5,{sqrt(28)});
\draw[color=black] (10,{sqrt(32)+0.5}) node[scale=0.8] {$(c^2-a^2)^{\frac{1}{2}}$};
\draw[color=black] (10,{sqrt(60)+0.5}) node[scale=0.8] {$(d^2-a^2)^{\frac{1}{2}}$};
\draw[color=black] (10,{sqrt(28)-0.5}) node[scale=0.8] {$(d^2-b^2)^{\frac{1}{2}}$};
\path[fill=cyan!35, fill opacity=0.4, draw=none]
        (8,0)                    
        arc[start angle=0, end angle=90, radius=8]  
        -- (0,6)                                    
        arc[start angle=90, end angle=0, radius=6]  
        -- cycle;  
\draw[color=cyan, thick] (6,0) arc[start angle=0, end angle=90, radius=6];
\draw[color=cyan, dashed, thick] (8,0) arc[start angle=0, end angle=90, radius=8];
\draw[color=black] (2,-0.5) node[scale=0.8] {$a$};
\draw[color=black] (6,-0.5) node[scale=0.8] {$b$};
\draw[color=black] (-0.5,6) node[scale=0.8] {$c$};
\draw[color=black] (-0.5,8) node[scale=0.8] {$d$};
\end{tikzpicture}
\hspace{0.3cm}
\begin{tikzpicture}[line cap=round,line join=round,,x=1.0cm,y=1.0cm,scale=0.5]
\fill[fill=red!25] (5,0) rectangle (2,11.5);
\draw[->,color=black,dashed] (0,0) -- (11.5,0);
\draw[->,color=black,dashed] (0,0) -- (0,11.5);
\draw[color=red,line width=0.8mm] (2,0) -- (5,0);
\draw[color=cyan,line width=0.8mm] (0,6) -- (0,8);
\fill [color=blue] (2,0) circle (1.5mm);
\fill [color=blue] (5,0) circle (1.5mm);
\fill [color=blue] (0,6) circle (1.5mm);
\fill [color=blue] (0,8) circle (1.5mm);
\draw[color=cyan,line width=0.3mm] (0,6) -- (6,6);
\draw[color=cyan,line width=0.3mm] (6,6) -- (6,0);
\draw[color=cyan,line width=0.3mm] (0,8) -- (8,8);
\draw[color=cyan,dashed,line width=0.3mm] (8,8) -- (8,0);
\fill[cyan!25,fill opacity=0.4] (0,6) -- (6,6) -- (6,0) -- (8,0) -- (8,8) --(0,8) -- cycle;
\draw[color=black,line width=0.3mm] (0,6) -- (11.5,6);
\draw[color=black,line width=0.3mm] (0,8) -- (11.5,8);
\draw[color=black,line width=0.3mm] (0,5) -- (11.5,5);
\draw[color=black] (10,6.5) node[scale=0.8] {$c$};
\draw[color=black] (10,5.5) node[scale=0.8] {$b$};
\draw[color=black] (10,8.5) node[scale=0.8] {$d$};
\draw[color=black] (2,-0.5) node[scale=0.8] {$a$};
\draw[color=black] (5,-0.5) node[scale=0.8] {$b$};
\draw[color=black] (-0.5,6) node[scale=0.8] {$c$};
\draw[color=black] (-0.5,8) node[scale=0.8] {$d$};
\end{tikzpicture}
\caption {\rmfamily Illustration of the $\scHom^{\cpnorm}$ functor applied to two interval modules: $M=\mathbf{k}[a,b)$ and $N=\mathbf{k}[c,d)$. Only the cases $p=\frac{1}{2}$ (top left), $p=1$ (top right), $p=2$ (bottom left), $p=\infty$ (bottom right) and $0\le d^p- b^p\le c^p-a^p$ are shown. The red regions are points where the module $\pi_{1}^{-1}M:\mathbb{R}^2_{+}\to\cat{Vect_k}$ is nonzero. The cyan regions are the points where the module $(\cpnorm)^{-1}N:\mathbb{R}^2_{+}\to\cat{Vect_k}$ is nonzero.}
\label{fig:hom_of_interval_modules2}
\end{figure}

\begin{example}
\label{example:multiparameter_hom_underline_of_rect_modules}
Let $M,N:\cat{R}^n_+\to\cat{Vect_k}$ be interval modules of the form $M=\cat{k}[A]$ and $N=\cat{k}[B]$, where $A$ and $B$ are the hyper-rectangles
\[A=[a_1,b_1)\times \cdots \times [a_n,b_n), \quad \text{ and } \quad B=[c_1,d_1)\times \cdots \times [c_n,d_n).\]
Consider the hyper-rectangle
\begin{align*}
C&=[(\max(0,c_1^p-a_1^p,d_1^p-b_1^p))^{\frac{1}{p}},(\max(0,d_1^p-a_1^p))^{\frac{1}{p}})\times \cdots\\
&\times [(\max(0,c_n^p-a_n^p,d_n^p-b_n^p))^{\frac{1}{p}},(\max(0,d_n^p-a_n^p))^{\frac{1}{p}}).
\end{align*}
Then, similar to the one-parameter case, we find that $\scHom^{\cpnorm}(M,N)=\cat{k}[C]$.
\end{example}

\subsection{The $\varphi$-Tor and $\varphi$-Ext}
Let $M:\cat{P}\to \cat{Vect_k}$ be a persistence module and let $\varphi:P\times P\to P$ be an order-preserving morphism.  The category $\cat{k^P}$ is a Grothendieck category with a generator $U\od\bigoplus_{p\in P}\cat{k}[U_p]$ (\cite[Proposition 2.23]{Bubenik2021}) and therefore $\cat{k^P}$ has enough projectives and injectives. This allows us to define the derived functors of $\otimes_{\varphi}$ and $\scHom^{\varphi}$ in what follows.

By \cref{corollary:tensor_and_internal_hom_adjoints} the covariant functors $M\otimes_{\varphi}-,-\otimes_{\varphi}M:\cat{k^P}\to \cat{k^P}$ are right exact. Since $\cat{k^P}$ is known to have enough projectives the left derived functors of $M\otimes_{\varphi}-$ and $-\otimes_{\varphi}M$ are well defined. We will denote them by $\cat{Tor}_i^{\varphi}(M,-)$ and $\cat{Tor}_i^{\varphi}(-,M)$. In particular, if 
\[\cdots \to P_2\to P_1\to P_0\to N\to 0\]
is a projective resolution of $N$, applying $M\otimes_{\varphi}-$ yields the chain complex:

\[\cdots \to M\otimes_{\varphi}P_2\to M\otimes_{\varphi} P_1 \to M\otimes_{\varphi}P_0\to 0.\]

Then $\cat{Tor}^{\varphi}_i(M,N)$ is the $i$-th homology of the above complex. On the other hand, $\cat{Tor}^{\varphi}_i(N,M)$ is defined by applying the functor $-\otimes_{\varphi}M$ instead and then taking homology. Note that if $M$ is left (resp. right) $\otimes_{\varphi}$-acyclic then $\cat{Tor}_i^{\varphi}(M,N)=0$ (resp. $\cat{Tor}_i^{\varphi}(N,M)=0$) for $i\ge 1$ and for any persistence module $N$. In particular, if $M$ is projective then $\cat{Tor}_i^{\varphi}(M,N)=\cat{Tor}_i^{\varphi}(N,M)=0$ for $i\ge 1$. It follows from definition that  $\cat{Tor}_0^{\varphi}(M,N)=M\otimes_{\varphi}N$. Finally, observe that if $\varphi$ is symmetric then $\cat{Tor}_i^{\varphi}(M,N)=\cat{Tor}_i^{\varphi}(N,M)$ for any $i\ge 0$ and any persistence modules $M$ and $N$. This follows by \cref{corollary:tensor_for_symmetric_phi}.

The functor $\scHom^{\varphi}(M,-):\cat{k^P}\to \cat{k^P}$ is a covariant left exact functor by \cref{corollary:tensor_and_internal_hom_adjoints}. Since $\cat{k^P}$ is known to have enough injectives the right derived functor of $\scHom^{\varphi}(M,-)$ is well defined. We will denote this functor by $\cat{Ext}^i_{\varphi}(M,-)$. In particular, if 
\[0\to N \to I_0\to I_1\to I_2\to \cdots\]
is an injective resolution of $N$, applying $\scHom^{\varphi}(M,-)$ yields the cochain complex:

\[0 \to \scHom^{\varphi}(M,I_0)\to \scHom^{\varphi}(M,I_1) \to \scHom^{\varphi}(M,I_2)\to \cdots .\]

Then $\cat{Ext}_{\varphi}^ i(M,N)$ is the $i$-th cohomology of the above sequence. 

On the other hand, the functor $\scHom^{\varphi}(-,M):\cat{k^P} \to \cat{k^P}$ is a contravariant left exact functor. To see this, note that $\scHom^{\varphi}(-,M)\od \pi_{2*}\scHom_{\cat{P\times P}}(\pi_1^{-1}-,\varphi^{-1}M)$. The functor $\pi_{1}^{-1}$ is exact as it  has left and right adjoints by \cref{theorem:direct_image_adjunction_1,theorem:direct_image_adjunction_2}. The sheaf internal hom $\scHom_{\cat{P\times P}}$ is known to be left exact and contravariant in the first argument. Finally, the functor $\pi_{2*}$ is left exact. Therefore, as $\scHom^{\varphi}(-,M)$ is a composition of these functors, the claim follows. Thus to define $\cat{Ext}_{\varphi}(-,M)$ we do the following.

Consider a projective resolution of $N$, 
\[\cdots \to P_2\to P_1\to P_0\to N\to 0.\]

Applying $\scHom^{\varphi}(-,M)$ yields the cochain complex:

\[0 \to \scHom^{\varphi}(P_0,M)\to \scHom^{\varphi}(P_1,M) \to \scHom^{\varphi}(P_2,M)\to \cdots .\]

Then $\cat{Ext}_{\varphi}^ i(N,M)$ is the $i$-th cohomology of the above sequence. Note that if $M$ is projective or $N$ is injective, we have  $\cat{Ext}^i_{\varphi}(M,N)=0$ for $i\ge 1$. It follows from definition that  $\cat{Ext}^0_{\varphi}(M,N)=\scHom^{\varphi}(M,N)$. 

We can define $_{\varphi}\cat{Ext}^i(-,-)$, the $i$-th derived functors of $^{\varphi}\scHom(-,-)$ using the same procedure. However in all the computations in this work, the map $\varphi$ will be symmetric which will imply that $_{\varphi}\cat{Ext}^i(-,-)$ and  $\cat{Ext}_{\varphi}(-,-)$ are isomorphic by \cref{corollary:tensor_for_symmetric_phi} and so we won't bother with this here.

\subsection{Free, projective, injective and $\otimes_\varphi$-acyclic persistence modules}
To compute the derived functors of $\otimes_{\varphi}$ and $\scHom^{\varphi}$, for a poset morphism $\varphi:P\times P\to P$, we need to understand projective and injective resolutions of persistence modules.

A classification of interval modules into projective and injective modules was done in \cite[Theorem 6.5]{Bubenik2021}. However this was analysis was done for persistence modules over the poset $\mathbb{R}$ (or $\mathbb{R}^n$). Here we characterize interval modules into projective, injective and $\otimes_{\varphi}$-acyclic for the case $P=\mathbb{R}_+$, and thus some of this analysis has to be redone as the notion of projective and injective objects in the category $\cat{k^P}$ will depend on the poset $P$. However, many of the results can be recovered.

\begin{definition}
    A persistence module $M:\cat{P}\to \cat{Vect_k}$ is: 
    \begin{enumerate}[left=0pt]
    	\item \emph{free} if $M$ is isomorphic to the direct sum $\bigoplus_{a\in G}\cat{k}[U_a]$, where $G$ is a multiset of elements in $P$ and $U_a$ is the principal up-set at $a$.
        \item \emph{projective} if the functor $\Hom_{\cat{k^P}}(M,-):\cat{k}^P\to \cat{Vect_k}$ is exact.
        \item \emph{injective} if the functor $\Hom_{\cat{k^P}}(-,M):(\cat{k}^P)^{\op}\to \cat{Vect_k}$ is exact.
        \item \emph{left} $\otimes_{\varphi}$-\emph{acyclic} if $\cat{Top}^{\varphi}_i(M,N)=0$ for any $i>0$ and any persistence module $N$.
        \item \emph{right} $\otimes_{\varphi}$-\emph{acyclic} if if $\cat{Top}^{\varphi}_i(N,M)=0$ for any $i>0$ and any persistence module $N$.

    \end{enumerate} 
\end{definition}

\begin{remark}
	If the given poset morphism $\varphi$ were symmetric, $\varphi(p,q)=\varphi(q,p)$ for all $p,q\in P$, then $M\otimes_{\varphi}N\approx N\otimes_{\varphi}M$ for all persistence modules $M$ and $N$ by \cref{lemma:generalized_convolution}. Thus, in this case right $\otimes_{\varphi}$-acyclic objects coincide with left $\otimes_{\varphi}$-acyclic objects in $\cat{k^P}$. Note also that  projective persistence modules are left and right $\otimes_{\varphi}$-acyclic by definition. 
\end{remark}

We now introduce some results that will be of use in classifying interval modules.

\begin{lemma}
\label{lemma:free_is_projective}
	A free persistence module is projective.
\end{lemma}

\begin{proof}
	We first prove that for any $a\in P$, the interval module $\cat{k}[U_a]$ is projective. To show that $\Hom_{\cat{k^P}}(\cat{k}[U_a],-)$ is exact, we show that any $f$ in the diagram
	
	\begin{figure}[H]
		\begin{tikzcd}
		&\cat{k}[U_a]\arrow[d,"f"]\arrow[dl,dashed,"\tilde{f}"']\\
		N\arrow[r,"h",twoheadrightarrow] &L
	\end{tikzcd}
	\end{figure}
	where the $h$ is an epimorphism, has an extension $\tilde{f}$ that makes the diagram commute. Note that $h$ being an epimorphism means that every component $h_b:N_b\to L_b$ is surjective, for all $b\in P$. If $b\not \in U_a$, then $f_b=0$ and we set $\tilde{f}_b=0$.  Let $1$ be the generator of $\cat{k}[U_a]_a=\cat{k}$.  Let $v=f_a(1)$. Since $h_a$ is surjective, there is a $w\in N_a$ such that $h_a(w)=v$. Define $\tilde{f}_a(1)=w$ and extend linearly.  Suppose that $a\le b$. Define
	\[\tilde{f}_b(1)= N_{a\le b}(\tilde{f}_a(1))=N_{a\le b}(w),\]
	and extend linearly.
	Then, it is immediate by the definition that $\tilde{f}$ is a natural transformation. Furthermore, $\tilde{f}$ makes the diagram commute.
	
	Since direct sums of projective objects in an abelian category are projective it follows that any free module is projective.
\end{proof}

Since projective persistence modules are immediately left and right $\otimes_{\varphi}$-acyclic we immediately have the following:

\begin{corollary}
\label{corollary:free_modules_are_flat}
	A free persistence module is left $\otimes_{\varphi}$-acyclic and right $\otimes_{\varphi}$-acyclic.
\end{corollary}

The following will allow us to determine $\otimes_{\varphi}$-acyclic interval modules in $\cat{k^R_+}$.

\begin{lemma}
\label{lemma:colimits_of_projectives}
    Filtered colimits of projective persistence modules are left and right $\otimes_{\varphi}$-acyclic.
\end{lemma}

\begin{proof}
	We prove the statement for the right $\otimes_{\varphi}$-acyclic case. The other case is similar.
	Recall that $\cat{k^P}$ is a Grothendieck category. In particular, filtered colimits are exact functors on $\cat{k}^P$. 
	
	Let $\{P_i\}_{i}$ be a filtered diagram of projective persistence modules. 
	Note that since each $P_i$ is projective, it is $\otimes_{\varphi}$-acyclic and thus each $-\otimes_{\varphi}P_i:\cat{k^P}\to \cat{k^P}$ is an exact functor.  
	Let $0\to A\to B\to C\to 0$ be a short exact sequence of persistence modules. The for all $i$, we have the short exact sequence
	\[0\to A\otimes_{\varphi}P_i\to B\otimes_{\varphi}P_i\to C\otimes_{\varphi}P_i\to 0.\] 
	Since filtered colimits are exact in Grothendieck categories we also have the short exact sequence
	\[0\to \colim_iA\otimes_{\varphi}P_i\to \colim_iB\otimes_{\varphi}P_i\to \colim_iC\otimes_{\varphi}P_i\to 0.\] 
	Since $M\otimes_{\varphi}-$ is right exact for any persistence module $M$, $M\otimes_{\varphi}-$ commutes with colimits. Therefore we also have the short exact sequence
	\[0\to (A\otimes_{\varphi}\colim_iP_i)\to (B\otimes_{\varphi}\colim_iP_i)\to (C\otimes_{\varphi}\colim_iP_i)\to 0.\]
	Therefore $-\otimes_{\varphi}\colim_iP_i:\cat{k^P}\to \cat{k^P}$ is an exact functor. It then follows that $\cat{Tor}_j^{\varphi}(M,\colim_iP_i)=0$ for all $j>0$. Hence $\colim_iP_i$ is right $\otimes_{\varphi}$-acyclic.
\end{proof}

The following is immediate.

\begin{lemma}
\label{lemma:colimits_of_intervals}
    If $a<c\in \mathbb{R}_{+}\cup\{\infty\}$, then the filtered colimit $\underset{a<b<c}{\colim}\,\cat{k}[b,c)$ (where the morphisms are the obvious inclusions) is isomorphic to $\cat{k}(a,c)$. Dually, if $c<a \in \mathbb{R}_{+}$ then the filtered limit $\underset{c<b<a}{\lim}\cat{k}(c,b]$ is isomorphic to $\cat{k}(c,a)$. 
\end{lemma}

Recall that persistence modules are equivalent to sheaves by \cref{theorem:functors_and_sheaves}. In particular, every persistence module is a $\cat{k}_{P}$ module, where $\cat{k}_{P}$ is the constant sheaf of rings that assigns to every upset in $P$ the field $\cat{k}$. Recall that a sheaf $F$ on a space $X$ is \emph{flabby} if all the sheaf maps $F(U\subset X):F(X)\to F(U)$ are surjective for all open $U\subset X$. In particular, a persistence module $M$ is flabby as a sheaf the induced morphism $\lim_{p\in P}M_p\to \lim_{p\in U}M_p$ is surjective for all upsets $U$ in $P$. We have the following result that tells us that injective and flabby persistence modules are one and the same.

\begin{proposition}{\cite[Exercise 2.10 b)]{Kashiwara1990}}
\label{prop:injectives_are_flabby}
Let $\cat{k}_{X}$ be a sheaf of rings on a topological space $X$. A $\cat{k}_X$-module $M$ is injective if and only if it is flabby as a sheaf.
\end{proposition}

From now on we assume that $P=\mathbb{R}_+$, which is a lattice.

\begin{proposition}
\label{proposition:projective_intervals}
    The interval module $\cat{k}(a,\infty)$ is not projective for any $a>0$.
\end{proposition}

\begin{proof}
    The same arguments as in \cite[Proposition 6.1]{Bubenik2021}, where $a$ was allowed to be in $\mathbb{R}$, apply here.
\end{proof}

We accumulate all these observations below. 

\begin{proposition}
\label{proposition:injective_projective_classification}
    We have the following classification of interval modules into injective, projective and $\otimes_{\varphi}$-flat objects in $\cat{k^{R_+}}$, for any order-preserving map $\varphi:\mathbb{R}^2_+\to \mathbb{R}_+$.
    \begin{enumerate}[left=0pt]
        \item Every persistence module is $\otimes_{\ell^{\infty}_c}$-acyclic.
        \item The interval module $\cat{k}[a,\infty)$ is projective and left and right $\otimes_{\varphi}$-acyclic  but not injective, for $a>0$. 
         \item The interval module $\cat{k}[0,\infty)$ is both projective and injective and left and right $\otimes_{\varphi}$-acyclic.
        \item The interval module $\cat{k}(a,\infty)$ is not projective nor injective, but it is left and right $\otimes_{\varphi}$-acyclic for all $a>0$. 
        \item The interval modules $\cat{k}[0,a)$ and $\cat{k}[0,a]$ are not projective but are injective, for all $a\in \mathbb{R}_+$.
    \end{enumerate}
\end{proposition}

\begin{proof}
\begin{enumerate}[left=0pt]
	\item This was proved in \cite[Theorem 7.3]{Bubenik2021} for the case $P=\mathbb{R}^n$, but the same arguments apply here. The skeleton of the argument is that by \cref{lemma:sh_tensor_as_a_convolution}, $\otimes_{\ell^{\infty}_c}$ corresponds to ``pointwise" $\cat{k}$-tensor products of vector spaces, which are exact. 
	
	\item The interval module $\cat{k}[a,\infty)$ is free and thus projective for any $a\in \mathbb{R}_+$, by \cref{lemma:free_is_projective}. It is therefore also left and right $\otimes_{\varphi}$-acyclic by \cref{corollary:free_modules_are_flat}. Note that if $a>0$ we have that $\lim_{x\in \mathbb{R}_+}\cat{k}[a,\infty)_x=0$ and $\lim_{x\in U}\cat{k}[a,\infty)_x\approx k$ for any upset $U\subset [a,\infty)$. Therefore $\cat{k}[a,\infty)$ is not flabby as a sheaf and therefore not an injective object by  \cref{prop:injectives_are_flabby}.
	
	\item By (2), $\cat{k}[0,\infty)$ is projective and left and right $\otimes_{\varphi}$-acyclic. Furthermore, it is not hard to see that $\cat{k}[0,\infty)$ is flabby as a sheaf and thus injective by  \cref{prop:injectives_are_flabby}.
	
	\item Let $a>0$. The interval module $\cat{k}(a,\infty)$ is not projective by \cref{proposition:projective_intervals}. Similarly, it is not hard to see that  $\cat{k}(a,\infty)$ is not flabby as a sheaf and is therefore not injective by \cref{prop:injectives_are_flabby}. Note that $\cat{k}(a,\infty)$ is the  filtered colimit of the diagram of interval modules $\cat{k}[c,\infty)\hookrightarrow \cat{k}[b,\infty)$ for $a<b<c$. Therefore $\cat{k}(a,\infty)$ is left and right $\varphi$-acyclic by \cref{lemma:colimits_of_intervals,lemma:colimits_of_projectives}.  
	
	\item For $a\in \mathbb{R}_+$, the interval modules $\cat{k}[0,a)$ and $\cat{k}[0,a]$ are both flabby as sheaves and  therefore are injective by \cref{prop:injectives_are_flabby}. It is known that persistence modules  $M:\cat{R}_+\to \cat{Vect_k}$ are $\mathbb{R}_+$ graded modules, see \cite[Section 2]{Bubenik2021}. In particular by \cite[Corollary 2.20]{Bubenik2021}, a finitely generated persistence module is projective if and only if it is free. The persistence modules $\cat{k}[0,a)$ and $\cat{k}[0,a]$ are finitely generated but not free and thus not projective.	\qedhere
\end{enumerate}
\end{proof}

\begin{remark}
\label{remark:acyclic}
We only need to understand projective and injective objects in order to compute derived functors. However the module $\cat{k}(a,b]$ for example does not have a simple projective resolution. There is no nonzero map $\cat{k}[a,\infty)\to \cat{k}(a,b]$. Thus a projective resolution could perhaps start with $\bigoplus_{a<c\le b}\cat{k}[c,\infty)\to \cat{k}(a,b]$. However, we have a very simple length one $\otimes_{\varphi}$-acyclic resolution
\[ 0\to \cat{k}(b,\infty)\to \cat{k}(a,\infty)\to \cat{k}(a,b]\to 0,\]
and acyclic resolutions can also be used to compute derived functors.
This is why we spent the extra effort in classifying $\otimes_{\varphi}$-acyclic interval modules as well since in some literature interval modules of the form $\cat{k}(a,b]$ are prefered, see for example \cite{Polterovich2017}.
\end{remark}

\subsection{The p-quasinorm Tor}
Here we compute the value of $\Torp_i$, the $i$-th right derived functor of $\otimes_{\cpnorm}$, when applied to interval modules in $\cat{k^R_+}$.

\begin{proposition}
\label{proposition:tor_interval_modules}
    Let $p\in (0,\infty)$. Then 
    \begin{enumerate}[left=0pt]
	\item $\Torp_i(\cat{k}[a,b), \cat{k}[c,d))\approx  \cat{k}[\max(||(b,c)||_p,||(a,d)||_p),||(b,d)||_p)$ if $i=1$, and for $i\ge 2$ we have $\Torp_i(\cat{k}[a,b), \cat{k}[c,d))=0$
        \item $\Torp_i(\cat{k}[a,\infty), M)=\Torp_1(M,\cat{k}[a,\infty))=0$ for $i\ge 1$ and any persistence module $M:\cat{R}_+\to \cat{Vect_k}$.
        \item $\cat{Tor}_i^{\ell_c^{\infty}}(M,N)=0$ for $i\ge 1$ and for any persistence modules $M,N:\cat{R}^n_+\to \cat{Vect_k}$.
    \end{enumerate}
\end{proposition}

\begin{proof}
    The statement in (2) is immediate as $\cat{k}[a,\infty)$ is projective. The statement in (3) is immediate as any persistence module $M$ is $\otimes_{\ell^{\infty}_c}$-acyclic by \cref{proposition:injective_projective_classification}.
    For the statement in (1),  Consider the interval modules $\cat{k}[a,b)$ and $\cat{k}[c,d)$. We have the following projective resolution of $\cat{k}[a,b)$:
    \[0\to \cat{k}[b,\infty)\to \cat{k}[a,\infty)\to \cat{k}[a,b)\to 0.\]
    Since the projective resolution has length $1$ we have  $\Torp_i(\cat{k}[a,b), \cat{k}[c,d))=0$ for $i\ge 2$.
    Apply the functor $-\ptensor \cat{k}[c,d)$ and to the projective resolution. By \cref{proposition:tensor_product_formulas} we have the following (no longer exact) sequence:
    \[0\to \cat{k}[\cpnorm(b,c),\cpnorm(b,d))\to \cat{k}[\cpnorm(a,c),\cpnorm(a,d))\to 0.\]
    Calculating homology (i.e., taking the kernel of the middle map) we get the following:
    \begin{align*}
    \Torp_1(\cat{k}[a,b), \cat{k}[c,d))&\approx \cat{k}[\max(\cpnorm(b,c),\cpnorm(a,d)),\cpnorm(b,d))\\
    &\od \cat{k}[\max(||(b,c)||_p,||(a,d)||_p),||(b,d)||_p).\qedhere
    \end{align*}
\end{proof}

\subsection{The p-quasinorm Ext} 
 Here we compute the value of $\Extp^i$, the $i$-th right derived functor of $\scHom^{\cpnorm}$, when applied to interval modules in $\cat{k^R_+}$.
  
\begin{proposition}
\label{proposition:ext_interval_modules}
    We have the following:
   	\begin{enumerate}[left=0pt]
       \item For $p\in (0,\infty)$, \[\Extp^i(\cat{k}[a,b), \cat{k}[c,d))\approx\cat{k}[(\max(0,c^p-b^p))^{\frac{1}{p}},\min(\max(0,c^p-a^p)^{\frac{1}{p}},\max(0,d^p-b^p)^{\frac{1}{p}})),\]
       if $i=1$ and $\Extp^i(\cat{k}[a,b), \cat{k}[c,d))=0$ for $i\ge 2$.
      \item  $\cat{Ext}_{\ell^{\infty}_c}^1(\cat{k}[a,b), \cat{k}[c,d))=\begin{cases}
 \cat{k}[0,c), & \text{if } a<c\le b<d\\
0, & \text{otherwise}
\end{cases}$ and $\cat{Ext}_{\ell^{\infty}_c}^i(\cat{k}[a,b), \cat{k}[c,d))=0$ for $i\ge 2$.

      \item For $p\in (0,\infty]$, $\Extp^i(\cat{k}[a,\infty),M)=0$, for any $i\ge 1$ and $M:\cat{R}_+\to\cat{Vect_k}$.
		\item For $p\in (0,\infty]$, $\Extp^i(M, \cat{k}[0,d))=0$, for any $i\ge 1$ and $M:\cat{R}_+\to\cat{Vect_k}$.  	  \end{enumerate}
\end{proposition}

\begin{proof}
Statements (3) and (4) are immediate since $\cat{k}[a,\infty)$ is projective and $\cat{k}[0,d)$ is injective by  \cref{proposition:injective_projective_classification}. Thus, we only need to show the first two statements. 

Let $p\in (0,\infty)$ and consider the interval modules $\cat{k}[a,b)$ and $\cat{k}[c,d)$.  By \cref{proposition:injective_projective_classification}, we have the following augmented injective resolution of $\mathbf{k}[c,d)$ 
    \[0\to \cat{k}[c,d)\to \cat{k}[0,d)\to \cat{k}[0,c)\to 0.\]
    Since the resolution has length $1$, it follows that $\Extp^i(\cat{k}[a,b), \cat{k}[c,d))=0$ for $i\ge 2$.
    Apply the functor $\scHom^{\cpnorm}(\cat{k}[a,b),-)$ to the injective resolution. By \cref{proposition:internal_hom_formulas}  we have the following (no longer exact) sequence:
    \[0\to \cat{k}[(\max(0,d^p-b^p))^{\frac{1}{p}},(\max(0,d^p-a^p))^{\frac{1}{p}})\to\cat{k}[(\max(0,c^p-b^p))^{\frac{1}{p}},(\max(0,c^p-a^p))^{\frac{1}{p}})\to 0.\]
    There are four cases to consider. 
    \begin{enumerate}[left=0pt]
        \item Suppose that $0\le c^p-b^p$. Since $c^p-b^p< d^p-b^p$ we also get that $0< d^p-b^p$. These inequalities then also imply that $0<c^p-a^p,d^p-a^p$. Thus the sequence above becomes
        \[0\to \cat{k}[(d^p-b^p)^{\frac{1}{p}},(d^p-a^p)^{\frac{1}{p}})\to\cat{k}[(c^p-b^p)^{\frac{1}{p}},(c^p-a^p)^{\frac{1}{p}})\to 0.\]
        Calculating cohomology (i.e. taking the cokernel of the middle map) we get the interval module $\cat{k}[(c^p-b^p)^{\frac{1}{p}},(\min(c^p-a^p,d^p-b^p))^{\frac{1}{p}})$.
        \item Suppose that $c^p-b^p\le 0\le d^p-b^p$. Then we also have that $0\le d^p-a^p$. Thus the sequence above becomes
        \[0\to \cat{k}[(d^p-b^p)^{\frac{1}{p}},(d^p-a^p)^{\frac{1}{p}})\to\cat{k}[0,(\max(0,c^p-a^p))^{\frac{1}{p}})\to 0.\]
        Calculating cohomology (i.e. taking the cokernel of the middle map) we get the interval module $\cat{k}[0,(\min(\max(0,c^p-a^p),d^p-b^p))^{\frac{1}{p}})$.
        \item Suppose that $d^p-b^p\le 0\le c^p-b^p$. Since $c^p-b^p< d^p-b^p$ this cannot happen and thus this case is impossible. 
        \item Suppose that $c^p-b^p,d^p-b^p\le 0$. Then the sequence above becomes
        \[0\to \cat{k}[0,(\max(0,d^p-a^p))^{\frac{1}{p}})\to\cat{k}[0,(\max(0,c^p-a^p))^{\frac{1}{p}})\to 0.\]
        If $0<c^p-a^p$ then also $0<c^p-a^p<d^p-a^p$ and the middle map above is thus surjective. Thus, in this case the cohomology (cokernel of the middle map) is the 
        \[0=\cat{k}[0,(\min(c^p-a^p,0))^{\frac{1}{p}})=\cat{k}[0,0)\] 
        module. If $c^p-a^p<0$, then $\cat{k}[0,(\max(0,c^p-a^p))^{\frac{1}{p}})=\cat{k}[0,0)$ and thus the cohomology (cokernel of the middle map) is the $0=\cat{k}[0,\min(0,0))$ module. If $d^p-a^p<0$, then also $c^p-a^p<0$ and thus both modules in the above sequence are $0$ and thus the homology is the $0=\cat{k}[0,\min(0,0))$ module. 
    \end{enumerate}
    Combining all these different possible cases we conclude that the formula
    \[\Extp^1(\cat{k}[a,b),\cat{k}[c,d))\approx\cat{k}[(\max(0,c^p-b^p))^{\frac{1}{p}},\min(\max(0,c^p-a^p),\max(0,d^p-b^p))^{\frac{1}{p}}),\]
    accounts for all possible cases.
    
    Now consider once again the augmented projective resolution of $\cat{k}[a,b)$: 
     \[0\to \cat{k}[b,\infty)\to \cat{k}[a,\infty)\to \cat{k}[a,b)\to 0.\]
   Since the resolution has length $1$, it follows that $\cat{Ext}_{\ell^{\infty}_c}^i(\cat{k}[a,b), \cat{k}[c,d))=0$ for $i\ge 2$. Apply the functor $\scHom^{\ell^{\infty}_c}(-,\cat{k}[c,d))$ to the projective resolution to get the (no longer exact) sequence:
   \[0\to \scHom^{\ell^{\infty}_c}(\cat{k}[a,\infty),\cat{k}[c,d))\to \scHom^{\ell^{\infty}_c}(\cat{k}[b,\infty),\cat{k}[c,d))\to 0.\]
    By \cref{corollary:inf_sc_hom}, the resulting (no longer exact) sequence falls in one of the following cases:
    \[\begin{cases}
    0\to \cat{k}[c,d)\to \cat{k}[c,d)\to 0, &  \text{if } a<b\le c<d\\
    0\to \cat{k}[c,d)\to \cat{k}[0,d) \to 0, &  \text{if } a<c\le b<d\\
    0\to \cat{k}[c,d)\to 0\to 0, &  \text{if } a<c<d\le b\\
    0\to \cat{k}[0,d)\to \cat{k}[0,d)\to 0, & \text{if } c\le a<b<d\\
    0\to \cat{k}[0,d)\to 0\to 0, & \text{if } c\le a<d\le b\\
    0\to 0\to 0\to 0, & \text{if } c<d\le a< b
    \end{cases}.\]
By definition, $\cat{Ext}_{\ell^{\infty}_c}^i(\cat{k}[a,b),\cat{k}[c,d))$ is the cohomology (cokernel of the middle map). Therefore 
\[\cat{Ext}_{\ell^{\infty}_c}^i(\cat{k}[a,b),\cat{k}[c,d))=\begin{cases}
 \cat{k}[0,c), & \text{if } a<c\le b<d\\
0, & \text{otherwise}
\end{cases}.\qedhere\]
\end{proof}

\section{The K\"unneth and Universal Coefficient Theorems for persistence modules}

\subsection{Tensor and hom of double complexes}
\label{subsection:tensor_chain_complex}

Let $\cat{A}$ be an abelian category and consider $\cat{Ch(A)}$ the category  of chain complexes valued in $\cat{A}$. Assume that $\cat{A}$ comes equipped with a monoidal product, say $\otimes_*$ and an adjoint for the monoidal product an internal hom, say $\Hom^*$. Let $(C,d^C)$ and $(D,d^D)$ be chain complexes valued in $\cat{A}$. 

\begin{definition}
\label{def:monoidal_prod_of_chain_complexes}
    The \emph{monoidal product double complex} $(C\otimes_*D)_{\bullet,\bullet}$ is the double complex given by 
    \[(C\otimes_* D)_{p,q}\od C_p\otimes_* D_q,\]
    with differentials $d^h_{p,q}:C_p\otimes_* D_q\to C_{p-1}\otimes_* D_q$ defined by $d^h_{p,q}\od d_p^C\otimes_* \id_{D_q}$ and $d^v_{p,q}: C_p\otimes_* D_q\to C_p\otimes_* D_{q-1}$ defined by 
    $d^v_{p,q}\od (-1)^p\id_{C_p}\otimes_*d_q^D$.
    These give rise to the total complex $(C\otimes_* D)_{\bullet}$ defined by 
    \[(C\otimes_*D)_n\od \bigoplus\limits_{p+q=n}(C_p\otimes_* D_q),\] and the differential 
    \[d_n\od \bigoplus_{p+q=n}(d^h_{p,q}+d^v_{p,q}).\]
\end{definition} 

Dually we also have the total complex for the $\Hom^*$ bifunctor.

\begin{definition}
    The \emph{internal hom double complex} $\Hom^*(C,D)^{\bullet,\bullet}$ is the double complex given by 
    \[\Hom^*(C, D)^{p,q}\od \Hom^*(C_p, D_q),\]
    with differentials 
    \[d_h^{p,q}\od (-1)^{p+q+1}(d^C_{p+1})^* :\Hom^*(C_p,D_q)\to \Hom^*(C_{p+1},D_q),\]  
    where $(d^C_{p+1})^*=\Hom^*(d^C_{p+1},\id)$ and  
     \[d_v^{p,q}\od (d^D_{q+1})^* :\Hom^*(C_p,D_q)\to \Hom^*(C_{p},D_{q+1}),\]
        where $(d^D_{p+1})^*=\Hom^*(\id,d^D_{p+1})$.     
        These give rise to the total cochain complex $\Hom^*(C, D)^{\bullet}$ defined by 
    \[\Hom^*(C,D)^n\od\prod\limits_{p+q=n}\Hom^*(C_{p},D_q),\] and the differential 
    \[d^n\od \prod_{p+q=n}(d_h^{p,q}+d_v^{p,q}).\]
\end{definition}

\subsection{K\"unneth theorems for persistence modules}

In this section we state K\"unneth and universal coefficient theorems for chain complexes of persistence modules. We apply these theorems to products of filtered CW complexes. In what follows, persistence modules are functors $M:\cat{P}\to \cat{Vect_k}$ and $\varphi:P\times P\to P$ is order-preserving. Whenever $K$ and $L$ are complexes of persistence modules, we denote by $K\otimes_{\varphi}L$ and $\scHom^{\varphi}(K,L)$ the resulting total complexes as described above.

\begin{theorem}[K\"unneth short exact sequence for persistence modules, I]
\label{theorem:kunneth_1}
Let $(K,d^K)$ and $(L,d^L)$ be two (positive) complexes of persistence module. Suppose that $K$ and its subcomplex of boundaries have all terms projective persistence modules. Then for all $n\ge 0$ there is a natural short exact sequence 
\[0\to \bigoplus_{i+j=n}H_i(K)\otimes_{\varphi} H_j(L)\to H_n(K\otimes_{\varphi} L)\to\bigoplus_{i+j=n-1} \cat{Tor}^{\varphi}_1(H_{i}(K),H_j(L))\to 0.\]
If in addition the complex $L$ and its subcomplex of boundaries have all terms projective, then
the short exact sequence splits, although not naturally.
\end{theorem}

For a proof see \cref{section:appendix_kunneth}. Furthermore, applying the same arguments dually to the functor $\scHom^{\varphi}(-,-)$ as opposed to $-\otimes_{\varphi}-$ we immediately have the following.

\begin{theorem}[K\"unneth short exact sequence for persistence modules, II]
\label{theorem:kunneth_2}
Let $(K,d^K)$ and $(L,d^L)$ be a positive and a negative complex of persistence module, respectively. Suppose that $K$ and its subcomplex of boundaries have all terms projective persistence modules. Then for all $n\ge 0$ there is a natural short exact sequence 
\[0\to \prod_{i+j=n-1}\cat{Ext}^1_{\varphi}(H_i(K),H_{j}(L))\to H^n(\scHom^{\varphi}(K, L))\to\prod_{i+j=n} \scHom^{\varphi}(H_i(K),H_{j}(L))\to 0.\]
If in addition the complex $L$ and its subcomplex of boundaries have all terms injective, then
the short exact sequence splits, although not naturally.
\end{theorem}

Recalling that $-\ell^{\infty}_c-$ is exact in both arguments, its derived functor $\cat{Tor}^{\ell^{\infty}_c}$ is entirely trivial. Therefore we have the following.

\begin{corollary}
\label{corollary:kunneth_infty}
Let $(K,d^K)$ and $(L,d^L)$ be two (positive) complexes of persistence module indexed by $\mathbb{R}^n_+$. Then for all $n\ge 0$ there is an isomorphism 
\[H_n(K\otimes_{\ell^{\infty}_c} L) \approx \bigoplus_{i+j=n}H_i(K)\otimes_{\ell^{\infty}_c} H_j(L).\]
\end{corollary}

By assuming the complex $L$ is a single persistence module concentrated in degree $0$ we also show the following.

\begin{theorem}[Universal Coefficient Theorem for persistence modules, I]
\label{theorem:UCT_1}
Let $A$ be a persistence module. Let $(K,d)$ be a chain complex that has all terms projective and whose subcomplex of boundaries $B$ also has all terms projective. For all $n\ge 0$, there is a natural short exact sequence
\[0\to H_n(K)\otimes_{\varphi} A\to H_n(K\otimes_{\varphi} A)\to \cat{Tor}^{\varphi}_1(H_{n-1}(K),A)\to 0.\]
Furthermore the short exact sequence splits, although not naturally.
\end{theorem}

For a proof see \cref{section:appendix_kunneth}. Similarly, using dual arguments we also have the following.

\begin{theorem}[Universal coefficient theorem persistence modules, II]
\label{theorem:UCT_2}
Let $A$ be a persistence module. Let $(K,d)$ be a chain complex that has all terms projective and whose subcomplex of boundaries $B$ also has all terms projective. For all $n\ge 0$, there is a natural short exact sequence
\[0\to \cat{Ext}^1_{\varphi}(H_{n-1}(K),A)\to H^n(\scHom^{\varphi}(K, A))\to \scHom^{\varphi}(H_n(K),A)\to 0.\]
Furthermore the short exact sequence splits, although not naturally.
\end{theorem}

\subsection{K\"unneth theorem for filtered CW complexes}
All persistence modules here are objects in $\cat{k^R_+}$ and $\varphi:\mathbb{R}_+^2\to \mathbb{R}_+$ is order preserving.

\begin{theorem}
\label{theorem:kunneth_filtered_cw}
    Let $(K,d^K)$ and $(L,d^L)$ be two chain complexes of persistence modules obtained from filtered CW complexes $X$ and $Y$ with filtration functions $f$ and $g$, respectively, as discussed in \cref{subsection:chain_complexes_of_pers_mod}. Then, the $F=\varphi\circ(f\times g)$ filtration on $X\times Y$ induces the chain complex of persistence modules $K\otimes_{\varphi} L$. In particular, the persistent homology of the $F$ filtration can be calculated from \cref{theorem:kunneth_1}.
\end{theorem}

\begin{proof}
    Let $\sigma$ be an $n$-cell of $X$ and let $\tau$ be an $m$-cell of $Y$. These cells have corresponding free summands $\cat{k}[f(\sigma),\infty)$ and $\cat{k}[g(\tau),\infty)$ in $K_n$ and $L_m$. Consider the $F=\varphi\circ(f\times g)$ filtration on $X\times Y$. Then, $\sigma\times \tau$ is an $(n+m)$-cell of $X\times Y$ with corresponding free summand $\cat{k}[f(\sigma),\infty)\otimes_{\varphi}\cat{k}[g(\tau),\infty)\approx\cat{k}[\varphi(f(\sigma),g(\tau)),\infty)$, by \cref{corollary:tensor_product_of_infinite_intervals}), in $(K\otimes_{\varphi} L)_{n+m}$. Note that this correspondence is compatible with the cellular boundary $\partial(\sigma\times \tau)=\partial(\sigma)\times \tau +(-1)^{|\sigma|}\sigma\times\partial(\tau)$ (see, for example, [22, Proposition 3.B.1]), and the boundary map in $K\otimes_{\varphi} L$ (\cref{subsection:tensor_chain_complex}). Thus, $K\otimes_{\varphi} L$ is the chain complex of persistence modules induced by the $F=\varphi\circ(f\times g)$ filtration on $X\times Y$. Note that $K$ and $L$ are chain complexes consisting of free persistence modules and their complexes of boundaries also consist of free persistence modules, by construction. In particular, since free persistence modules are projective by \cref{lemma:free_is_projective}, \cref{theorem:kunneth_1} can be applied to compute the persistent homology of the filtration $F:X\times Y\to \mathbb{R}_+$.
\end{proof}

We illustrate this result with the following example.

\begin{example}
\label{example:kunneth_square}
See \cref{fig:kunneth_square}. Let $(K,d^K)$ and $(L, d^L)$ be chain complexes of persistence modules determined by filtrations of the $1$-simplex. In particular, let
\begin{gather*}
  K_0=\mathbf{k}[a_1,\infty)\oplus \mathbf{k}[b_1,\infty), \quad
  K_1=\mathbf{k}[c_1,\infty),\\
  L_0=\mathbf{k}[a_2,\infty)\oplus \mathbf{k}[b_2,\infty), \quad
  L_1=\mathbf{k}[c_2,\infty),
\end{gather*}
where $a_1\le b_1\le c_1$ and $a_2\le b_2\le c_2$, and let $d^K$ and $d^L$ be the induced boundary maps by the boundary maps of the $1$-simplex, as discussed in \cref{subsection:chain_complexes_of_pers_mod}.
Now consider the product complex $K \ptensor L$ (\cref{def:monoidal_prod_of_chain_complexes}). Note that it is the chain complex of persistence modules corresponding to the filtered cubical complex given by the square in \cref{fig:kunneth_square}, which assigns each cell the $\ell^p_c$ combination of filtration values of corresponding cells in the two $1$-simplices.
\begin{figure}[H]
\centering
\begin{tikzpicture}[line cap=round,line join=round,x=1.0cm,y=1.0cm,scale=1]
\draw[fill=green!20] (6,0) rectangle (8,2);
\draw[color=blue] (0,0)--(2,0);
\filldraw[blue] (1,0) circle (0pt) node[anchor=south,color=black] {$c_1$};
\draw[color=blue] (4,0)--(4,2);
\filldraw[blue] (4,1) circle (0pt) node[anchor=east,color=black] {$c_2$};
\draw[color=blue] (6,0)--(8,0);
\filldraw[blue] (7,0) circle (0pt) node[anchor=south,color=black,scale=0.8] {$\cpnorm(c_1,a_2)$};
\draw[color=blue] (8,0)--(8,2);
\filldraw[blue] (7,2) circle (0pt) node[anchor=north,color=black,scale=0.8] {$\cpnorm(c_1,b_2)$};
\draw[color=blue] (8,2)--(6,2);
\filldraw[blue] (8,1) circle (0pt) node[anchor=north,color=black,scale=0.8,rotate=90] {$\cpnorm(b_1,c_2)$};
\draw[color=blue] (6,2)--(6,0);
\filldraw[blue] (6,1) circle (0pt) node[anchor=south,color=black,scale=0.8,rotate=90] {$\cpnorm(a_1,c_2)$};
\draw[color=black] (3,1) node {$\times$};
\draw[color=black,->] (4.75,1)--(5.25,1);
\draw[color=black](7,1) node[scale=0.8] {$\cpnorm(c_1,c_2)$};
\fill [color=red] (0,0) circle (0.5mm) node[anchor=south,color=black] {$a_1$};
\fill [color=red] (2,0) circle (0.5mm) node[anchor=south,color=black] {$b_1$};
\fill [color=red] (4,0) circle (0.5mm) node[anchor=east,color=black] {$a_2$};
\fill [color=red] (4,2) circle (0.5mm) node[anchor=east,color=black] {$b_2$};
\fill [color=red] (6,0) circle (0.5mm) node[anchor=north,color=black,scale=0.8] {$\cpnorm(a_1,a_2)$};
\fill [color=red] (8,0) circle (0.5mm) node[anchor=north,color=black,scale=0.8] {$\cpnorm(b_1,a_2)$};
\fill [color=red] (8,2) circle (0.5mm) node[anchor=south,color=black,scale=0.8] {$\cpnorm(b_1,b_2)$};
\fill [color=red] (6,2) circle (0.5mm) node[anchor=south,color=black,scale=0.8] {$\cpnorm(a_1,b_2)$};
\end{tikzpicture}
\caption{A product complex with a filtration with respect to $\ptensor$, visualized.}
\label{fig:kunneth_square}
\end{figure}
One can compute that the only non-trivial homology groups are
\[
    H_0(K)=\cat{k}[a_1,\infty)\oplus\cat{k}[b_1,c_1),\, H_0(L)=\cat{k}[a_2,\infty)\oplus \cat{k}[b_2,c_2),
\]
\begin{align*}
H_0(K\ptensor L)&=\cat{k}[\cpnorm(a_1,a_2),\infty) \oplus \cat{k}[\cpnorm(a_1,b_2),\cpnorm(a_1,c_2))\oplus \cat{k}[\cpnorm(b_1,a_2),\cpnorm(c_1,a_2))\\
&\oplus \mathbf{k}[\cpnorm(b_1,b_2),\min(\cpnorm(b_1,c_2),\cpnorm(c_1,b_2)),    
\end{align*}
and
\[H_1(K\ptensor L)=\mathbf{k}[\max(\cpnorm(b_1,c_2),\cpnorm(c_1,b_2),\cpnorm(c_1,c_2)).\]
Note that
$H_0(K \ptensor L) \approx H_0(K) \ptensor H_0(L)$ and
$H_1(K \ptensor L) \approx \mathbf{Tor}^{\cpnorm}_1(H_0(K),H_0(L))$,
which agrees with \cref{theorem:kunneth_1}.
\end{example}

\begin{remark}
In \cref{theorem:kunneth_filtered_cw} $\varphi$ was not assumed to be symmetric. Thus, if one has filtrations $f:X\to \mathbb{R}_+$ and $g:Y\to \mathbb{R}_+$ and wishes to ``weigh`` them differently when combined on $X\times Y$ our theorem in its full generality allows them to do so. The only limitation is deriving formulas for $\cat{k}[a,b)\otimes_{\varphi}\cat{k}[c,d)$ and $\cat{Tor}_1^{\varphi}(\cat{k}[a,b),\cat{k}[c,d))$ for a given $\varphi$. We spend quite a bit of effort doing this for the map $\ell^p_c$ for $p\in (0,\infty]$ here and we hope the homological algebra methods developed here will serve future such endeavors for other maps $\varphi$. 
\end{remark}

\subsection{Persistent Borel-Moore homology of open complexes}
\label{section:borel_moore}
	Let $X$ be a simplicial complex and let $T\subset X$ be a nonempty subcomplex, and $H=X\setminus T$. Any such $H$ is called an \emph{open complex} \cite{knudson2026discretemorsetheoryopen}. Open complexes naturally arise in discrete stratified Morse theory \cite{Knudson2022}.
	The \emph{Borel-Moore} homology $H_i^{BM}(H)$ is the relative homology $H_i(X, T)$. Borel–Moore homology was originally developed in \cite{Borel1960} as a homology theory for
locally compact (in particular, non-compact) spaces. The standard reference is \cite[Chapter 5]{Bredon1997}.   Equivalently, Borel-Moore homology can be viewed as the relative homology of the one-point compactification of $H$ (relative to the point at infinity), i.e. $H_i^{BM}(H)=H_i(H^*,\{\infty\})$ where $H^*$ is the one-point compactification of $H$. There are other equivalent formulations of Borel-Moore homology but it is not necessary to mention them here. We are still working with coefficients in a field $\cat{k}$ when considering $H_i^{BM}(H)$. That is $H_i^{BM}(H)=H_i(X, T;\cat{k})$.
	
Let $f:X\to \mathbb{R}_+$ be a filtration of a simplicial complex $X$ (we defined a filtration of CW complexes and every simplicial complex is a CW complex). If particular for whenever $\sigma\subset \tau$ are simplices in $X$ we have $f(\sigma)\le f(\tau)$. Assume that $\alpha>f(\sigma)$ for all simplices $\sigma$ of $X$ and let $A=\cat{k}[\alpha,\infty)$. Let $C^f(X)$ be the chain complex of persistence modules obtained from $X$ and $f$ as described in \cref{subsection:chain_complexes_of_pers_mod}.

Let $p\in (0,\infty)$ and consider the cochain complex of persistence modules $\scHom^{\cpnorm}(C^f(X),A)$. In particular, $\scHom^{\cpnorm}(C^f(X),A)^n\od \scHom^{\cpnorm}(C^f_n(X),A)$ for all $n\ge 0$. By \cref{proposition:internal_hom_formulas} we have that 

\[\scHom^{\cpnorm}(C^f(X),A)^n\approx\bigoplus_{\sigma\in X^{(n)}} \cat{k}[(\alpha^p-f(\sigma)^p)^{\frac{1}{p}},\infty),\]
where $X^{(n)}$ are the $n$-simplices of $X$. Note that if $f(\sigma)<f(\tau)$, then $\alpha^p-f(\tau)^p<\alpha^p-f(\sigma)^p$, therefore 
\[\sigma\mapsto (\alpha^p-f(\sigma)^p)^{\frac{1}{p}},\]
 for $\sigma$ a simplex in $X$, defines a function which we will denote by $f_{\alpha^p}$, $f_{\alpha^p}:X\to \mathbb{R}_+$. This function is like an ``opposite filtration" in the sense that it reverses the order in which the simplices in $X$ appeared under $f$, and it additionally shifts the reversed order by the parameter $\alpha$. Since the filtration $f$ had to satisfy $f(\sigma)\le f(\tau)$ for $\sigma\subset \tau$ we had simplices appearing after their boundaries in the sublevel sets. However, in the filtration $f_{\alpha^p}$ we have simplices appearing before their boundaries. This results in sublevel sets of $f_{\alpha^p}$ consisting of typically non-compact spaces, see \cref{fig:borel-moore-fig}. In particular, each sublevel set $f^{-1}_{\alpha^p}[0,t]$ satisfies
 \[f^{-1}_{\alpha^p}[0,t]\od\{\sigma\,|\, \alpha^p-f(\sigma)^p\le t^p\}=\{\sigma\,|\, f(\sigma)^p\ge \alpha^p-t^p\}.\]
 Therefore $f^{-1}_{\alpha^p}[0,t]$ is the set 
 \begin{align*}
 f^{-1}_{\alpha^p}[0,t]&=f^{-1}[\max(0,\alpha^p-t^p)^{\frac{1}{p}},\infty)=f^{-1}[0,\infty)\setminus  f^{-1}[0,\max (0,\alpha^p-t^p)^{\frac{1}{p}})\\
 &=X\setminus  f^{-1}[0,\max (0,\alpha^p-t^p)^{\frac{1}{p}}).
 \end{align*}
 This is typically an open complex, unless we end up substracting the empty set. Let \[T_t\od  f^{-1}[0,\max (0,\alpha^p-t^p)^{\frac{1}{p}}).\]
 
Let $H_{t}\od f^{-1}_{\alpha^p}[0,t]$. Then, $H_t=X\setminus T_t$. Each $T_t$ is a closed subcomplex of $X$, therefore for $s\le t$, we have open inlclusions $H_t\hookrightarrow H_s$. Borel-Moore homology is contravariant with respect to open inclusions, thus we have linear maps maps 
\[H_n^{BM}(H_t)\od H_n(X, T_t)\to H_n^{BM}(H_s)\od H_n(X,T_s)\] 
for $s\le t$. We call this sequence of vector spaces and linear maps the \emph{persistent Borel-Moore homology} of the filtration $f_{\alpha^p}$. We denote this persistence module by $\mathcal{H}^{BM,f_{\alpha^p}}_n:\cat{R}_+^{\op}\to \cat{Vect_k}$ (note that this is a contravariant functor on $\mathbb{R}_+$, unlike our standard persistence modules on $\mathbb{R}_+$). However, if we take duals  $\Hom_{\cat{k}}(H_n^{BM}(H_t),\cat{k})$ we do have induced linear maps 
\[\Hom_{\cat{k}}(H_n^{BM}(H_s),\cat{k})= \Hom_{\cat{k}}(H_n(X, T_s),\cat{k})\to \Hom_{\cat{k}}(H_n^{BM}(H_t),\cat{k})= \Hom_{\cat{k}}(H_n(X,T_t),\cat{k}),\]
whenever $s\le t$.
Note that because everything has coefficients in the field $\cat{k}$ we have isomorphisms $H_n^{BM}(H_t)\approx \Hom_{\cat{k}}(H_n^{BM}(H_t),\cat{k})$. Therefore, the duals of the Borel-Moore homology of the filtration $f_{\alpha^p}$ give a covariant functor $\cat{R}_+\to \cat{Vect}_k$. From all these observations we get the following.
  
\begin{theorem}
	\label{theorem:borel_moore}
	Let $C^f(X)$ be a chain complexes of persistence modules obtained from a filtered CW complexes $X$ with filtration function $f$, as discussed in \cref{subsection:chain_complexes_of_pers_mod}. Let $A=\cat{k}[\alpha,\infty)$ where $\alpha$ is larger than any value of $f$. Then, $H^n(\scHom^{\cpnorm}(C^f(X),A))$ is the dual of the $n$-th persistent Borel-Moore homology of the filtration $f_{\alpha^p}$ on $X$. In particular, the persistent Borel-Moore homology of the $f_{\alpha^p}$ filtration can be calculated from \cref{theorem:UCT_2}.
\end{theorem}

\begin{proof}
Let $X^{(n)}$ denote the set of $n$-cells of X. For $n\ge 0$, define 
\[C^{f_{\alpha^p}}_n(X)\od\bigoplus_{\sigma\in X^{(n)}}\cat{k}[(\alpha^p-f(\sigma)^p)^{\frac{1}{p}},\infty).\]
where now $\cat{k}[(\alpha^p-f(\sigma)^p)^{\frac{1}{p}},\infty):\cat{R}_+^{\op}\to \cat{Vect_k}$ is contravariant. Note that for each $t\in \mathbb{R}_+$, $C_n^{f_{\alpha^p}}(X)_t$ can be identified with the relative chains $C_n(X,T_t)$ as per the discussion above. We thus have boundary maps $C_n^{f_{\alpha^p}}(X)\to C_{n-1}^{f_{\alpha^{p}}}(X)$ induced by the relative boundaries $C_n(X,T_t)\to C_{n-1}(X,T_t)$. Dualizing everything gives us the cochain complexes of persistence modules:
\[C^n_{f_{\alpha^p}}(X)\approx \bigoplus_{\sigma\in X^{(n)}}\cat{k}[(\alpha^p-f(\sigma)^p)^{\frac{1}{p}},\infty),\]
where now $\cat{k}[(\alpha^p-f(\sigma)^p)^{\frac{1}{p}},\infty):\cat{R}_+\to \cat{Vect_k}$ is covariant and we can identify $C^n_{f_{\alpha^p}}(X)_t$ with the cochains $C^n(X,T_t)$ for all $t\in \mathbb{R}_+$. The cochain of complexes of persistence modules and its coboundary are precisely the cochain complex $\scHom^{\cpnorm}(C^f(X),A)$ by construction. Thus the result follows.
\end{proof}

\begin{example}
\label{example:borel_moore_persistent}
Let $0\le a_0\le a_1\le a_2\le b_0\le b_1\le b_2\le c$ be real numbers and consider the filtration of the $2$-simplex $X$ in \cref{fig:21} given by $f(x)=a_0$, $f(y)=a_1$, $f(z)=a_2$, $f([x,y])=b_0$, $f([y,z])=b_1$, $f([x,z])=b_2$ and $f([x,y,z])=c$.
\begin{figure}[H]
\centering

\begin{tikzpicture}[line cap=round,line join=round,x=1.0cm,y=1.0cm,scale=0.5]
\fill [color=blue!80] (-1,2) circle (1mm);
\draw[color=black] (-1,1) node[font=\scriptsize]  {$a_0$};
\draw[color=black] (-0.25,3) node[scale=1] {$\hookrightarrow$};

\fill [color=blue!80] (0.5,2) circle (1mm);
\fill [color=blue!80] (2.5,2) circle (1mm);

\draw[color=black] (1.5,1) node[font=\scriptsize]  {$a_1$};
\draw[color=black] (3.25,3) node[scale=1] {$\hookrightarrow$};
\fill[color=blue!80] (4,2) circle(1mm);
\fill[color=blue!80] (6,2) circle(1mm);
\fill[color=blue!80] (5,4) circle (1mm);
\draw[color=black] (6.75,3) node[scale=1] {$\hookrightarrow$};
\draw[color=black] (5,1) node[font=\scriptsize]  {$a_2$};
\draw[color=red!35, line width=0.7mm] (7.5,2)--(9.5,2);
\fill[color=blue!80](7.5,2) circle(1mm);
\fill[color=blue!80] (9.5,2) circle (1mm);
\fill[color=blue!80] (8.5,4) circle (1mm);
\draw[color=black] (8.5,1) node[font=\scriptsize]  {$b_0$};

\draw[color=black] (10.25,3) node[scale=1] {$\hookrightarrow$};

\draw[color=red!35, line width=0.7mm] (11,2)--(13,2);
\draw[color=red!35, line width=0.7mm] (13,2)--(12,4);
\fill[color=blue!80](11,2) circle(1mm);
\fill[color=blue!80] (13,2) circle (1mm);
\fill[color=blue!80] (12,4) circle (1mm);
\draw[color=black] (12,1) node[font=\scriptsize]  {$b_1$};

\draw[color=black] (13.75,3) node[scale=1] {$\hookrightarrow$};

\draw[color=red!35, line width=0.7mm] (14.5,2)--(16.5,2);
\draw[color=red!35, line width=0.7mm] (16.5,2)--(15.5,4);
\draw[color=red!35, line width=0.7mm] (15.5,4)--(14.5,2);
\fill[color=blue!80](14.5,2) circle(1mm);
\fill[color=blue!80] (16.5,2) circle (1mm);
\fill[color=blue!80] (15.5,4) circle (1mm);

\draw[color=black] (15.5,1) node[font=\scriptsize]  {$b_2$};

\draw[color=black] (17.25,3) node[scale=1] {$\hookrightarrow$};
\draw[fill=green!35, line width=0.2mm] (18,2)--(20,2)--(19,4)--cycle;
\draw[color=red!35, line width=0.7mm] (18,2)--(20,2)--(19,4)--cycle;
\fill[color=blue!80](18,2) circle(1mm);
\fill[color=blue!80] (20,2) circle (1mm);
\fill[color=blue!80] (19,4) circle (1mm);
\draw[color=black] (19,1) node[font=\scriptsize] {$c$};
\end{tikzpicture}

\begin{tikzpicture}[line cap=round,line join=round,x=1.0cm,y=1.0cm,scale=0.5]

\fill[green!35] (-1,2)--(1,2)--(0,4)--cycle;
\draw[color=black] (0,1) node[font=\scriptsize]  {$(\alpha^p-c^p)^{\frac{1}{p}}$};

\draw[color=black] (1.75,3) node[scale=1] {$\hookrightarrow$};

\fill[green!35] (2.5,2)--(4.5,2)--(3.5,4)--cycle;
\draw[color=red!35, line width=0.7mm] (3.43,3.86)--(2.57,2.14);
\draw[color=black] (3.5,1) node[font=\scriptsize]  {$(\alpha^p-b_2^p)^{\frac{1}{p}}$};

\draw[color=black] (5.25,3) node[scale=1] {$\hookrightarrow$};

\fill[green!35] (6,2)--(8,2)--(7,4)--cycle;

\draw[color=red!35, line width=0.7mm] (6.93,3.86)--(6.07,2.14);
\draw[color=red!35, line width=0.7mm] (7.07,3.86)--(7.93,2.14);

\draw[color=black] (8.75,3) node[scale=1] {$\hookrightarrow$};
\draw[color=black] (7,1) node[font=\scriptsize]  {$(\alpha^p-b_1^p)^{\frac{1}{p}}$};

\fill[green!35] (9.5,2)--(11.5,2)--(10.5,4)--cycle;

\draw[color=red!35, line width=0.7mm] (10.43,3.86)--(9.57,2.14);
\draw[color=red!35, line width=0.7mm] (10.57,3.86)--(11.43,2.14);
\draw[color=red!35, line width=0.7mm] (11.4,2)--(9.6,2);

\draw[color=black] (10.5,1) node[font=\scriptsize]  {$(\alpha^p-b_0^p)^{\frac{1}{p}}$};

\draw[color=black] (12.25,3) node[scale=1] {$\hookrightarrow$};



\fill[green!35] (13,2)--(15,2)--(14,4)--cycle;
\draw[color=red!35, line width=0.7mm] (14,4)--(14.93,2.14);
\draw[color=red!35, line width=0.7mm] (14,4)--(13.07,2.14);
\draw[color=red!35, line width=0.7mm] (14.9,2)--(13.1,2);
\fill[color=blue!80] (14,4) circle (1mm);

\draw[color=black] (14,1) node[font=\scriptsize]  {$(\alpha^p-a_2^p)^{\frac{1}{p}}$};

\draw[color=black] (15.75,3) node[scale=1] {$\hookrightarrow$};

\fill[green!35] (16.5,2)--(18.5,2)--(17.5,4)--cycle;
\draw[color=red!35, line width=0.7mm] (17.5,4)--(18.5,2);
\draw[color=red!35, line width=0.7mm] (17.5,4)--(16.57,2.14);
\draw[color=red!35, line width=0.7mm] (18.5,2)--(16.6,2);
\fill [color=blue!80] (18.5,2) circle (1mm);
\fill [color=blue!80] (17.5,4) circle (1mm);

\draw[color=black] (17.5,1) node[font=\scriptsize]   {$(\alpha^p-a_1^p)^{\frac{1}{p}}$};

\draw[color=black] (19.25,3) node[scale=1] {$\hookrightarrow$};
\draw[fill=green!35, line width=0.2mm] (20,2)--(22,2)--(21,4)--cycle;
\draw[color=red!35, line width=0.7mm] (20,2)--(22,2)--(21,4)--cycle;
\fill [color=blue!80] (20,2) circle (1mm);
\fill [color=blue!80] (22,2) circle (1mm);
\fill [color=blue!80] (21,4) circle (1mm);
\draw[color=black] (21,1) node[font=\scriptsize]  {$(\alpha^p-a_0^p)^{\frac{1}{p}}$};

\end{tikzpicture}
\caption{A $2$-simplex $X$ with filtration $f$ (top) and the ``reversed filtration" $f_{\alpha^p}$ (bottom).}
\label{fig:borel-moore-fig}
\end{figure}
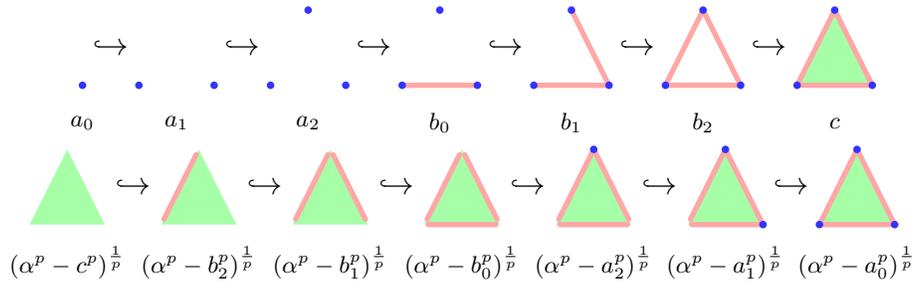

The corresponding chain complex of persistence modules associated with this filtration is:
\begin{align*}
    C^f_0(X)&=\cat{k}[a_0,\infty)\oplus\cat{k}[a_1,\infty)\oplus\cat{k}[a_2,\infty),\\
    C^f_1(X)&=\cat{k}[b_0,\infty)\oplus\cat{k}[b_1,\infty)\oplus \cat{k}[b_2,\infty),\\
    C^f_2(X)&=\cat{k}[c,\infty), C^f_n(X)=0,n\ge 3,
\end{align*}
with boundary maps induced by the boundary maps of the $2$-simplex $X$.
The homology of the chain complex $C^f(X)$ is:
\[H_0(C^f(X))=\mathcal{H}_0(X^f)=\cat{k}[a_0,\infty)\oplus\cat{k}[a_1,b_0)\oplus \cat{k}[a_2,b_1),
H_1(C^f(X))=\mathcal{H}_1(X^f)=\cat{k}[b_2,c),\]
and $H_n(C^f(X))=\mathcal{H}_n(X^f)$ for $n\ge 2$.

 Let $A=\cat{k}[\alpha,\infty)$ where $\alpha> c$. Then, by \cref{proposition:internal_hom_formulas} $\scHom^{\cpnorm}(C^f(X),A)$ is the cochain complex:
 \begin{align*}
    \scHom^{\cpnorm}(C^f(X),A)^0&=\cat{k}[(\alpha^p-a^p_0)^{\frac{1}{p}},\infty)\oplus\cat{k}[(\alpha^p-a^p_1)^{\frac{1}{p}},\infty)\oplus\cat{k}[(\alpha^p-a^p_2)^{\frac{1}{p}},\infty),\\
    \scHom^{\cpnorm}(C^f(X),A)^1&=\cat{k}[(\alpha^p-b^p_0)^{\frac{1}{p}},\infty)\oplus\cat{k}[(\alpha^p-b^p_1)^{\frac{1}{p}},\infty)\oplus \cat{k}[(\alpha^p-b^p_2)^{\frac{1}{p}},\infty),\\
   \scHom^{\cpnorm}(C^f(X),A)^2&=\cat{k}[(\alpha^p-c^p)^{\frac{1}{p}},\infty), \scHom^{\cpnorm}(C^f(X),A)^n=0, n\ge 3.
\end{align*}
Because $\alpha^p-c^p<\alpha^p-b_2^p<\alpha^p-b_1^p<\alpha^p-b_0^p<\alpha^p-a_2^p<\alpha^p-a_1^p<\alpha^p-a_0^p$ we can think of $\scHom^{\cpnorm}(C^f(X),A)$ as the cochain complex of the ``filtration" in \cref{fig:borel-moore-fig}. By \cref{theorem:UCT_2,proposition:internal_hom_formulas,proposition:ext_interval_modules} the cohomology of this cochain complex is given by 
 \begin{align*}
    H^0(\scHom^{\cpnorm}(C^f(X),A))&\approx \scHom^{\cpnorm}(H_0(C^f(X),A))\approx \cat{k}[(\alpha^p-a^p_0)^{\frac{1}{p}},\infty),\\
    H^1(\scHom^{\cpnorm}(C^f(X),A))&\approx \cat{Ext}^{1}_{\cpnorm}(H_0(C^f(X)),A)\oplus \scHom^{\cpnorm}(H_1(C^f(X)),A)\approx  \\
    &\approx
    \cat{k}[(\alpha^p-b^p_0)^{\frac{1}{p}},(\alpha^p-a_1^p)^{\frac{1}{p}})\oplus\cat{k}[(\alpha^p-b^p_1)^{\frac{1}{p}},(\alpha^p-a^p_2)^{\frac{1}{p}})\oplus 0,\\
   H^2(\scHom^{\cpnorm}(C^f(X),A))&=\cat{Ext}^1_{\cpnorm}(H_1(C^f(X)),A)\approx\cat{k}[(\alpha^p-c^p)^{\frac{1}{p}},(\alpha^p-b_2^p)^{\frac{1}{p}}),\\
    H^n(\scHom^{\cpnorm}(C^f(X),A))&=0, n\ge 3.
\end{align*}
Note that this computation agrees with the persistent Borel-Moore homology of the filtration in \cref{fig:borel-moore-fig}. For example at the filtration value of $(\alpha^p-c^p)^{\frac{1}{p}}$ we have the interior of a $2$-simplex whose Borel-Moore homology is only nonzero in degree $2$. A way to see this is that its one-point compactification is the sphere. Similarly, at the filtration value $(\alpha^p-b^p_0)^{\frac{1}{p}}$ we have two nonzero homology classes in degree one. The one-point compactification is the triangle with the 3-vertices identified which once again agrees with the Borel-Moore homology of this non-compact space.
\end{example}

\section{Different filtrations on product metric spaces}

\label{section:approximating}

Let $(X,d_X)$ and $(Y,d_Y)$ be two metric spaces. Let $d_{X\times Y}^p$ be the $p$-norm distance induced by $d$ and $e$ on $X\times Y$. In other words, $d^p_{X\times Y}((x_1,y_1),(x_2,y_2))=||d_X(x_1,x_2),d_Y(y_1,y_2)||_p$. In this section we discuss different filtered CW complexes arising from different samplings of points on $X\times Y$. The goal of this section is to relate various persistent modules on $(X\times Y,d_{X\times Y})$ to those on $(X,d_X)$ and $(Y,d_Y)$ by using the K\"unneth theorems we proved. 

\subsection{The Vietoris-Rips simplicial complex}

For a metric space $(X,d)$, let $\mathbf{VR}_{a}(X,d)$ be the Vietoris-Rips complex at scale $a$. This is an abstract simplicial complex on the vertex set $X$ whose $n$-simplices are given by

\[\cat{VR}_{a}^n(X)\od\{\sigma\subset X\,|\, |\sigma|=n+1, x,x'\in \sigma\Longrightarrow d(x,x')\le a\}.\]

In other words, $\sigma$ is a simplex in $\cat{VR}_{a}(X,d)$ if its diameter is less than $a$. Varying $a\in \mathbb{R}_+$ we have the collection $\cat{VR(X,d)}_{\bullet}\od \{\cat{VR}_a(X,d)\}_{a\in \mathbb{R}_+}$ which is an example of filtered CW complex as discussed in \cref{subsection:chain_complexes_of_pers_mod}. This gives rise to a chain complex of persistence modules whose homology is the persistent Vietoris-Rips homology. The Vietoris-Rips complex is the most common construction in topological data analysis since there is a highly efficient algorithm for computing the persistent homology of its filtration \cite{Bauer2021}. The Vietoris-Rips filtered complex is typically constructed on finite point cloud samples from a metric space. Next we discuss sampling in products of metric spaces.

\subsection{Sampling points in a metric measure space}

Let $(X,d_X,\mu_X)$ and $(Y,d_Y,\mu_Y)$ be two metric Borel probability measure spaces. Here we discuss 2 ways to sample points in $(X\times Y,d_{X\times Y}^p, \mu_{X}\times \mu_{Y})$, where $d^p$ is the $\ell^p$ metric induced from $d_X$ and $d_Y$.

The first method is to sample $S^X_N=\{x_1,\dots , x_N\}$ i.i.d from $(X,\mu_X)$ and $S^M_M=\{y_1,\dots, y_M\}$ i.i.d from $(Y,\mu_Y)$. This yields a sample of points $S^X_N\times S^Y_M$ in $X\times Y$.

Alternatively we can sample points $S^{X\times Y}_{NM}=\{(x,y)_1,\dots, (x,y)_{N\times M}\}$ i.i.d from $(X\times Y,\mu_X\times \mu_Y)$. Obviously the points in $S^X_N\times S^Y_M$ always form a ``grid" in $X\times Y$ whereas the points in $S^{X\times Y}_{NM}$ do not. However as $N,M\to \infty$ we can ask what happens to the distance between the sets $S^{X\times Y}_{NM}$ and $S^X_N\times S^Y_M$, in $(X\times Y,d^p_{X\times Y})$. By distance we mean the Hausdorff distance between sets.

\begin{definition}
\label{def:hausdorff_distance}
Let $(X,d)$ be a metric space. For each pair of non-empty subsets $A,B\subset X$, the \emph{Hausdorff distance} between $A$ and $B$ is defined as
	\[d_H(A,B)=\max\left\{\sup_{a\in A}\text{dist}(a,B),\sup_{b\in B}\text{dist}(b,A)\right\},\]
	where $\text{dist}(x,C)=\inf_{y\in C}d(x,y)$ is the distance of a point $x$ to a set $C$.
\end{definition}

The Hausdorff distance allows us to define a distance between any pair of compact metric spaces $X$ and $Y$ by setting
\[d_{\text{GH}}(X,Y)=\inf_{f,g,Z}d_H^Z(f(X),g(Y)),\]
where $f:X\to Z$, $g:Y\to Z$ are isometric embeddings into a common metric space $Z$ and $d_H^Z$ is the Hausdorff distance between subsets in $Z$. The value $d_{\text{GH}}(X,Y)$ is the \emph{Gromov-Hausdorff distance} of $X$ and $Y$. For a comprehensive treatment of Gromov-Hausdorff distance see \cite{Burago2001}.
 
The following results show that as $N,M\to \infty$, the Hausdorff distance between the sets $S^{X\times Y}_{NM}$ and $S^X_N\times S^Y_M$ goes to $0$ almost surely.

\begin{proposition}
\label{prop:iid_sample}
	Let $(X,d)$ be a compact metric space and $\mu$ a Borel probability measure on $X$. Let $x_1, x_2, \dots, x_N$ be i.i.d. samples from $\mu$. Then
	\[\lim_{N\to \infty}\sup_{x\in \text{supp}(\mu)}\min_{i=1,\dots, N}d(x,x_i)=0 \text{ almost surely}.\]
\end{proposition} 

\begin{proof}
For a proof, see \cref{section:appendix_B}.
\end{proof}

\begin{proposition}
\label{prop:grid_sample}
	Let $(X,d_X)$ and $(Y,d_Y)$ be compact metric spaces with Borel probability measures $\mu_X$ and $\mu_Y$, respectively, both having full support. Let $\{x_i\}_{i=1}^N$ and $\{y_i\}_{i=1}^M$ be i.i.d. samples from $\mu_X$ and $\mu_Y$, respectively. Let $d$ be the metric on $X\times Y$ defined by $d((x,y),(x',y'))=(d_X(x,x')^p+d_Y(y,y')^p)^{1/p}$ for $p\ge 1$. Then
	\[\lim_{N,M\to \infty}\sup_{(x,y)\in X\times Y}\min_{1\le i\le N,1\le j\le M}d((x,y),(x_i,y_j))=0 \text{ almost surely}.\]
\end{proposition}

\begin{proof}
For a proof, see \cref{section:appendix_B}.
\end{proof}

\begin{theorem}
\label{theorem:hausdorff_distance_grid_and_sample}
	Let $(X,d_X), (Y,d_Y)$ be compact metric spaces with Borel probability measures $\mu_X$, $\mu_Y$ respectively, both having full support.  Let $d$ be the metric on $X\times Y$ defined by $d((x,y),(x',y'))=(d_X(x,x')^p+d_Y(y,y')^p)^{1/p}$ for $p\ge 1$ and $d_H$ the induced Hausdorff distance on subsets of $X\times Y$.
	Then 
	\[d_H(G_{NM},S_{NM})\to 0 \text{ almost surely as } N,M\to \infty.\]
\end{theorem}

\begin{proof}
	For any $N,M \ge 0$, the sets $S^{X\times Y}_{NM}$ and $S^X_N\times S^Y_M$ are finite and hence compact. Since the Hausdorff distance satisfies the triangle inequality when restricted to compact sets we have
	\[d_H(S_N^X\times S^Y_M,S^{X\times Y}_{NM})\le d_H(S_N^X\times S_M^Y,X\times Y)+d_H(X\times Y, S^{X\times Y}_{NM}),\]
	for any $N,M\ge 0$. We have that 
	\[\lim_{N,M\to \infty}\sup_{(x,y)\in X\times Y}\min_{(x_i,y_j)\in S_N^X\times S_M^Y}d((x,y),(x_i,y_j))=0 \text{ almost surely},\]
	by \cref{prop:grid_sample}. This means that  $d_H(S_N^X\times S_M^Y,X\times Y)\to 0$ almost surely. By \cref{prop:iid_sample} we have that 
	\[\lim_{N,M\to \infty}\sup_{(x,y)\in X\times Y}\min_{s_i\in S^{X\times Y}_{NM}}d((x,y),s_i)=0 \text{ almost surely}.\]
	This means that $d_H(S^{X\times Y}_{NM},X\times Y)\to 0$ almost surely. Therefore, by the triangle inequality we have that $d_H(S_N^X\times S_M^Y,S^{X\times Y}_{NM})\to 0$ almost surely.
\end{proof}

We now ask how different are the persistence modules obtained from the Vietoris-Rips complex filtrations of the point clouds $S^{X\times Y}_{NM}$ and $S^X_N\times S^Y_M$ in $(X\times Y,d^p)$, respectively. A standard way of measuring distance between persistence modules that are interval decomposable is the bottleneck distance (see \cite{CohenSteiner2006}) which we describe below. 

Let $M:\cat{R}_+\to \cat{Vect_k}$ be an interval decomposable persistence modules. Let $\text{Dgm}(M)$ be \emph{persistence diagram} of $M$, i.e. the multiset of points in $(\mathbb{R}_+\sqcup\{\infty\})^2$, that represent the endpoints of intervals in the interval module decomposition of $M$. That is $(a,b)\in \text{Dgm}(M)$ if $\cat{k}[A]$ is a direct summand of $M$ and $[a,b]$ is the smallest closed interval that contains $A$. This means for example that  $\cat{k}[a,b)$ and $\cat{k}(a,b]$ both contribute the point $(a,b)$ to $\text{Dgm}(M)$. 

A matching $\gamma$ between persistence diagrams $\text{Dgm}(M)$ and $\text{Dgm}(N)$ is a bijection between subsets of the diagrams, where any unmatched interval is paired with its own projection onto the diagonal $\Delta = \{(x, x)\, |\, x \in \mathbb{R}_+\}$. The \emph{bottleneck distance} between persistence modules $d_B(M,N)$ is then defined as 

\[d_B(M,N)\od \inf_{\gamma}\sup_{(x,y)\in \text{Dgm}(M)}||(x,y)-\gamma(x,y)||_{\infty},\]
where $||\cdot ||_{\infty}$ is the sup norm on $(\mathbb{R}_+\sqcup\{\infty\})^2$, where $\infty-\infty=0$. The following is the famous stability result for persistent homology of Vietoris-Rips filtrations. The statement was proven for the Gromov-Hausdorff distance, but the Gromov-Hausdorff distance is a lower bound on the Hausdorff distance when $X$ and $Y$ are subsets of the same metric space.

\begin{theorem}{\cite[Theorem 5.2]{Chazal2013}}
	\label{theorem:stability_hausdorff}
Let $X,Y\subset Z$ be finite point clouds in a metric space $Z$. Let $\mathcal{H}_n(\cat{VR}_{\bullet}(X))$ and  $\mathcal{H}_n(\cat{VR}_{\bullet}(Y))$ be the $n$-dimensional homology persistence modules of the Vietoris-Rips complex filtrations of $X$ and $Y$, respectively. Then 
	\[d_B(\text{Dgm}(\mathcal{H}_n(\cat{VR}_{\bullet}(X))),\text{Dgm}(\mathcal{H}_n(\cat{VR}_{\bullet}(Y))))\le 2d_{\text{GH}}(X,Y)\le 2d^Z_H(X,Y).\]
\end{theorem}

\subsection{Approximating persistence diagrams in a product metric}
Let $S_N^X$, $S_M^Y$, $S_{NM}^{X\times Y}$ be as above. Let $\cat{VR}^p_{\bullet}(S_N^X\times S_M^Y)$ and $\cat{VR}^p_{\bullet}(S_{NM}^{X\times Y})$ be Vietoris-Rips filtrations of $S_N^X\times S_M^Y$ and $S_{NM}^{X\times Y}$ with respect to the $\ell^p$ metric on $X\times Y$ for $p\in [1,\infty]$, respectively. Then from \cref{theorem:stability_hausdorff,theorem:hausdorff_distance_grid_and_sample} the following is immediate.

\begin{corollary}
\label{corollary:grid_and_sample_pers_hom} We have that
\[d_B(\text{Dgm}(\mathcal{H}_n(\cat{VR}^p_{\bullet}(S_N^X\times S_M^Y))),\text{Dgm}(\mathcal{H}_n(\cat{VR}^p_{\bullet}(S_{NM}^{X\times Y}))))\to 0\text{ almost surely as }N,M\to\infty,\]
where $\mathcal{H}_n(-)$ are the $n$-dimensional persistent homology modules.
\end{corollary}

Let $\cat{VR}_{\bullet}(S_N^X)$ and $\cat{VR}_{\bullet}(S_M^Y)$ be the Vietoris-Rips filtrations of $S_N^X$ and $S_M^Y$ with respect to the metrics in $X$ and $Y$, respectively. These are effectively filtrations on the simplex  $\Delta^{N}$ and the simplex $\Delta^{M}$, respectively. Our K\"unneth short exact sequence for filtered CW complexes (\cref{theorem:kunneth_filtered_cw}) allows us to efficiently compute the persistent homology of the $\cpnorm$ filtration on $\Delta^N\times \Delta^M$, where this is a product of CW complexes. We denote this filtered CW complex by
\[(\cat{VR}(S_N^X)\times \cat{VR}(S_M^Y))_{\cpnorm(\bullet,\bullet)}.\] 

 However, $(\cat{VR}(S_N^X)\times \cat{VR}(S_M^Y))_{\cpnorm(\bullet,\bullet)}$  and $\cat{VR}^p_{\bullet}(S_N^X\times S_M^Y)$ are two different filtered CW complexes. Note that, as in \cref{fig:kunneth_square}, $(\cat{VR}(S_N^X)\times \cat{VR}(S_M^Y))_{\cpnorm(\bullet,\bullet)}$ will typically be a filtered cubical complex, whereas $\cat{VR}_{\bullet}^p(S_N^X\times S_M^Y)$ is always a filtered simplicial complex. Even if we triangulate $\Delta^N\times \Delta^M$ by instead taking the product in the category of simplicial complexes, it is not clear how to induce a new filtration on the triangulation that agrees with the filtration $\cat{VR}_{\bullet}^p(S_N^X\times S_M^Y)$. This was first observed in \cite{Carlsson2020} for the case of the $\ell^1$ metric on the product space. 
 
  The case $p=\infty$ is a notable exception, as the product complex $\cat{VR}_a(S_N^X)\times \cat{VR}_a (S_M^Y)$ (in the category of simplicial complexes) and the simplicial complex $\cat{VR}_{a}^{\infty}(S_N^X\times S_M^Y)$ are homotopy equivalent, for all $a\in \mathbb{R}_+$, see \cite[Proposition 10.2]{Adamaszek2017}. Therefore, as was also observed in \cite{gakhar2019} we have:
  
\begin{proposition}{\cite[Theorem 4.4]{gakhar2019}}
	\label{proposition:vr_product_sup_metric}
For $n\ge 0$, 
\[\mathcal{H}_n((\cat{VR}(S_N^X)\times \cat{VR}(S_M^Y))_{\ell^{\infty}_c(\bullet,\bullet)})=\mathcal{H}_n(\cat{VR}^{\infty}_{\bullet}(S_N^X\times S_M^Y)),\] 
where $\mathcal{H}_n(-)$ is the $n$-dimensional persistent homology.
\end{proposition}
 
Furthermore, note that the one-skeletons of $(\cat{VR}(S_N^X)\times \cat{VR}(S_M^Y))_{\cpnorm(\bullet,\bullet)}$ and $\cat{VR}_{\bullet}^p(S_N^X\times S_M^Y)$ will have the same filtrations with respect to any $\ell^p$ metric. Thus we also have the following:

\begin{lemma}
\label{lemma:pers_hom_degree_0_grid_and_product}
We have \[\mathcal{H}_0((\cat{VR}(S_N^X)\times \cat{VR}(S_M^Y))_{\cpnorm(\bullet,\bullet)})=\mathcal{H}_0(\cat{VR}^p_{\bullet}(S_N^X\times S_M^Y)),\] where $\mathcal{H}_0(-)$ is the $0$-dimensional persistent homology.
\end{lemma}

Therefore, even with the result in \cref{corollary:grid_and_sample_pers_hom}, there is still a gap in applying our K\"unneth Theorems for approximating persistent homology in the $\ell^p$ metrics, at least in dimensions other than $0$ or $p=\infty$.

Using the same arguments as in the proof of \cite[Theorem 4.9]{Carlsson2020} which was about the $\ell^1$ metric, we can prove the following:

\begin{proposition}
 	\label{prop:bottleneck_distance}
	Let $r_1=\text{diam}(S_N^X)$ and $r_2=\text{diam}(S_M^Y)$ be the diameters of the samples $S_N^X$ and $S_M^Y$, respectively. Then
	\[d_B(\text{Dgm}(\mathcal{H}_n(\cat{VR}^p_{\bullet}(S_N^X\times S_M^Y))),\text{Dgm}(\mathcal{H}_n(\cat{VR}(S_N^X)\times \cat{VR}(S_M^Y))_{\cpnorm(\bullet,\bullet)})))\le \min(r_1,r_2),\]
	for all $n\ge 0$.
\end{proposition}

At this point in time we have no other theoretical guarantees or bounds on the bottleneck distances between the resulting persistence diagrams of the filtered complexes $(\cat{VR}(S_N^X)\times \cat{VR}(S_M^Y))_{\cpnorm(\bullet,\bullet)}$ and $\cat{VR}_{\bullet}^p(S_N^X\times S_M^Y)$. We do present some experiments on synthetic examples in the following subsection in order to form some conjectures.

\subsection{Experimental results}
All experiments were conducted on Kaggle's free cloud computing platform using a CPU session with 30GB RAM. We sampled points from two unit intervals, taking their product gives us a grid sample of the square. We did the same with two circles of radius $1$ to get a grid sample of the torus, see \cref{fig:grid_samples}. The metric on the unit intervals was the standard one and the metric on the circle was the geodesic one. The point samples inherited these metrics. The resulting metric on their products were the $\ell^p$ metrics, where $p\in \{1,2,5\}$. We are not interested in the case  $p=\infty$ for this analysis due to \cref{proposition:vr_product_sup_metric}.

\begin{figure}[H]
	\includegraphics[scale=0.25]{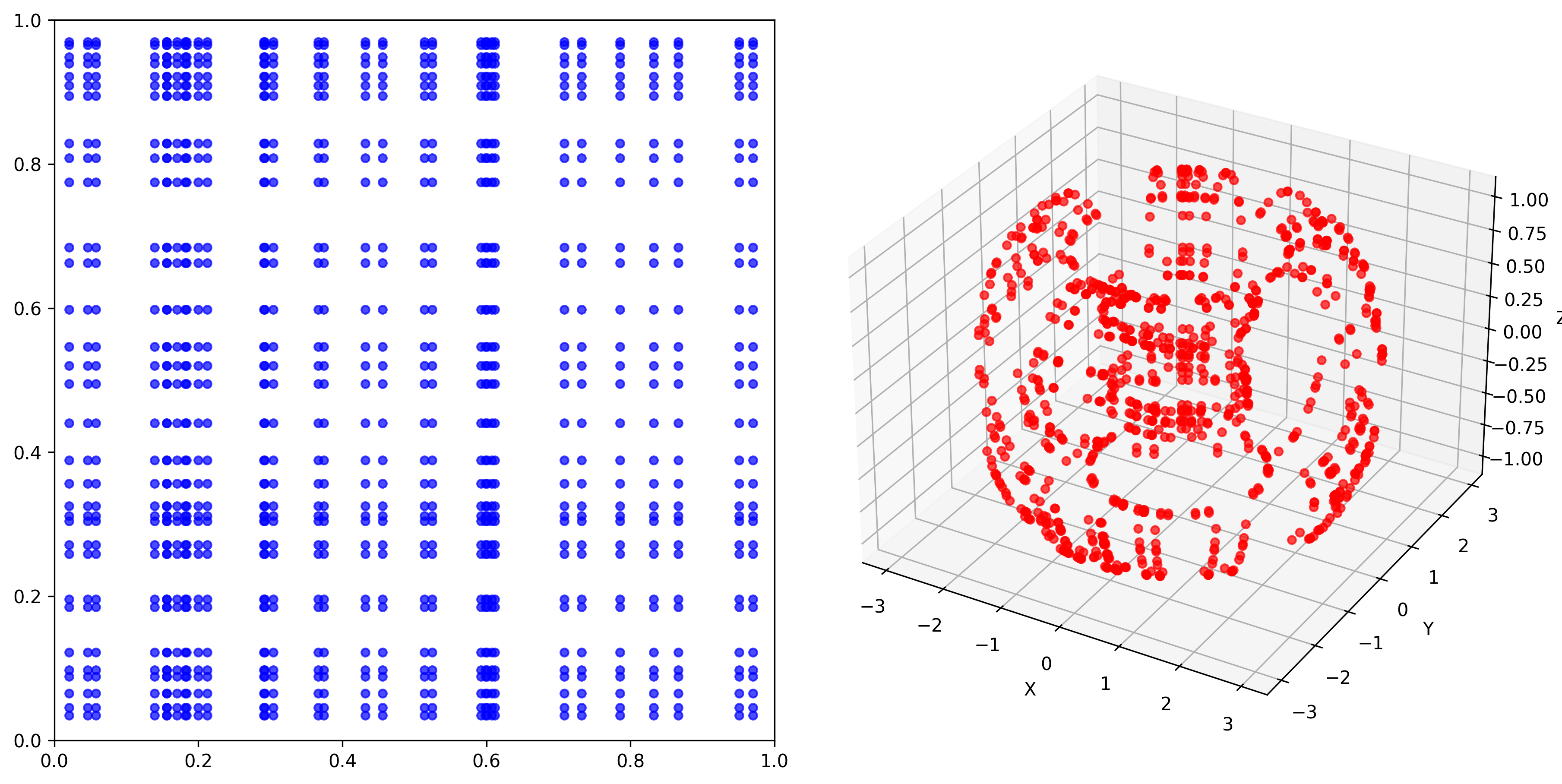}
	\caption{Grid samples of a unit square (blue) and torus (red). $32$ points were sampled from uniform measures on each unit interval and each circle.}
	\label{fig:grid_samples}
\end{figure}

We sampled a range of points from $5$ to $32$, from each unit interval and circle. The resulting product samples thus contained a range of points from $25$ to $1024$.
We computed the Vietoris-Rips persistent homology of the samples of the unit intervals and circles and combined them using \cref{theorem:kunneth_filtered_cw}. We then computed persistent homology of the resulting Vietoris-Rips filtrations of the corresponding products, with respect to the $\ell^p$ metrics. All computations were done using the Ripser algorithm \cite{Bauer2021}. We did not compute persistent homology in degree $0$ for this analysis due to \cref{lemma:pers_hom_degree_0_grid_and_product}. 

In the case of the square sample, we did not compute persistent homology in degree $2$ or higher. This is because the persistent homology of the samples in the two unit intervals in degree higher than $0$ is trivial, so the K\"unneth short exact sequence will only give terms in at most degree $1$.

In the case of the torus sample, we did not compute persistent homology in degree $3$ or higher. This is because of the computational complexity of computing persistent homology in higher degrees. Our Kaggle session in fact ran out of memory when trying to compute persistent homology for more than $1024$ points in degree $2$.

 The results for the square are in \cref{fig:square_comparison}. Across the full range of sample size in our experiments, the bottleneck distance between the two different filtrations has very small values, for persistent homology diagrams in dimension $1$. It also seems to be decreasing as sample size is increasing.

\begin{figure}[H]
	\includegraphics[scale=0.35]{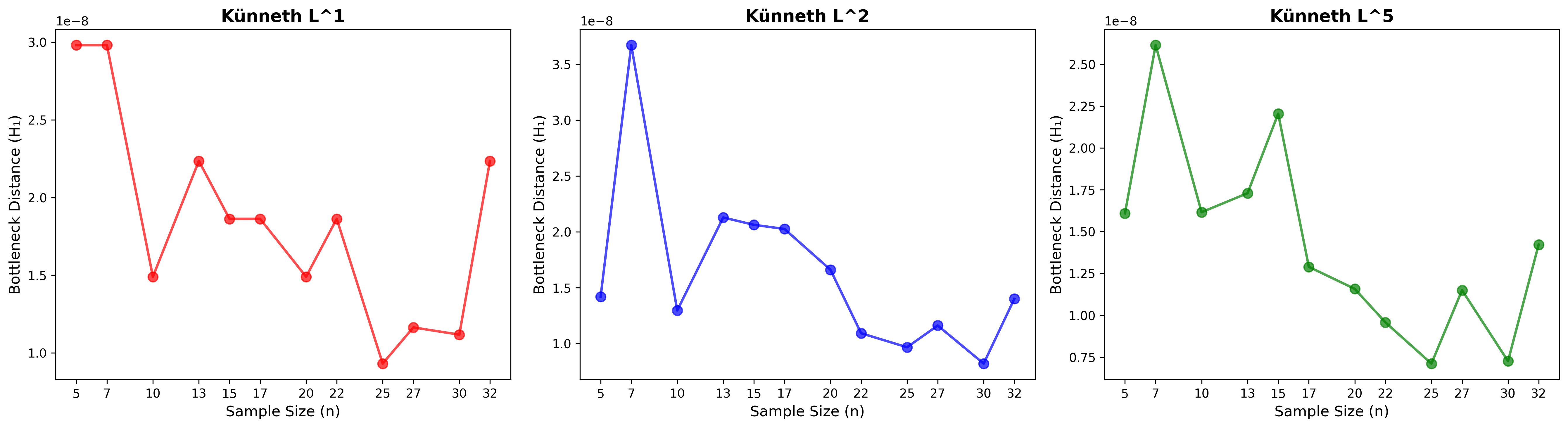}
	\caption{Points were sampled uniformly from two unit intervals. Their product is a grid sample of the unit square. Persistent homology in dimension $1$ of the grid sample was computed with respect to the $\ell^p$ metric. We also computed the persistent homology in dimension $1$ of the resulting filtered cubical complex. The bottleneck distances between the resulting persistence diagrams is shown on the $y$-axis. The $x$-axis shows the sample size from each interval ($n$). Plots for $p=1$ in red, $p=2$ in blue and $p=5$ in green are illustrated.}
	\label{fig:square_comparison}
\end{figure}

 The results for the torus are in \cref{fig:torus_comparison}. Across the full range of sample size in our experiments, the bottleneck distance between the two different filtrations has very small values, for persistent homology diagrams in degree $1$. It also seems to be decreasing as sample size is increasing. However, for persistent homology in dimension $2$, the bottleneck distance is several orders of magnitude larger. For both dimensions, the bottleneck distance goes down with sample size, for all values of $p$ tested. We thus formulate our first conjecture.
 
\noindent \textbf{Conjecture 1.} As $N,M\to \infty$,  \[d_B(\text{Dgm}(\mathcal{H}_n(\cat{VR}^p_{\bullet}(S_N^X\times S_M^Y))),\text{Dgm}(\mathcal{H}_n((\cat{VR}(S_N^X)\times \cat{VR}(S_M^Y))_{\cpnorm(\bullet,\bullet)})))\to 0,\]
for all $n\ge 0$. 

It also appears that for higher values of $p$, the bottleneck distance goes to $0$ faster as the sample size increases. This makes sense, as the parameter $p$ alternates between $1$ and $\infty$ and we know by \cref{proposition:vr_product_sup_metric} that the distance is exactly $0$ for $p=\infty$.

\noindent \textbf{Conjecture 2.} As $N,M\to \infty$, 
 \[d_B(\text{Dgm}(\mathcal{H}_n(\cat{VR}^q_{\bullet}(S_N^X\times S_M^Y))),\text{Dgm}(\mathcal{H}_n((\cat{VR}(S_N^X)\times \cat{VR}(S_M^Y))_{\ell^q_c(\bullet,\bullet)})))\to 0,\]
at a faster rate than $d_B(\text{Dgm}(\mathcal{H}_n(\cat{VR}^p_{\bullet}(S_N^X\times S_M^Y))),\text{Dgm}(\mathcal{H}_n((\cat{VR}(S_N^X)\times \cat{VR}(S_M^Y))_{\cpnorm(\bullet,\bullet)})))$ whenever $p<q$, for all $n\ge 0$. 

Finally, as noted above the bottleneck distances for persistent homology in dimension $2$ are orders of magnitude larger than in the dimension $1$ case. This appears true across all sample sizes and values of $p$. Therefore, we state our final conjecture.

\noindent \textbf{Conjecture 3.} As $N,M\to \infty$,  
\[d_B(\text{Dgm}(\mathcal{H}_n(\cat{VR}^p_{\bullet}(S_N^X\times S_M^Y))),\text{Dgm}(\mathcal{H}_n((\cat{VR}(S_N^X)\times \cat{VR}(S_M^Y))_{\cpnorm(\bullet,\bullet)})))\to 0,\]
at a faster rate than $d_B(\text{Dgm}(\mathcal{H}_m(\cat{VR}^q_{\bullet}(S_N^X\times S_M^Y))),\text{Dgm}(\mathcal{H}_m((\cat{VR}(S_N^X)\times \cat{VR}(S_M^Y))_{\ell^q_c(\bullet,\bullet)})))$ whenever $n\le m$, if the persistent homology in dimensions $n$ and $m$ is nontrivial.

\begin{figure}[H]
	\includegraphics[scale=0.35]{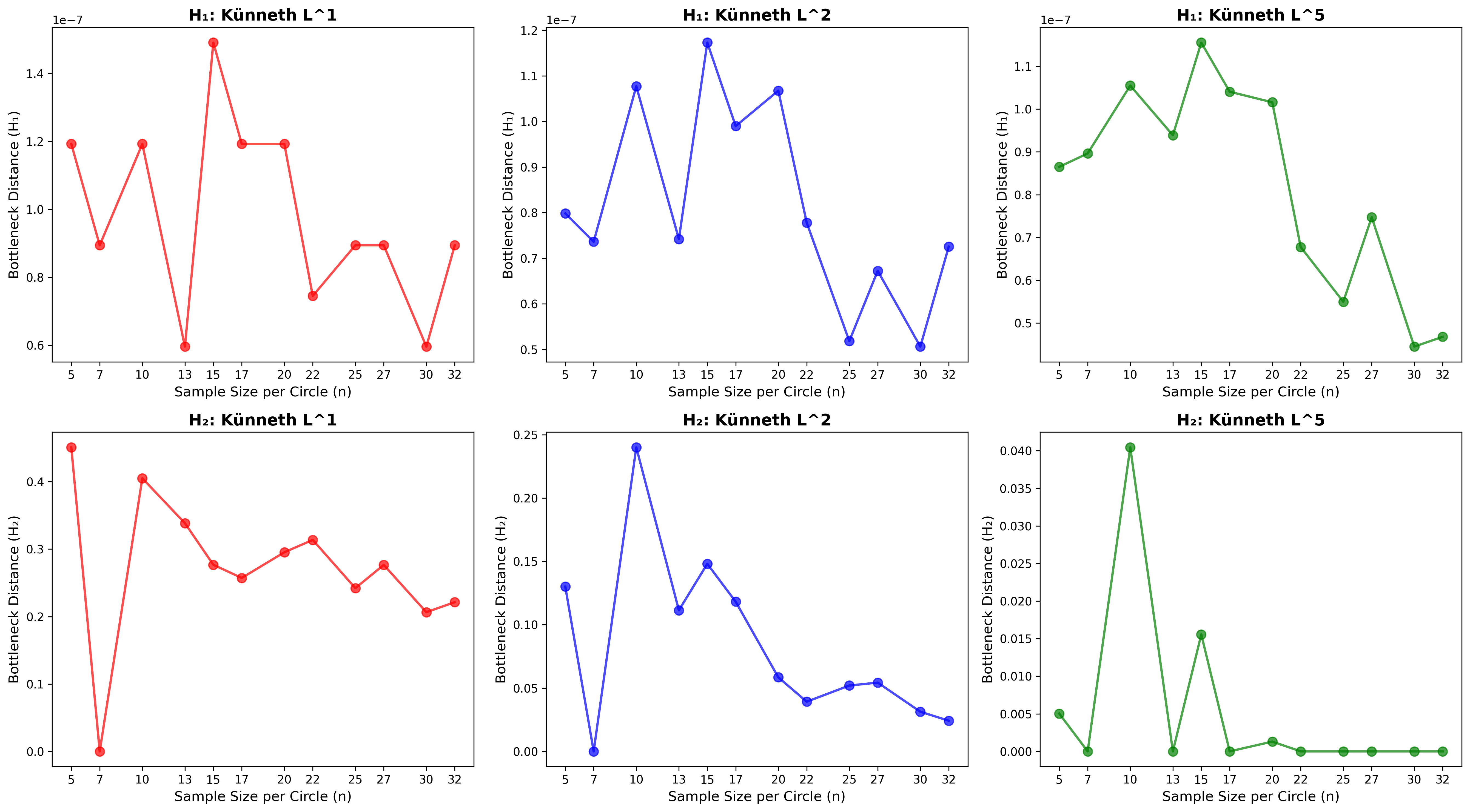}
	\caption{Points were sampled uniformly from two circles. Their product is a grid sample of the torus. Persistent homology in dimension $1$ and $2$ of the grid sample was computed with respect to the $\ell^p$ metric. We also computed the persistent homology in dimensions $1$ and $2$ of the resulting filtered cubical complex. The bottleneck distances between the resulting persistence modules is shown on the $y$-axis. The top row shows the bottleneck distances for persistence diagrams in dimension $1$. The bottom row shows the bottleneck distances for persistence diagrams in dimension $2$. The $x$-axis shows the sample size from each circle ($n$). Plots for $p=1$ in red, $p=2$ in blue and $p=5$ in green are illustrated.}
	\label{fig:torus_comparison}
\end{figure}

\section{Discussion}

Here we discuss the implications of this manuscript.

\begin{itemize}[left=0pt]
	\item We developed K\"unneth short exact sequences for persistence modules in 
    	significantly greater generality than previously existed in the literature, 
    	expanding the homological algebra of persistence modules. Our examples and 
    	theorems lay groundwork for future research. Notably, our results accelerate 
    	computations of persistent homology of product spaces, and our Universal 
    	Coefficient Theorems appear applicable to computing the Borel-Moore homology 
    	of open complexes, with potential implications for discrete stratified Morse 
    	theory.
	 
	 \item Most examples in this paper are restricted to $P = \mathbb{R}_+^n$, 
    	since the $\cpnorm$ functions are order-preserving on this domain. However, 
    	our results hold in greater generality for arbitrary posets $P$ and 
    	order-preserving maps $\varphi$. In particular, for any order-preserving 
    	$\varphi \colon \mathbb{R}^n \times \mathbb{R}^n \to \mathbb{R}^n$ or 
    	$\varphi \colon \mathbb{Z}^n \times \mathbb{Z}^n \to \mathbb{Z}^n$, which 
    	are standard in persistence theory. Furthermore, our classification of 
    	$\otimes_\varphi$-acyclic, projective, and injective persistence modules 
    	can be extended to other posets using similar arguments. We hope these results 
    	facilitate new computational approaches to topological data analysis.
	 
	\item The standard persistent homology pipeline begins with a filtered 
    	simplicial complex, most commonly a Vietoris-Rips complex built from a 
    	finite metric space. However, the persistence algorithm requires only a 
    	symmetric dissimilarity matrix, i.e. a weighted graph and does 
    	not depend on a metric space structure. Our framework applies to any 
    	product filtration $F : X \times Y \to \mathbb{R}_+$ arising from an 
    	order-preserving map $\varphi: \mathbb{R}_+^2 \to \mathbb{R}_+$ 
    	composed with two individual filtration functions, which may be chosen 
    	freely.
	
	\item We attempted to give some bounds on bottleneck distances 
	\[d_B(\text{Dgm}(\mathcal{H}_n(\cat{VR}^p_{\bullet}(S_N^X\times S_M^Y))),\text{Dgm}(\mathcal{H}_n((\cat{VR}(S_N^X)\times \cat{VR}(S_M^Y))_{\cpnorm(\bullet,\bullet)}))).\]
	The need for this arises because $\cat{VR}(S_N^X)\times \cat{VR}(S_M^Y))_{\cpnorm(\bullet,\bullet)}$ and $\cat{VR}^p_{\bullet}(S_N^X\times S_M^Y)$ are different filtered complexes. However, one could simply always use the filtered complex $\cat{VR}(S_N^X)\times \cat{VR}(S_M^Y))_{\cpnorm(\bullet,\bullet)}$ and never consider $\cat{VR}^p_{\bullet}(S_N^X\times S_M^Y)$ when analyzing $S_N^X\times S_M^Y$. This approach is more computationally efficient and is a valid construction based on homological algebra.
	
	\item We illustrated our framework extensively for $\varphi = \cpnorm$. 
   	 The most direct application is approximating persistent homology on a 
    	product space equipped with the $\ell^p$ metric, such as the torus. 
    	Strikingly, our pipeline can reduce the computation of persistent homology 
   	 of a point cloud with millions of points to the same computational 
   	 complexity as that of a point cloud with roughly one thousand points.
    	
	\item Our theory applies to products of arbitrary filtered CW complexes and is
	not restricted to Vietoris-Rips filtrations. \v{C}ech \cite{Edelsbrunner2009}, alpha
	\cite{Edelsbrunner1983}, Dowker \cite{Dowker1952}, and flood
	 \cite{graf2026floodcomplexlargescalepersistent} complexes are all compatible 
	 alternatives.
\end{itemize}

\appendix
\label{section:appendix}

\section{Morphisms of interval modules}
\label{section:appendix_A}

We restate \cref{lemma:nonzero_map_of_interval_modules} for the sake of convenience.

\begin{lemma}
    Suppose that $A$ and $B$ are intervals in $P$ such that $A\cap B$ is connected.
    Assume there is a nonzero map $f:\cat{k}[A]\to \cat{k}[B]$ of interval modules. Then (up to isomorphism) $f_a=1$ if $a\in A\cap B$ and $f_a=0$ otherwise.
\end{lemma}

\begin{proof}
    Assume $f\neq 0$. Then there is a $b\in A\cap B$ such that $f_b$ is nonzero. Without loss of generality, we may assume that $f_b=1$. 
    Since $A\cap B$ is connected, for any $a\in A\cap B$ there is a sequence $a=p_0\le q_1\ge p_1\le q_2\ge\cdots \ge p_n\le q_n=b$ for some $n\in \mathbb{N}$. Suppose that there is a $c\in A\cap B$ such that $a\le c\ge b$. For the general case the argument will follow by induction. 
    Then from the commutative diagram
    \begin{figure}[H]
        \centering
        \begin{tikzcd}
            \cat{k}[A]_{a}\arrow[r,"1"]\arrow[d,"f_{a}"]& \cat{k}[A]_c\arrow[d,"f_c"]& \cat{k}[A]_b\arrow[l,"1"]\arrow[d,"f_b=1"]\\
            \cat{k}[B]_{a}\arrow[r,"1"]& \cat{k}[B]_c &\cat{k}[B]_b\arrow[l,"1"]
        \end{tikzcd}
    \end{figure}
    we conclude that $f_c=1$ and thus $f_a=1$. Thus, for any $a\in A\cap B$ we have $f_a=1$. If $a\not\in A\cap B$, then either $\cat{k}[A]_a=0$ or $\cat{k}[B]_a=0$, which implies that $f_a=0$.
\end{proof}

We restate \cref{lemma:morphisms_between_intervals} for the sake of convenience.

\begin{lemma}
Let $A,B$ be intervals of $P$ such that $A\cap B$ is connected. Then 
    \[\Hom_{\cat{k^{P}}}(\cat{k}[A],\cat{k}[B]) \approx
  \begin{cases}
    \cat{k}, &\text{if } A \text{ and } B \text{ overlap nicely}\\
    0, &\textnormal{otherwise}
  \end{cases}.\]
\end{lemma}

\begin{proof}
    We prove that there is a nonzero map $f:\cat{k}[A]\to \cat{k}[B]$ if and only if $A$ and $B$ overlap nicely. 

    Suppose that $f\neq 0$. Then for all $a\in A\cap B$ without loss of generality, we can assume that $1=f_a:\cat{k}[A]_a\to\cat{k}[B]_a$.
    Assume that $A$ and $B$ do not overlap nicely.
    Then there exists a $b\in B\setminus A$ such that there is an $a\in A\cap  B$ with $a\le b$, or there exists an $a\in A\setminus B$ such that there is a $b\in A\cap  B$ with $a\le b$.

    Let $b\in B\setminus A$ be such that there is an $a\in A\cap  B$ with $a\le b$. From the commutative diagram on the left below
      \begin{figure}[H]
        \centering
        \begin{tikzcd}
            \cat{k}[A]_{a}\arrow[r,"0"]\arrow[d,"f_{a}"]& \cat{k}[A]_b=0\arrow[d,"f_b=0"]\\
            \cat{k}[B]_{a}\arrow[r,"1"]& \cat{k}[B]_b
        \end{tikzcd}
        \hspace{1cm}
        \begin{tikzcd}
            \cat{k}[A]_{a}\arrow[r,"1"]\arrow[d,"f_{a}=0"]& \cat{k}[A]_b=0\arrow[d,"f_b"]\\
            \cat{k}[B]_{a}\arrow[r,"0"]& \cat{k}[B]_b
        \end{tikzcd}
    \end{figure}
    we conclude that it must be the case that $f_a=0$. But $a\in A\cap B$ which contradicts our previous observation. Let $a\in A\setminus B$ be such that there is a $b\in A\cap  B$ with $a\le b$. From the commutative diagram on the right above we conclude that it must be the case that $f_b=0$. But $b\in A\cap B$ which contradicts our previous observation. 

    Therefore if $A$ and $B$ do not overlap nicely we get a contradiction, in either case. Thus it must be the case that $A$ and $B$ overlap nicely.

    Now suppose that $A$ and $B$ overlap nicely. Define $f:\cat{k}[A\to \cat{k}[B]$ by $f_a=1$ if $a\in A\cap B$ and $f_a=0$ otherwise. We argue that $f$ is a natural transformation. Let $a\le b$ in $P$. There are several possibilities to consider.
    \begin{enumerate}[left=0pt]
        \item If $a,b\in A\cap B$, then $f_a=1,f_b=1,\cat{k}[A]_{a\le b}=1$ and $\cat{k}[B]_{a\le b}=1$ and thus $\cat{k}[B]_{a\le b}\circ f_a=f_b\circ \cat{k}[A]_{a\le b}$.
        \item If $a,b\not\in A\cup B$, then $f_a=0,f_b=0,\cat{k}[A]_{a\le b}=0$ and $\cat{k}[B]_{a\le b}=0$ and thus $\cat{k}[B]_{a\le b}\circ f_a=f_b\circ \cat{k}[A]_{a\le b}$.
        \item If $a\in A\setminus B$ and $b\in B\setminus A$, then $f_a=0,f_b=0,\cat{k}[A]_{a\le b}=0$ and $\cat{k}[B]_{a\le b}=0$ and thus $\cat{k}[B]_{a\le b}\circ f_a=f_b\circ \cat{k}[A]_{a\le b}$.
        \item If $b\in A\setminus B$ and $a\in B\setminus A$, then $f_a=0,f_b=0,\cat{k}[A]_{a\le b}=0$ and $\cat{k}[B]_{a\le b}=0$ and thus $\cat{k}[B]_{a\le b}\circ f_a=f_b\circ \cat{k}[A]_{a\le b}$.
        \item If $a\in A\setminus B$ and $b\in A\cap B$, then $a\le b$ cannot happen since $A$ and $B$ overlap nicely. So this case does not happen.
        \item If $b\in B\setminus A$ and $a\in A\cap B$, then $a\le b$ cannot happen since $A$ and $B$ overlap nicely. So this case does not happen.
        \item If $b\in A\setminus B$ and $a\in A\cap B$, then $f_a=1,f_b=0,\cat{k}[A]_{a\le b}=1$ and $\cat{k}[B]_{a\le b}=0$ and thus $\cat{k}[B]_{a\le b}\circ f_a=f_b\circ \cat{k}[A]_{a\le b}$.
        \item If $a\in A\setminus B$ and $b\in A\cap B$, then $f_a=0,f_b=1,\cat{k}[A]_{a\le b}=0$ and $\cat{k}[B]_{a\le b}=1$ and thus $\cat{k}[B]_{a\le b}\circ f_a=f_b\circ \cat{k}[A]_{a\le b}$.
    \end{enumerate}
    Thus, in all possible cases we have $\cat{k}[B]_{a\le b}\circ f_a=f_b\circ \cat{k}[A]_{a\le b}$. Therefore $f$ is a nonzero natural transformation.
\end{proof}

\section{K\"unneth theorem proofs}
\label{section:appendix_kunneth}

Here we prove the existence of a split K\"unneth short exact sequences for persistence modules. The proof consists of standard arguments involving spectral sequences of filtrations of double complexes in abelian categories. We assume the reader is familiar with spectral sequences; otherwise \cite[Chapter 10]{rotman2009introduction} is good introductory material.

In what follows, $(P,\le)$ is an arbitrary poset and $\varphi:P\times P\to P$ is an order-preserving morphism. All persistence modules are objects in $\cat{k}^P$. A positive complex of persistence modules is a chain complex of persistence modules $(K,d^K)$ such that $K_n=0$ for $n<0$ and the differentials $d^K_n:K_n\to K_{n-1}$ go down in degree and satisfy $d^K_{n-1}d^K_n=0$. A negative complex of persistence modules is a chain complex of persistence modules $(K,d^K)$ such that $K_n=0$ for $n>0$ and the differentials $d^K_{-n}:K_{-n}\to K_{-n-1}$ go down in degree and satisfy $d^K_{n-1}d^K_n=0$. Equivalently, this also defines a positive cochain complex of persistence modules by the mapping $n\mapsto -n$.

\begin{theorem}[K\"unneth short exact sequence for persistence modules, I]
Let $(K,d^K)$ and $(L,d^L)$ be two positive complexes of persistence module. Suppose that $K$ and its subcomplex of boundaries have all terms projective persistence modules. Then for all $n\ge 0$ there is a natural short exact sequence 
\[0\to \bigoplus_{i+j=n}H_i(K)\otimes_{\varphi} H_j(L)\to H_n(K\otimes_{\varphi} L)\to\bigoplus_{i+j=n-1} \cat{Tor}^{\varphi}_1(H_{i}(K),H_j(L))\to 0.\]
If in addition the complex $L$ and its subcomplex of boundaries have all terms projective, then
the short exact sequence splits, although not naturally.
\end{theorem}

\begin{proof}
For a given $n$, let $B_{n-1}=\im d^K_n$ denote the boundary and $Z_n=\ker d_n$ denote the kernel of the differential in $K$ at the $n$-th term, respectively. From the natural short exact sequence
\[0\to Z_n\to K_n\to B_{n-1}\to 0,\]
and the assumption of $B_{n-1}$ being projective we have a splitting $K_n\approx Z_n\oplus B_{n-1}$. Since $K_n$ is projective, its direct summand $Z_n$ is also projective. Define the ``elementary'' complex of persistence modules $S_p$ concentrated in degrees $p$ and $p+1$ by
\[ S_p: \dots \to 0 \to B_p \xrightarrow{i_p} Z_p \to 0 \to \dots \]
where $i_p: B_p \hookrightarrow Z_p$ is the inclusion. Note $H_p(S_p) \approx H_p(K)$ and $H_{k}(S_p) = 0$ for $k \neq p$, by construction. Since $K \approx \bigoplus_p S_p$, and both $\otimes_{\varphi}$ and $H_n$ commute with direct sums we have that
\[ H_n(K \otimes_{\varphi} L) \approx \bigoplus_p H_n(S_p \otimes_{\varphi} L). \]

For a fixed $p$, $S_p \otimes_{\varphi} L$ is a double complex with only two rows ($p$ and $p+1$). We consider the following spectral sequence filtered by rows. The $E^0$ page is 
\[E^0_{i,j}=(S_p)_i\otimes_{\varphi}L_j.\]
Taking homology with respect to $d^L$ (the horizontal differential), the $E^1$ page is
\[E^1_{p,q}=Z_p\otimes_{\varphi} H_q(L),\quad  E^1_{p+1,q}=B_p\otimes_{\varphi} H_q(L).\]
All other $E^1$ page entries are $0$. This is because $B_p$ and $Z_p$ are projective, hence left $\otimes_{\varphi}$-acyclic. Therefore $B_p\otimes_{\varphi}-$ and $Z_p\otimes_{\varphi}-$ are exact functors and thus commute with taking kernels and images and thus they commute with homology as well.

The $d^1:E^1\to E^2$ differential is induced by $i_p\otimes_{\varphi}\id$. For a fixed $q$, there is a vertical complex of persistence modules
\[0\to B_p\otimes_{\varphi}H_q(L)\xrightarrow{i_p\otimes_{\varphi}\id}Z_p\otimes_{\varphi}H_q(L)\to 0.\]
Since $0\to B_p\to Z_p\to H_p(K)\to 0$ is a projective resolution of $H_p(K)$, the homology of this vertical complex gives the $E^2$ page:
\[E^2_{p,q}\approx H_p(K)\otimes_{\varphi} H_q(L),\quad E^2_{p+1,q}=\cat{Tor}_1^{\varphi} (H_p(K),H_q(L)).\]
Because there are only two columns ($p$ and $p+1$), all higher differentials $d^r$  
($r\ge 2$) vanish. The sequence thus collapses at $E^2$ and thus we have short exact sequence of persistence modules
\[0\to H_p(K)\otimes_{\varphi} H_{n-p}(L)\to H_n(S_p\otimes_{\varphi}L)\to \cat{Tor}_1^{\varphi}(H_p(K),H_{n-p-1}(L))\to 0.\]
Direct summing over all $p$ and noting that $H_n(K\otimes_{\varphi} L)\approx \bigoplus_p H_n(S_p\otimes_{\varphi} L)$ gives the short exact sequence
\[0\to \bigoplus_{i+j=n}H_i(K)\otimes_{\varphi} H_j(L)\to H_n(K\otimes_{\varphi} L)\to\bigoplus_{i+j=n-1} \cat{Tor}^{\varphi}_1(H_{i}(K),H_j(L))\to 0.\]

Now assume that the complex $L$ and its subcomplex of boundaries also has all terms projective. Let $B_{n=1}^K=\im d^K_n, B_{n=1}^L=\im d^L_n, Z_n^K=\ker d^K_n, Z_n^L=\ker d^L_n$.
	Then we have the natural short exact split sequences
	\[0\to Z^K_n\xrightarrow{i_n^K} K_n\to B_{n-1}^K\to 0,\]
	\[0\to Z^L_n\xrightarrow{i_n^L} L_n\to B_{n-1}^L\to 0.\]
	Since the sequences are split, we have retractions $r^K_n:K_n\to Z^K_n$ and $r^L_n:L_n\to Z^L_n$ (which satisfy $r_n^K\circ i_n^K=\id_{Z_n^K}$ and $r_n^L\circ i_n^L=\id_{Z_n^L}$). Composing these retractions with the quotients $Z_n^K \to H_nK)$ and $Z_n^L\to H_n(L)$,  gives us chain maps $K\to H(K)$ and $L\to H(L)$ where $H(K)$ and $H(L)$ are the complexes of homologies with zero differentials. We therefore get a chain map 
	\[K\otimes_{\varphi} L\to H(K)\otimes_{\varphi}H(L).\] 
	
	The induced map on homology for this last chain map
is the desired splitting map since the chain complex $H(K)\otimes_{\varphi}H(L)$ equals its own
homology, the boundary maps being trivial.
	 Therefore the induced map on homology is
		\[H_n(K\otimes_{\varphi} L)\to \bigoplus_{i+j=n}H_i(K)\otimes_{\varphi}H_j(L).\] 
 		is a splitting of the short exact sequence 
		\[0\to \bigoplus_{i+j=n}H_i(K)\otimes_{\varphi} H_j(L)\to H_n(K\otimes_{\varphi} L)\to\bigoplus_{i+j=n-1} \cat{Tor}^{\varphi}_1(H_{i}(K),H_j(L))\to 0.\qedhere\]
\end{proof}

Restricting to the special case when $L$ is a single persistence module thought of as a complex concentrated in degree $0$ we have the following.

\begin{theorem}[Universal Coefficient Theorem for persistence modules, I]
Let $A$ be a persistence module. Let $(K,d)$ be a chain complex that has all terms projective and whose subcomplex of boundaries $B$ also has all terms projective. For all $n\ge 0$, there is a natural short exact sequence
\[0\to H_n(K)\otimes_{\varphi} A\to H_n(K\otimes_{\varphi} A)\to \cat{Tor}^{\varphi}_1(H_{n-1}(K),A)\to 0.\]
Furthermore the short exact sequence splits, although not naturally.
\end{theorem}

\begin{proof}
Since $K$ and its subcomplex of boundaries have all terms projective, there is a short exact sequence
\[0\to H_n(K)\otimes_{\varphi} A\to H_n(K\otimes_{\varphi} A)\to \cat{Tor}^{\varphi}_1(H_{n-1}(K),A)\to 0.\]

Recall the splitting $K_n\approx Z_n\oplus B_{n-1}$. Applying the functor $-\otimes_{\varphi}A$ we see that $Z_n\otimes_{\varphi}A$ is a direct summand of $K_n\otimes_{\varphi}A$. Furthermore, $Z_n\otimes_{\varphi}A$ is also a direct summand of 
\[\ker(d_n\otimes_{\varphi}\id_A:K_n\otimes_{\varphi}A\to K_{n-1}\otimes_{\varphi}A).\]
Taking quotients of both $Z_n\otimes_{\varphi}A$ and $\ker(d_n\otimes_{\varphi}\id_A)$ by the common image of $d_{n+1}\otimes_{\varphi}\id_A$ we see that $H_n(K)\otimes_{\varphi}A$ is a direct summand of $H_n(K\otimes_{\varphi}A)$. Therefore the K\"unneth short exact sequence for $A$ splits.
\end{proof}

\section{Metric measure spaces}
\label{section:appendix_B}

\begin{lemma}[Borel-Cantelli lemma]
\label{lemma:borel_cantelli}
Let $E_1, E_2, \dots $ be a sequence of events in some probability space. If the sum of the probabilities of the events $\{E_n\}$ is finite, i.e. $\sum_{n=1}^{\infty} P(E_n)<\infty$, then the probability that infinitely many of them occur is $0$, that is

\[P(\limsup_{n\to\infty}E_n)=P(\bigcap_{n=1}^{\infty}\bigcup_{k=n}^{\infty}E_k)=0.\]
\end{lemma}

We restate \cref{prop:iid_sample} for the sake of convenience.

\begin{proposition}
	Let $(X,d)$ be a compact metric space and $\mu$ a Borel probability measure. Let $x_1, x_2, \dots, x_N$ be i.i.d. samples from $\mu$. Then
	\[\lim_{N\to \infty}\sup_{x\in \text{supp}(\mu)}\min_{i=1,\dots, N}d(x,x_i)=0 \text{ almost surely}.\]
\end{proposition}

\begin{proof}
	Recall that the support of $\mu$, denoted $\text{supp}(\mu)$, is defined as:
	
	\[\text{supp}(\mu)=\{x\in X\, |\, \mu(B(x,r))>0 \text{ for all }r>0\},\]
	where $B(x,r)$ denotes the open ball of radius $r$ centered at $x$. The support, $\text{supp}(\mu)$, is closed  in $X$ and since $X$ is compact, $\text{supp}(\mu)$ is also compact.  Since $\text{supp}(\mu)$ is compact, for any $\epsilon > 0$, there exists a finite set of points $\{y_1, y_2, \dots, y_k\} \subset \text{supp}(\mu)$ such that:
	
	\[\text{supp}(\mu)\subset \bigcup_{j=1}^kB\left(y_j,\dfrac{\epsilon}{2}\right).\]
	
	In particular, for each $1\le j\le k$,  we have $\mu\left(B\left(y_j, \dfrac{\epsilon}{2}\right)\right) > 0$, by definition of $\text{supp}(\mu)$. For each $1\le j\le  k$, define the event
	
	\[A_{j,N}=\left\{\text{at least one of } x_1,x_2,\dots ,x_N \text{ is in } B\left(y_j,\dfrac{\epsilon}{2}\right)\right\}.\]
	
	Since $\mu\left(B\left(y_j, \dfrac{\epsilon}{2}\right)\right) > 0$, the probability that any single sample falls in $B\left(y_j, \dfrac{\epsilon}{2}\right)$ is $p_j = \mu\left(B\left(y_j, \dfrac{\epsilon}{2}\right)\right) > 0$. The probability that none of the first $N$ samples falls in $B\left(y_j, \dfrac{\epsilon}{2}\right)$ is $P(\text{not} A_{j,N})=(1-p_j)^N$, because the samples were independent. Since $p_j > 0$, we have
	
	\[\lim_{N\to \infty}P(\text{not }A_{j,N})=\lim_{N\to \infty}(1-p_j)^N=0.\]
	
	Therefore $P(A_{j,N})\to 1$ as $N\to \infty$. Define the event
	
	\[A_N^{\epsilon}=\bigcap_{j=1}^kA_{j,N}=\left\{\text{for all }  j, \text{ at least one of } x_1,x_2,\dots, x_N\text{ is in }B\left (y_j,\dfrac{\epsilon}{2}\right)\right\}.\]
	
	Since the events $A_{j,N}$ are not independent, we cannot directly multiply probabilities. However, we still have
	
	\[P(\text{not}A_N^{\epsilon})=P\left(\bigcup_{j=1}^k\text{not }A_{j,N}\right)\le\sum_{j=1}^kP(\text{not }A_{j,N})=\sum_{j=1}^k(1-p_j)^N.\]
	
	Since each $p_j > 0$ and $k$ is finite we have
	
	\[\lim_{N\to \infty}P(\text{not} A_N^{\epsilon})\le\lim_{N\to \infty} \sum_{j=1}^k(1-p_j)^N=\sum_{j=1}^k\lim_{N\to \infty}(1-p_j)^N=0.\]
	
	Therefore $P(A_N^{\epsilon})\to 1$ as $N\to \infty$. Suppose the event $A_N^{\epsilon}$ occurs. Then for each $1\le j \le k$, there exists an $1\le i_j\le N$ such that $x_{i_j}\in B\left(y_j,\dfrac{\epsilon}{2}\right)$. For any $x\in \text{supp}(\mu)$, there exists some $j$ such that $x\in B\left(y_j,\dfrac{\epsilon}{2}\right)$ and hence $d(x,y_j)<\dfrac{\epsilon}{2}$. Since $x_{i_j}\in B\left(y_j,\dfrac{\epsilon}{2}\right)$ it follows that $d(x_{i_j},y_j)<\dfrac{\epsilon}{2}$. By the triangle inequality we have
	\[d(x,x_{i_j})\le d(x,y_j)+d(y_j,x_{i_j})<\dfrac{\epsilon}{2}+\dfrac{\epsilon}{2}=\epsilon.\]
	
	Therefore we have
	\[\min_{i=1,\dots, N}d(x,x_i)\le d(x,x_{i_j})<\epsilon.\]
	Since this holds for all $x\in \text{supp}(\mu)$ we have
	
	\[\sup_{x\in \text{supp}(\mu)}\min_{i=1,\dots, N}d(x,x_i)\le \epsilon.\]
	
	For any $\epsilon>0$ we showed above that $P(\text{not}A_N^{\epsilon})\le \sum_{j=1}^k(1-p_j)^N$. In particular, since $0<1-p_j<1$ for each $j$ we have that 
	\[\sum_{N=1}^{\infty}P (\text{not}A_N^{\epsilon})\le \sum_{N=1}^{\infty}\sum_{j=1}^k(1-p_j)^N=\sum_{j=1}^k\sum_{N=1}^{\infty}(1-p_j)^N.\]
	Note that the inner sum on the right hand side is a convergent geometric series and therefore we have
	\[\sum_{N=1}^{\infty}(1-p_j)^N=\dfrac{1-p_j}{1-(1-p_j)}=\dfrac{1-p_j}{p_j}<\infty.\]
	Since we have finitely many such series the entire sum is finite. Therefore $\sum_{N=1}^{\infty}P (\text{not}A_N^{\epsilon})\le \infty$. Hence, by the Borel-Cantelli lemma (\cref{lemma:borel_cantelli}) we have that 
	\[P(\limsup_{N\to\infty} (\text{not}A_N^{\epsilon}))=0.\]
	This is equivalent to
	\[P(\liminf_{N\to\infty} A_N^{\epsilon})=1,\]
	
	which means that with probability $1$, $A_N^{\epsilon}$ occurs for sufficiently large $N$. Note that whenever $A_N^{\epsilon}$ occurs we have that 
	\[\sup_{x\in \text{supp}(\mu)}\min_{i=1,\dots, N}d(x,x_i)\le \epsilon\].
	
	Therefore
	\[P\left(\sup_{x\in \text{supp}(\mu)}\min_{i=1,\dots, N}d(x,x_i)\le \epsilon \text{ for all sufficiently large N}\right)=1.\]
	This holds for every $\epsilon > 0$.
	
	Now, for each positive integer $m$, with probability $1$,  there exists a (random) $N_m$ such that for all $N \ge N_m$ we have
	
	\[\sup_{x\in \text{supp}(\mu)}\min_{i=1,\dots, N}d(x,x_i)\le \dfrac{1}{m}.\]
	
	Since we have countably many events (one for each $m$), and each occurs with probability $1$, their intersection also occurs with probability $1$. Therefore	
	\[\lim_{N\to \infty}\sup_{x\in \text{supp}(\mu)}\min_{i=1,\dots, N}d(x,x_i)=0 \text{ almost surely}.\qedhere\]
\end{proof}

We restate \cref{prop:grid_sample} for the sake of convenience.

\begin{proposition}
	Let $(X,d_X)$ and $(Y,d_Y)$ be compact metric spaces with Borel probability measures $\mu_X$ and $\mu_Y$, respectively, both having full support. Let $\{x_i\}_{i=1}^N$ and $\{y_i\}_{i=1}^M$ be i.i.d. samples from $\mu_X$ and $\mu_Y$, respectively. Let $d$ be the metric on $X\times Y$ defined by $d((x,y),(x',y'))=(d_X(x,x')^p+d_Y(y,y')^p)^{1/p}$ for $p\ge 1$. Then
	\[\lim_{N,M\to \infty}\sup_{(x,y)\in X\times Y}\min_{1\le i\le N,1\le j\le M}d((x,y),(x_i,y_j))=0 \text{ almost surely}.\]
\end{proposition}

\begin{proof}
	Let $(x,y)\in X\times Y$ and let $i^*=\argmin\limits_{1\le i\le N} d_X(x,x_i)$ and $j^*=\argmin\limits_{1\le j\le M} d_Y(y,y_j)$. Then
	
	\[\min_{i,j}d((x,y),(x_i,y_j))\le d((x,y),(x_{i^*},y_{j^*})).\]
	This implies that
	\[\min_{i,j}d((x,y),(x_i,y_j))\le\left(\left(\min_{1\le i\le N}d_X(x,x_i)\right)^p+\left(\min_{1\le j\le M}d_Y(y,y_j)\right)^p\right)^{\frac{1}{p}}.\]
	
	Taking the supremum of the previous inequality we have
	\[\sup_{(x,y)\in X\times Y}\min_{i,j}d((x,y),(x_i,y_j))\le\left(\left(\sup_{x\in X}\min_{1\le i\le N}d_X(x,x_i)\right)^p+\left(\sup_{y\in Y}\min_{1\le j\le M}d_Y(y,y_j)\right)^p\right)^{\frac{1}{p}}.\]
	
		By \cref{prop:iid_sample} we have that 
	\begin{itemize}
		\item $A_N=\lim\limits_{N\to \infty}\sup\limits_{x\in X}\min\limits_{1\le i\le N}d_X(x,x_i)\to 0$ almost surely.
		\item $B_M=\lim\limits_{M\to \infty}\sup\limits_{y\in Y}\min\limits_{1\le j\le M}d_Y(y,y_j)\to0$ almost surely.
	\end{itemize}
	Therefore for any $\epsilon>0$, we have that 
	\begin{itemize}
		\item $P\left(A_N<\dfrac{\epsilon}{2^{1/p}}\right)\to 1$ as $N\to \infty$. 
		\item$P\left(B_M<\dfrac{\epsilon}{2^{1/p}}\right)\to 1$ as $M\to \infty$.
	\end{itemize} 
	
	Since $\{x_i\}$ and $\{y_j\}$ are independent samples, the random variables $A_N$ and $B_M$ are also independent and therefore
	\[P\left(A_N<\dfrac{\epsilon}{2^{1/p}} \text{ and }B_M<\dfrac{\epsilon}{2^{1/p}}\right)=P\left(A_N<\dfrac{\epsilon}{2^{1/p}}\right )P\left(B_M<\dfrac{\epsilon}{2^{1/p}}\right)\to 1.\]
	When both events occur we have that 
	\[Z_{NM}=(A_N^p+B_M^p)^{1/p}<\left(\left(\dfrac{\epsilon}{2^{1/p}}\right)^p+\left(\dfrac{\epsilon}{2^{1/p}}\right)^p\right)^{1/p}=\left(\dfrac{\epsilon^p}{2}+\dfrac{\epsilon^p}{2}\right)^{1/p}=\epsilon.\]
	Therefore $P(Z_{NM}<\epsilon)\to 1$ as $N,M\to \infty$. Since this holds for any $\epsilon$ we have that $Z_{NM}\to 0$ in probability. Therefore
	
	\[\lim_{N,M\to \infty}\sup_{(x,y)\in X\times Y}\min_{1\le i\le N,1\le j\le M}d((x,y),(x_i,y_j))=0 \text{ almost surely}.\qedhere\]
\end{proof}

\printbibliography

@article{Bubenik2021,
  title = {Homological Algebra for Persistence Modules},
  volume = {21},
  ISSN = {1615-3383},
  DOI = {10.1007/s10208-020-09482-9},
  number = {5},
  journal = {Foundations of Computational Mathematics},
  publisher = {Springer Science and Business Media LLC},
  author = {Bubenik,  Peter and Milićević,  Nikola},
  year = {2021},
  month = jan,
  pages = {1233–1278}
}

@article{Bauer2021,
  title = {Ripser: efficient computation of Vietoris–Rips persistence barcodes},
  volume = {5},
  ISSN = {2367-1734},
  DOI = {10.1007/s41468-021-00071-5},
  number = {3},
  journal = {Journal of Applied and Computational Topology},
  publisher = {Springer Science and Business Media LLC},
  author = {Bauer,  Ulrich},
  year = {2021},
  month = june,
  pages = {391–423}
}

@misc{graf2026floodcomplexlargescalepersistent,
      title={The Flood Complex: Large-Scale Persistent Homology on Millions of Points}, 
      author={Florian Graf and Paolo Pellizzoni and Martin Uray and Stefan Huber and Roland Kwitt},
      year={2026},
      eprint={2509.22432},
      archivePrefix={arXiv},
      primaryClass={cs.LG},
}

@article{Dowker1952,
  title = {Homology Groups of Relations},
  volume = {56},
  ISSN = {0003-486X},
  DOI = {10.2307/1969768},
  number = {1},
  journal = {The Annals of Mathematics},
  publisher = {JSTOR},
  author = {Dowker,  C. H.},
  year = {1952},
  month = july,
  pages = {84}
}

@article{Edelsbrunner1983,
  title = {On the shape of a set of points in the plane},
  volume = {29},
  ISSN = {1557-9654},
  DOI = {10.1109/tit.1983.1056714},
  number = {4},
  journal = {IEEE Transactions on Information Theory},
  publisher = {Institute of Electrical and Electronics Engineers (IEEE)},
  author = {Edelsbrunner,  H. and Kirkpatrick,  D. and Seidel,  R.},
  year = {1983},
  month = july,
  pages = {551–559}
}

@book{Edelsbrunner2009,
  title = {Computational Topology},
  ISBN = {9781470412081},
  DOI = {10.1090/mbk/069},
  publisher = {American Mathematical Society},
  author = {Edelsbrunner,  Herbert and Harer,  John},
  year = {2009},
  month = dec 
}

@book{Burago2001,
  title = {A Course in Metric Geometry},
  ISBN = {9781470417949},
  ISSN = {1065-7339},
  DOI = {10.1090/gsm/033},
  journal = {Graduate Studies in Mathematics},
  publisher = {American Mathematical Society},
  author = {Burago,  Dmitri and Burago,  Yuri and Ivanov,  Sergei},
  year = {2001},
  month = jun 
}

@article{Knudson2022,
  title = {Discrete Stratified Morse Theory: Algorithms and A User’s Guide},
  volume = {67},
  ISSN = {1432-0444},
  DOI = {10.1007/s00454-022-00372-1},
  number = {4},
  journal = {Discrete \& Computational Geometry},
  publisher = {Springer Science and Business Media LLC},
  author = {Knudson,  Kevin and Wang,  Bei},
  year = {2022},
  month = mar,
  pages = {1023–1052}
}

@misc{knudson2026discretemorsetheoryopen,
      title={Discrete Morse theory for open complexes}, 
      author={Kevin P. Knudson and Nicholas A. Scoville},
      year={2026},
      eprint={2402.12116},
      archivePrefix={arXiv},
      primaryClass={math.AT},
}

@article{Chazal2013,
  title = {Persistence stability for geometric complexes},
  volume = {173},
  ISSN = {1572-9168},
  DOI = {10.1007/s10711-013-9937-z},
  number = {1},
  journal = {Geometriae Dedicata},
  publisher = {Springer Science and Business Media LLC},
  author = {Chazal,  Frédéric and de Silva,  Vin and Oudot,  Steve},
  year = {2013},
  month = dec,
  pages = {193–214}
}

@misc{gakhar2019,
      title={K\"unneth Formulae in Persistent Homology}, 
      author={Hitesh Gakhar and Jose A. Perea},
      year={2019},
      eprint={1910.05656},
      archivePrefix={arXiv},
      primaryClass={math.AT},
}

@inbook{Tamarkin2018,
  title = {Microlocal Condition for Non-displaceability},
  ISBN = {9783030015886},
  ISSN = {2194-1017},
  DOI = {10.1007/978-3-030-01588-6_3},
  booktitle = {Algebraic and Analytic Microlocal Analysis},
  publisher = {Springer International Publishing},
  author = {Tamarkin,  Dmitry},
  year = {2018},
  pages = {99–223}
}

@article{CohenSteiner2006,
  title = {Stability of Persistence Diagrams},
  volume = {37},
  ISSN = {1432-0444},
  DOI = {10.1007/s00454-006-1276-5},
  number = {1},
  journal = {Discrete \& Computational Geometry},
  publisher = {Springer Science and Business Media LLC},
  author = {Cohen-Steiner,  David and Edelsbrunner,  Herbert and Harer,  John},
  year = {2006},
  month = dec,
  pages = {103–120}
}

@book{Bredon1997,
  title = {Sheaf Theory},
  ISBN = {9781461206477},
  ISSN = {0072-5285},
  DOI = {10.1007/978-1-4612-0647-7},
  journal = {Graduate Texts in Mathematics},
  publisher = {Springer New York},
  author = {Bredon,  Glen E.},
  year = {1997}
}

@article{Borel1960,
  title = {Homology theory for locally compact spaces.},
  volume = {7},
  ISSN = {0026-2285},
  DOI = {10.1307/mmj/1028998385},
  number = {2},
  journal = {Michigan Mathematical Journal},
  publisher = {Michigan Mathematical Journal},
  author = {Borel,  A. and Moore,  J. C.},
  year = {1960},
  month = jan 
}

@inbook{Guillermou2014,
  title = {Microlocal Theory of Sheaves and Tamarkin’s Non Displaceability Theorem},
  ISBN = {9783319065144},
  ISSN = {1862-9113},
  DOI = {10.1007/978-3-319-06514-4_3},
  booktitle = {Homological Mirror Symmetry and Tropical Geometry},
  publisher = {Springer International Publishing},
  author = {Guillermou,  Stéphane and Schapira,  Pierre},
  year = {2014},
  pages = {43–85}
}

@article{Crawley-Boevey2015,
  title     = "Decomposition of pointwise finite-dimensional persistence
               modules",
  author    = "Crawley-Boevey, William",
  journal   = "J. Algebr. Appl.",
  publisher = "World Scientific Pub Co Pte Lt",
  volume    =  14,
  number    =  05,
  pages     = "1550066",
  month     =  jun,
  year      =  2015,
  language  = "en"
}

@article{Geist2023,
  title     = "Global dimension of real-exponent polynomial rings",
  author    = "Geist, Nathan and Miller, Ezra",
  journal   = "Algebra Number Theory",
  publisher = "Mathematical Sciences Publishers",
  volume    =  17,
  number    =  10,
  pages     = "1779--1788",
  month     =  sep,
  year      =  2023,
  language  = "en"
}

@article{Polterovich2017,
  title     = "Persistence modules with operators in Morse and floer theory",
  author    = "Polterovich, Leonid and Shelukhin, Egor and Stojisavljevi{\'c},
               Vuka{\v s}in",
  journal   = "Mosc. Math. J.",
  publisher = "National Research University, Higher School of Economics (HSE)",
  volume    =  17,
  number    =  4,
  pages     = "757--786",
  year      =  2017
}

@article{Carlsson2020,
  title     = "Persistent homology of the sum metric",
  author    = "Carlsson, Gunnar and Filippenko, Benjamin",
  journal   = "J. Pure Appl. Algebra",
  publisher = "Elsevier BV",
  volume    =  224,
  number    =  5,
  pages     = "106244",
  month     =  may,
  year      =  2020,
  copyright = "http://www.elsevier.com/open-access/userlicense/1.0/",
  language  = "en"
}

@book{curry2014sheaves,
    AUTHOR = {Curry, Justin Michael},
     TITLE = {Sheaves, cosheaves and applications},
      NOTE = {Thesis (Ph.D.)--University of Pennsylvania},
 PUBLISHER = {ProQuest LLC, Ann Arbor, MI},
      YEAR = {2014},
     PAGES = {317},
      ISBN = {978-1303-96615-6},
   MRCLASS = {Thesis},
  MRNUMBER = {3259939},
}

@article{Kashiwara2018,
  title     = "Persistent homology and microlocal sheaf theory",
  author    = "Kashiwara, Masaki and Schapira, Pierre",
  journal   = "J. Appl. Comput. Topol.",
  publisher = "Springer Science and Business Media LLC",
  volume    =  2,
  number    = "1-2",
  pages     = "83--113",
  month     =  oct,
  year      =  2018,
  language  = "en"
}

@book{Kashiwara1990,
  title = {Sheaves on Manifolds},
  ISBN = {9783662026618},
  ISSN = {0072-7830},
  DOI = {10.1007/978-3-662-02661-8},
  journal = {Grundlehren der mathematischen Wissenschaften},
  publisher = {Springer Berlin Heidelberg},
  author = {Kashiwara,  Masaki and Schapira,  Pierre},
  year = {1990}
}

@article{Miller2023,
  title     = "Stratifications of real vector spaces from constructible sheaves
               with conical microsupport",
  author    = "Miller, Ezra",
  journal   = "J. Appl. Comput. Topol.",
  publisher = "Springer Science and Business Media LLC",
  volume    =  7,
  number    =  3,
  pages     = "473--489",
  month     =  sep,
  year      =  2023,
  copyright = "https://www.springernature.com/gp/researchers/text-and-data-mining",
  language  = "en"
}

@article{Miller2025,
  title     = "Homological algebra of modules over posets",
  author    = "Miller, Ezra",
  journal   = "SIAM J. Appl. Algebr. Geom.",
  publisher = "Society for Industrial \& Applied Mathematics (SIAM)",
  volume    =  9,
  number    =  3,
  pages     = "483--524",
  month     =  sep,
  year      =  2025,
  language  = "en"
}

@book{rotman2009introduction,
  title={An introduction to homological algebra},
  author={Rotman, Joseph J and Rotman, Joseph J},
  volume={2},
  year={2009},
  publisher={Springer}
}

@article{Berkouk2021,
  title     = "Ephemeral persistence modules and distance comparison",
  author    = "Berkouk, Nicolas and Petit, Fran{\c c}ois",
  journal   = "Algebr. Geom. Topol.",
  publisher = "Mathematical Sciences Publishers",
  volume    =  21,
  number    =  1,
  pages     = "247--277",
  month     =  feb,
  year      =  2021,
  language  = "en"
}

@article{Berkouk2022,
  title     = "A derived isometry theorem for sheaves",
  author    = "Berkouk, Nicolas and Ginot, Gr{\'e}gory",
  journal   = "Adv. Math. (N. Y.)",
  publisher = "Elsevier BV",
  volume    =  394,
  number    =  108033,
  pages     = "108033",
  month     =  jan,
  year      =  2022,
  language  = "en"
}

@article{Fernandes2025,
  title     = "Computation of gamma-linear projected barcodes for
               multiparameter persistence",
  author    = "Fernandes, Alex and Oudot, Steve and Petit, Fran{\c c}ois",
  journal   = "J. Appl. Comput. Topol.",
  publisher = "Springer Science and Business Media LLC",
  volume    =  9,
  number    =  2,
  month     =  jun,
  year      =  2025,
  copyright = "https://www.springernature.com/gp/researchers/text-and-data-mining",
  language  = "en"
}

@misc{berkouk2024,
      title={Projected distances for multi-parameter persistence modules}, 
      author={Nicolas Berkouk and Francois Petit},
      year={2024},
      eprint={2206.08818},
      archivePrefix={arXiv},
      primaryClass={math.AT},
}

@article{Carlsson2009,
  title     = "The theory of multidimensional persistence",
  author    = "Carlsson, Gunnar and Zomorodian, Afra",
  journal   = "Discrete Comput. Geom.",
  publisher = "Springer Science and Business Media LLC",
  volume    =  42,
  number    =  1,
  pages     = "71--93",
  month     =  jul,
  year      =  2009,
  language  = "en"
}

@article{Lesnick2015,
  title     = "The theory of the interleaving distance on multidimensional
               persistence modules",
  author    = "Lesnick, Michael",
  journal   = "Found. Comut. Math.",
  publisher = "Springer Science and Business Media LLC",
  volume    =  15,
  number    =  3,
  pages     = "613--650",
  month     =  jun,
  year      =  2015,
  language  = "en"
}

@article{Zomorodian2005,
  title     = "Computing persistent homology",
  author    = "Zomorodian, Afra and Carlsson, Gunnar",
  journal   = "Discrete Comput. Geom.",
  publisher = "Springer Science and Business Media LLC",
  volume    =  33,
  number    =  2,
  pages     = "249--274",
  month     =  feb,
  year      =  2005,
  language  = "en"
}

@article{Bubenik2015,
  title     = "Metrics for generalized persistence modules",
  author    = "Bubenik, Peter and de Silva, Vin and Scott, Jonathan",
  journal   = "Found. Comut. Math.",
  publisher = "Springer Science and Business Media LLC",
  volume    =  15,
  number    =  6,
  pages     = "1501--1531",
  month     =  dec,
  year      =  2015,
  language  = "en"
}

@ARTICLE{Bubenik2014,
  title     = "Categorification of persistent homology",
  author    = "Bubenik, Peter and Scott, Jonathan A",
  journal   = "Discrete Comput. Geom.",
  publisher = "Springer Science and Business Media LLC",
  volume    =  51,
  number    =  3,
  pages     = "600--627",
  month     =  apr,
  year      =  2014,
  language  = "en"
}

@article{Takeuchi2021,
  title     = "The persistent homology of a sampled map: from a viewpoint of
               quiver representations",
  author    = "Takeuchi, Hiroshi",
  journal   = "J. Appl. Comput. Topol.",
  publisher = "Springer Science and Business Media LLC",
  volume    =  5,
  number    =  2,
  pages     = "179--213",
  month     =  jun,
  year      =  2021,
  copyright = "https://creativecommons.org/licenses/by/4.0",
  language  = "en"
}

@article{Escolar2016,
  title     = "Persistence modules on commutative ladders of finite type",
  author    = "Escolar, Emerson G and Hiraoka, Yasuaki",
  journal   = "Discrete Comput. Geom.",
  publisher = "Springer Science and Business Media LLC",
  volume    =  55,
  number    =  1,
  pages     = "100--157",
  month     =  jan,
  year      =  2016,
  language  = "en"
}

@article{Carlsson2010,
  title     = "Zigzag persistence",
  author    = "Carlsson, Gunnar and de Silva, Vin",
  journal   = "Found. Comut. Math.",
  publisher = "Springer Nature",
  volume    =  10,
  number    =  4,
  pages     = "367--405",
  month     =  aug,
  year      =  2010,
  language  = "en"
}

@book{Oudot2016,
  title     = "Persistence theory: From quiver representations to data analysis",
  author    = "Oudot, Steve Y",
  publisher = "American Mathematical Society",
  series    = "Mathematical Surveys and Monographs",
  month     =  sep,
  year      =  2016,
  address   = "Providence, RI"
}

@article{Harrington2019,
  title     = "Stratifying multiparameter persistent homology",
  author    = "Harrington, Heather A and Otter, Nina and Schenck, Hal and
               Tillmann, Ulrike",
  journal   = "SIAM J. Appl. Algebr. Geom.",
  publisher = "Society for Industrial \& Applied Mathematics (SIAM)",
  volume    =  3,
  number    =  3,
  pages     = "439--471",
  month     =  jan,
  year      =  2019,
  language  = "en"
}

@misc{harsu2024,
      title={Ephemeral Modules and Scott Sheaves on a Continuous Poset}, 
      author={Manu Harsu and Eero Hyry},
      year={2024},
      eprint={2411.16235},
      archivePrefix={arXiv},
      primaryClass={math.AT},
      url={https://arxiv.org/abs/2411.16235}, 
}

@misc{angel2025,
      title={2-Categorical Foundations for Multiparameter Persistence}, 
      author={Mauricio Angel},
      year={2025},
      eprint={2508.02914},
      archivePrefix={arXiv},
      primaryClass={math.AT},
}

@article{Torras-Casas2023,
  title     = "Distributing persistent homology via spectral sequences",
  author    = "Torras-Casas, {\'A}lvaro",
  journal   = "Discrete Comput. Geom.",
  publisher = "Springer Science and Business Media LLC",
  volume    =  70,
  number    =  3,
  pages     = "580--619",
  month     =  oct,
  year      =  2023,
  copyright = "https://creativecommons.org/licenses/by/4.0",
  language  = "en"
}

@misc{goldfarb2019,
      title={Singular persistent homology with geometrically parallelizable computation}, 
      author={Boris Goldfarb},
      year={2019},
      eprint={1607.01257},
      archivePrefix={arXiv},
      primaryClass={cs.CG},
}

@book{Chazal2016,
  title     = "The structure and stability of persistence modules",
  author    = "Chazal, Frederic and de Silva, Vin and Glisse, Marc and Oudot,
               Steve Y",
  publisher = "Springer International Publishing",
  series    = "SpringerBriefs in mathematics",
  edition   =  1,
  month     =  oct,
  year      =  2016,
  address   = "Cham, Switzerland",
  language  = "en"
}

@article{Botnan2024,
  title     = "On the bottleneck stability of rank decompositions of
               multi-parameter persistence modules",
  author    = "Botnan, Magnus Bakke and Oppermann, Steffen and Oudot, Steve and
               Scoccola, Luis",
  journal   = "Adv. Math. (N. Y.)",
  publisher = "Elsevier BV",
  volume    =  451,
  number    =  109780,
  pages     = "109780",
  month     =  aug,
  year      =  2024,
  copyright = "http://creativecommons.org/licenses/by/4.0/",
  language  = "en"
}

@article{Dey2024,
  title     = "Computing generalized rank invariant for 2-parameter persistence
               modules via zigzag persistence and its applications",
  author    = "Dey, Tamal K and Kim, Woojin and M{\'e}moli, Facundo",
  journal   = "Discrete Comput. Geom.",
  publisher = "Springer Science and Business Media LLC",
  volume    =  71,
  number    =  1,
  pages     = "67--94",
  month     =  jan,
  year      =  2024,
  copyright = "https://creativecommons.org/licenses/by/4.0",
  language  = "en"
}

@article{Blanchette2022,
  title     = "Homological approximations in persistence theory",
  author    = "Blanchette, Benjamin and Br{\"u}stle, Thomas and Hanson, Eric J",
  journal   = "Canad. J. Math.",
  publisher = "Canadian Mathematical Society",
  pages     = "1--38",
  month     =  dec,
  year      =  2022,
  language  = "en"
}

@article{Lupo2022,
  title     = "Persistence Steenrod modules",
  author    = "Lupo, Umberto and Medina-Mardones, Anibal M and Tauzin,
               Guillaume",
  journal   = "J. Appl. Comput. Topol.",
  publisher = "Springer Science and Business Media LLC",
  volume    =  6,
  number    =  4,
  pages     = "475--502",
  month     =  dec,
  year      =  2022,
  copyright = "https://creativecommons.org/licenses/by/4.0",
  language  = "en"
}

@misc{medinamardones2025,
      title={Persistent cohomology operations and Gromov-Hausdorff estimates}, 
      author={Anibal M. Medina-Mardones and Ling Zhou},
      year={2025},
      eprint={2503.17130},
      archivePrefix={arXiv},
      primaryClass={math.AT},
}

@misc{Contessoto2022,
  title     = "Persistent {Cup-Length}",
  author    = "Contessoto, Marco and M{\'e}moli, Facundo and Stefanou,
               Anastasios and Zhou, Ling",
  publisher = "Schloss Dagstuhl - Leibniz-Zentrum f{\"u}r Informatik",
  month     =  jun,
  year      =  2022
}

@article{Memoli2024,
  title     = "Persistent cup product structures and related invariants",
  author    = "M{\'e}moli, Facundo and Stefanou, Anastasios and Zhou, Ling",
  journal   = "J. Appl. Comput. Topol.",
  publisher = "Springer Science and Business Media LLC",
  volume    =  8,
  number    =  1,
  pages     = "93--148",
  month     =  mar,
  year      =  2024,
  copyright = "https://creativecommons.org/licenses/by/4.0",
  language  = "en"
}

@article{JMLR:v21:19-054,
  author  = {Oliver Vipond},
  title   = {Multiparameter Persistence Landscapes},
  journal = {Journal of Machine Learning Research},
  year    = {2020},
  volume  = {21},
  number  = {61},
  pages   = {1--38},
}

@article{Bubenik2015Landscapes,
author = {Bubenik, Peter},
title = {Statistical topological data analysis using persistence landscapes},
year = {2015},
issue_date = {January 2015},
publisher = {JMLR.org},
volume = {16},
number = {1},
issn = {1532-4435},
journal = {J. Mach. Learn. Res.},
month = jan,
pages = {77–102},
numpages = {26}
}

@article{Boissonnat2019,
  title     = "Computing persistent homology with various coefficient fields in
               a single pass",
  author    = "Boissonnat, Jean-Daniel and Maria, Cl{\'e}ment",
  journal   = "J. Appl. Comput. Topol.",
  publisher = "Springer Science and Business Media LLC",
  volume    =  3,
  number    = "1-2",
  pages     = "59--84",
  month     =  jun,
  year      =  2019,
  language  = "en"
}

@article{Ren2021-ox,
  title     = "Computational tools in weighted persistent homology",
  author    = "Ren, Shiquan and Wu, Chengyuan and Wu, Jie",
  journal   = "Chin. Ann. Math. Ser. B",
  publisher = "Springer Science and Business Media LLC",
  volume    =  42,
  number    =  2,
  pages     = "237--258",
  month     =  mar,
  year      =  2021,
  language  = "en"
}

@misc{romero2014,
      title={Defining and computing persistent Z-homology in the general case}, 
      author={Ana Romero and Jónathan Heras and Julio Rubio and Francis Sergeraert},
      year={2014},
      eprint={1403.7086},
      archivePrefix={arXiv},
      primaryClass={cs.CG},
}

@article{Obayashi2023,
  title     = "Field choice problem in persistent homology",
  author    = "Obayashi, Ippei and Yoshiwaki, Michio",
  journal   = "Discrete Comput. Geom.",
  publisher = "Springer Science and Business Media LLC",
  volume    =  70,
  number    =  3,
  pages     = "645--670",
  month     =  oct,
  year      =  2023,
  copyright = "https://creativecommons.org/licenses/by/4.0",
  language  = "en"
}

@article{Adams2025,
  title     = "Vietoris--rips complexes of torus grids",
  author    = "Adams, Henry and Adetowubo, Adenike Yeside and Barriga-Acosta,
               Hector and Feng, Ziqin and Sterling, John",
  journal   = "Mediterr. J. Math.",
  publisher = "Springer Science and Business Media LLC",
  volume    =  22,
  number    =  7,
  month     =  nov,
  year      =  2025,
  copyright = "https://www.springernature.com/gp/researchers/text-and-data-mining",
  language  = "en"
}

@article{Adamaszek2017,
  title = {The Vietoris–Rips complexes of a
circle},
  volume = {290},
  ISSN = {0030-8730},
  DOI = {10.2140/pjm.2017.290.1},
  number = {1},
  journal = {Pacific Journal of Mathematics},
  publisher = {Mathematical Sciences Publishers},
  author = {Adamaszek,  Michał and Adams,  Henry},
  year = {2017},
  month = july,
  pages = {1–40}
}

@article{Grothendieck1957,
  title     = "Sur quelques points d'alg{\`e}bre homologique, {I}",
  author    = "Grothendieck, Alexander",
  journal   = "Tohoku Math. J. (2)",
  publisher = "Mathematical Institute, Tohoku University",
  volume    =  9,
  number    =  2,
  pages     = "119--221",
  month     =  jan,
  year      =  1957
}

@article{Bubenik2018,
  title     = "Topological spaces of persistence modules and their properties",
  author    = "Bubenik, Peter and Vergili, Tane",
  journal   = "J. Appl. Comput. Topol.",
  publisher = "Springer Science and Business Media LLC",
  volume    =  2,
  number    = "3-4",
  pages     = "233--269",
  month     =  dec,
  year      =  2018,
  language  = "en"
}

\end{document}